\documentclass[11pt,reqno]{amsart}

\usepackage{amssymb}
\usepackage{amscd}
\usepackage{amsfonts}
\usepackage{mathrsfs}
\usepackage{setspace}
\usepackage{version}
\usepackage{mathtools}
\usepackage[pdftex,colorlinks,citecolor=blue]{hyperref}

\usepackage{float}
\usepackage{tikz-cd}

\usepackage{tikz}
\usetikzlibrary{calc}
\usetikzlibrary{shadings,intersections}
\usetikzlibrary{decorations.text}
\usepackage{pgfplots,wrapfig}
\usepackage{lipsum}

\usepackage{xcolor}

\theoremstyle{plain}
\numberwithin{equation}{section} 
\newtheorem{theorem}[subsection]{Theorem}
\newtheorem{proposition}[subsection]{Proposition}
\newtheorem{lemma}[subsection]{Lemma}

\newtheorem{corollary}[subsection]{Corollary}
\newtheorem{definition}[subsection]{Definition}

\newtheorem{claim}[subsection]{Claim}
\newtheorem{remark}[subsection]{Remark}

\newtheorem{Example}[subsection]{Example}

\usepackage{soul}

\renewcommand{\leq}{\leqslant}
\renewcommand{\geq}{\geqslant}
\newsavebox{\proofbox}
\savebox{\proofbox}{\begin{picture}(7,7)  \put(0,0){\framebox(7,7){}}\end{picture}}

\newcommand\Z{\mathbb{Z}}
\newcommand\R{\mathbb{R}}

\newcommand\C{\mathbb{C}}
\newcommand\N{\mathbb{N}}

\newcommand\M{\operatorname{M}}
\newcommand\Gr{\operatorname{Gr}}
\newcommand\SL{\operatorname{SL}}
\newcommand\SO{\operatorname{SO}}
\newcommand\PSL{\operatorname{PSL}}

\newcommand\inte{\operatorname{int}}

\newcommand\GL{\operatorname{GL}}

\newcommand\Mat{\operatorname{Mat}}

\newcommand\axe{\operatorname{axis}}
\newcommand\rk{\operatorname{rk}}
\newcommand\ri{\operatorname{ri}}

\newcommand\Supp{\operatorname{Supp}}
\newcommand\Ima{\operatorname{Im}}

\newcommand\tr{\operatorname{tr}}
\newcommand\Hom{\operatorname{Hom}}
\newcommand\Endo{\operatorname{End}}
\newcommand\Lie{\operatorname{Lie}}

\newcommand\diag{\operatorname{diag}}
\renewcommand\P{\mathbb{P}}

\newcommand\G{\mathbb{G}}
\newcommand\Q{\mathbb{Q}}

\renewcommand\b{{\bf b}}
\def\g{\mathfrak{g}}

\def\a{\mathfrak{a}}

\def\b{\mathfrak{b}}

\newcommand\eps{\varepsilon}
\newcommand\id{\operatorname{id}}

\begin{document}

\title{The Joint spectrum}
\author{Emmanuel BREUILLARD}
\author{Cagri SERT}
\address{DPMMS, Wilberforce Road, University of Cambridge, CB30WB, U.K.}
\email{breuillard@maths.cam.ac.uk}
\address{\textsc{Departement Mathematik, ETH Z\"{u}rich, 101 R\"{a}mistrasse, 8092, Z\"{u}rich, Switzerland}}
\email{cagri.sert@math.ethz.ch}


\begin{abstract}
We introduce the notion of \emph{joint spectrum} of a compact set of matrices $S \subset \GL_d(\C)$, which is a multi-dimensional generalization of the joint spectral radius. We begin with a thorough study of its properties (under various assumptions: irreducibility, Zariski-density, domination). Several classical properties of the joint spectral radius are shown to hold in this generalized setting and an analogue of the Lagarias-Wang finiteness conjecture is discussed. Then we relate the joint spectrum to matrix valued random processes and study what points of it can be realized as Lyapunov vectors. We also show how the joint spectrum encodes all word metrics on reductive groups. Several examples are worked out in detail. 
\end{abstract}

\maketitle

\section{Introduction}\label{intro.section}

Let $S$ be a compact subset of matrices $S \subset \M_d(\C)$. The \emph{joint spectral radius} $R(S)$ is the quantity

$$R(S) =\lim_{n \to +\infty} \sup_{g \in S^n} \|g\|^{\frac{1}{n}},$$
where $S^n$ is the $n$-fold product set $S^n:=\{g_1\cdot \ldots \cdot g_n : g_i \in S, i=1,\ldots, n\}$. The limit exists by submultiplicativity and does not depend on the choice of norm. In particular $R(S)$ is a conjugation invariant: $R(gSg^{-1})=R(S)$ for all $g \in \GL_d(\C)$. This notion was introduced by Rota and Strang \cite{rota-strang} in the 60's and has since been extensively studied in a variety of contexts, pure and applied, in particular in the study of wavelets, in control theory, ergodic optimization and beyond. See for example \cite{daubechies-lagarias, lagarias-wang, berger-wang}, \cite{barabanov, gurvitz} and \cite{bousch-mairesse, jenkinson-survey, morris-mather-sets} respectively as well as the many references therein.

The main goal of this paper is to introduce and study a very natural multi-dimensional version of this notion, which we will call the \emph{joint spectrum} $J(S)$ of the subset $S$. For simplicity in this paper we will assume that all matrices are invertible.  We denote by $a_1(g) \geq \ldots \geq a_d(g) > 0$ the singular values of a matrix $g\in \GL_d(\C)$, i.e. the square roots of the eigenvalues of $g^*g$. We set $\kappa(g)=(\log a_1(g),\ldots, \log a_d(g)) \in \R^d$ and call it the Cartan vector of $g$. So $\kappa(S)$ denotes the compact set of all $\kappa(g)$, $g \in S$. 

\begin{theorem}\label{main1} Let $S \subset \GL_d(\C)$ be a compact set and assume that the subgroup $\Gamma$ it generates acts irreducibly on $\C^d$. Then $\frac{1}{n}\kappa(S^n)$ converges in Hausdorff metric to a compact subset of $\R^d$, which we call the \emph{joint spectrum} $J(S)$ of $S$.
\end{theorem}

Clearly $J(S)$ is a conjugation invariant: $J(S)=J(gSg^{-1})$. The joint spectral radius $R(S)$ can easily be read off the joint spectrum. Indeed $\log R(S)=\max \{ x_1 : x=(x_1,\ldots,x_d) \in J(S)\}$. Also the lower joint spectral radius (see e.g. \cite{bochi-morris}) $R_{sub}(S):=\lim_{n\to +\infty} \min_{g \in S^n} \|g\|^{\frac{1}{n}}$ is such that  $\log R_{sub}(S)=\min \{ x_1 : x=(x_1,\ldots,x_d) \in J(S)\}$. In fact the joint spectral radius of $\rho(S)$ for any linear representation of $\rho$ of $\GL_d(\C)$ can also easily be read off the joint spectrum: its logarithm is $\max\{n_1x_1+\ldots+n_dx_d : x=(x_1,\ldots,x_d) \in J(S)\}$, where $(n_1,\ldots,n_d)\in \N^d$ is the highest weight of the representation $\rho$.

The proof of the Hausdorff convergence in this theorem makes use of the notion of proximal transformation, as in the proof of the Tits alternative, see e.g. \cite{abels-proximal, tits}. A proximal transformation is a linear map with a unique eigenvalue of maximal modulus. The main idea is to show, following \cite{AMS}, that there is a finite subset $F$ of the semigroup $\Gamma$ generated by $S$ such that every large matrix  $g$ in $\Gamma$ can be made simultaneously proximal in sufficiently many irreducible representations of $\Gamma$ simply by multiplying it on the left by some element $f$ of $F$. Then the Cartan vector of $g$ becomes very close to the Jordan vector of $gf$, namely the ordered vector of logarithms of moduli of eigenvalues of $gf$. It thus remains close to $\frac{1}{n}\kappa((gf)^n)$ for all $n \geq 1$, and the Hausdorff convergence follows easily.

This argument draws a connection between the joint spectrum and the eigenvalues of elements in $\Gamma$. In the study of discrete subgroups of Lie groups it is often important and challenging to understand what numbers (resp. vectors) can arise as eigenvalues (resp. vectors of eigenvalues) of group elements under a given linear representation. For example, the density properties of the set of traces of elements of a lattice in $\SL_2(\R)$ are related to the arithmeticity or non-arithmeticity of the lattice (see \cite{sarnak, schmutz, geninska}). Similarly, the arithmetic properties of the vectors of eigenvalues of elements of a discrete subgroup of a semisimple algebraic group yield information on the Zariski-closure as in the notion of weak-commensurability studied in \cite{prasad-rapinchuk}. More importantly for our considerations in this article, Benoist defined in \cite{benoist-asymptotic1, benoist-asymptotic2} the notion of limit cone of a semi-group $\Gamma$ of $\GL_d(\C)$, which we will call the \emph{Benoist cone} $\mathcal{BC}(\Gamma)$ in the sequel. It is the closed cone generated by all Jordan vectors of elements of $\Gamma$. Namely if $\lambda(g):=(\log |\lambda_1|,\ldots,\log |\lambda_d|) \in \R^d$ denotes the Jordan vector of $g$,  where the $\lambda_i$'s are the eigenvalues of $g$ ordered so that $|\lambda_1|\geq \ldots \geq |\lambda_d|$, then $\mathcal{BC}(\Gamma)$ is the closure of all positive linear combinations of all $ \lambda(g)$, $g \in \Gamma$.

\begin{theorem}\label{main2} Keep the notation and assumptions of Theorem \ref{main1}. Then  $\frac{1}{n}\lambda(S^n) \subset J(S)$ for all $n \geq 1$,  \begin{equation}\label{closure}\overline{\bigcup_{n \geq 1} \frac{1}{n}\lambda(S^n)}=J(S),\end{equation} and $\mathcal{BC}(\Gamma)$ is the cone spanned by $J(S)$.
\end{theorem}

The sequence $\frac{1}{n}\lambda(S^n)$ may not, in general, converge in Hausdorff metric, because of a periodicity phenomenon that can arise only when $S$ is contained in a single coset of a proper subgroup of finite index in $\langle S \rangle$ containing the commutator subgroup (see Example \ref{ex-non-conv} and the discussion in \S \ref{nonconnected}). Otherwise it converges.

Theorem \ref{main2} can be seen as a multidimensional version of the Berger-Wang theorem \cite{berger-wang, bochi-radius}, which asserts that the joint spectral radius $R(S)$ coincides with $\limsup_{n \to +\infty} \sup_{g \in S^n} |\lambda_1(g)|^{1/n}$. Indeed, in our context, this immediately follows from Theorem \ref{main2} and the aforementioned description of $R(S)$ in terms of $J(S)$. 

The joint spectrum $J(S)$ is not convex in general. In the general case we will show that the joint spectrum is the image of a closed convex set of non-empty interior (i.e. a \emph{convex body}) in some $\R^k$, $k \le d$, under a certain piecewise affine folding map $\phi: \R^k \to \R^d$. See Figure \ref{fig.folding}.


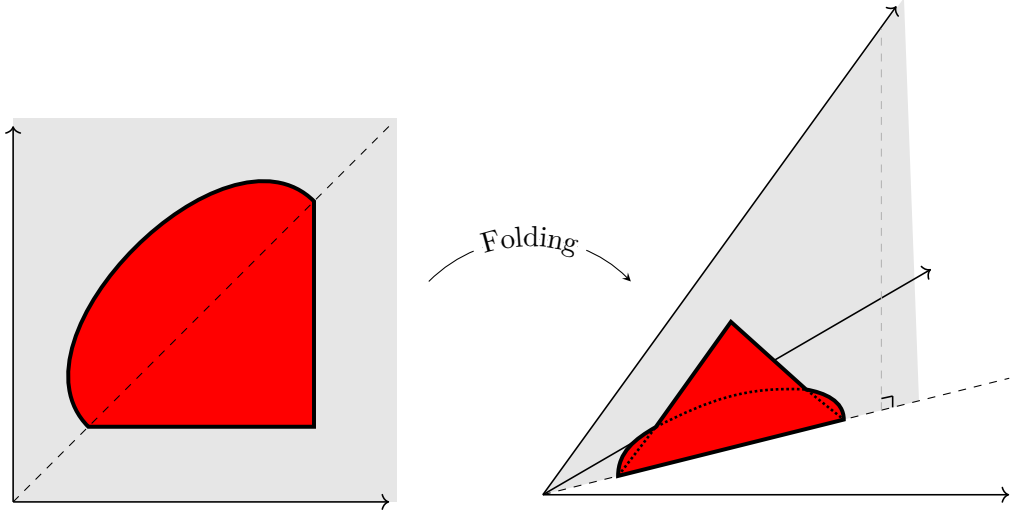
\begin{figure}
\begin{flushleft}
\begin{minipage}{0.25\textwidth}
\vspace{1.6cm}
\begin{flushleft}
\begin{tikzpicture}
\node (A) at (1,  1) {};
\node (B) at (4, 1) {};
\node (C) at (4,  4) {};

\fill[fill=gray!20] (0,0)--(0,5.1)--(5.1,5.1)--(5.1,0);
\filldraw[draw=black, fill=red, line width=1.5pt]
(A.center) to[out=135, in=135] (C.center) -- (B.center) -- cycle;
\draw[semithick, ->] (0,0) -- (5,0);
\draw[semithick, ->] (0,0) -- (0,5);
\draw[dashed] (0,0)--(5,5);
\end{tikzpicture}
\end{flushleft}
\end{minipage}
\hspace{2cm}
\begin{minipage}{0.1\textwidth}
\begin{tikzpicture}
\path[draw=black,-stealth, postaction={decorate,decoration={text effects along path,
    text={\ Folding\ }, text align=center,
    text effects/.cd, 
      text along path, 
      every character/.style={fill=white, yshift=-0.5ex}}}]
  (-2, 2) arc [start angle=135, end angle=45, radius=1.9];
\end{tikzpicture}
\end{minipage}
\begin{minipage}{0.4\textwidth}
\begin{flushright}
\begin{tikzpicture}
\node (A) at (1,  0.25) {};
\node (B) at (1.5, 0.9) {};
\node (C) at (2.5,  2.3) {};
\node (D) at (3.5,  1.4) {};
\node (E) at (4,1) {};

\fill[fill=gray!20] (0,0)--(4.8,6.61)--(5,1.25);
\draw[semithick, ->] (0,0) -- (6.2,0);
\draw[semithick, ->] (0,0) -- (5.16,3); 
\draw[semithick, ->] (0,0) -- (4.7,6.5); 
\draw[dashed] (0,0) -- (6.2,1.55); 
\draw[dashed,gray!60] (4.5,1.125) -- (4.5,6.21); 
\draw[semithick] (4.65,1.1625) -- (4.65,1.3125);
\draw[semithick] (4.5,1.28) -- (4.65,1.3125);
\filldraw[draw=black, fill=red, line width=1.5pt]
(A.center) to[out=90, in=210] (B.center) -- (C.center) -- (D.center) to[out=350, in=90] (E.center) -- cycle;
\draw[densely dotted,line width=1pt] (A.center) -- (B.center);
\draw[densely dotted,line width=1pt] (D.center) -- (E.center);
\draw[densely dotted,line width=1pt] (B.center) to[out=27, in=177] (D.center);
\end{tikzpicture}
\end{flushright}
\end{minipage}
\end{flushleft}
\caption{$G:=\SL_2(\R) \times \SL_2(\R)$ acts on $M_2(\R)$ by left and right multiplication, giving a representation $\rho:G \to \SL_4(\R)$. On the left hand side we see the joint spectrum of a set $S \subset G$ in the Weyl chamber $\R_+ \times \R_+$ and on the right hand side is the joint spectrum of $\rho(S)$ in the $3$-dimensional Weyl chamber of $\SL_4(\R)$.  }\label{fig.folding}
\end{figure}


 We will also show that every convex body contained in the Weyl chamber $\{x \in \R^d : x_1 \geq \ldots \geq x_d\}$ can be realized as the joint spectrum of some compact set $S \subset \GL_d(\C)$ generating a subgroup acting strongly irreducibly on $\C^d$, and that one can similarly realise every convex polyhedron with finitely many vertices in the chamber using a \emph{finite} set $S$. However we will also provide examples of finite sets with non-polyhedral joint spectrum, and even such examples where the boundary of the joint spectrum exhibits a dense set of conical singularities. This is related to the Lagarias-Wang finiteness conjecture, and will be discussed in Section \ref{section.non.polygonal.jointspectrum}.

The joint spectrum appears naturally in the context of  random matrix products. In \cite{sert-cras, Sert.LDP} the second named author establishes a large deviation principle for independent and identically distributed (i.i.d.) random matrix products. In particular, under the assumptions of Theorem \ref{main1}, if $\mu$ is a probability measure on $\GL_d(\C)$ whose support is $S$, he proves the existence of a rate function $I_\mu:\R^d \to \overline{\R}_+$ such that for an open set $U$
\begin{equation}\label{LDP}\P(\frac{1}{n}\kappa(g_1\cdot \ldots \cdot g_n) \in U) \simeq \exp(-n \inf_{x \in U} I_\mu(x))\end{equation}
as $n \to +\infty$ and the matrices $g_i$'s are chosen independently at random according to the distribution $\mu$. As it turns out, the joint spectrum $J(S)$ coincides with the closure of the effective support of the rate function $I_\mu$, namely 
$$J(S)=\overline{\{x \in \R^d: I_\mu(x) < +\infty}\}.$$This means that the Cartan vector of the random walk at time $n$ can get close to every point of the joint spectrum with at least exponentially small probability as $n \to +\infty$. In fact the function $I_\mu$ is strictly positive everywhere, except at a single point, which we call the \emph{Lyapunov vector} $\vec{\lambda}_\mu \in \R^d$, where it vanishes: $I_\mu(\vec{\lambda}_\mu)=0$. The Lyapunov vector is characterized by the law of large numbers for random matrix products \cite{furstenberg-kesten,bougerol-lacroix,benoist-quint-book}, which says that almost surely 
\begin{equation}\label{lyapunov}\frac{1}{n}\kappa(g_1\cdot \ldots \cdot g_n)  \underset{n \to +\infty}{\longrightarrow} \vec{\lambda}_\mu.\end{equation} Clearly the Lyapunov vector $ \vec{\lambda}_\mu$ belongs to $J(S)$. In fact we will show that it belongs to the relative interior of $J(S)$. For this we will rely on the central limit theorem for random matrix products \cite{goldsheid-guivarch, guivarch, benoist-quint-central} and \cite[Ch. 13]{benoist-quint-book}. 

It turns out  that when the law $\mu$ is allowed to vary among all laws supported on $S$,  the Lyapunov vector $\vec{\lambda}_\mu$ cannot reach every point in the relative interior of joint spectrum $J(S)$. We will give simple examples (see \S \ref{subsub.iid.Lyap}) showing that $\vec{\lambda}_\mu$ can be confined to a proper closed subset of $J(S)$. However this restriction disappears if we allow arbitrary ergodic stationary processes instead of restricting to the i.i.d.\ case. Indeed we will show in Section \ref{section5} that every point in the interior of $J(S)$ arises as the Lyapunov vector (defined by the convergence $(\ref{lyapunov})$) of an ergodic stationary process on the shift space $S^{\N}$. 

The set of such Lyapunov vectors is investigated in a recent article of Bochi \cite{bochi-icm}. Given a topological dynamical system $(X,T)$, with $X$ compact and $T:X\to X$ continuous, he fixes a  linear cocycle, i.e. a continuous map $F:X \to M_d(\C)$ and defines its Lyapunov spectrum $L^+(F) \subset \R^d$ as the set of all possible  Lyapunov vectors $\vec{\lambda}(\nu,F)$, where $\nu$ is an arbitrary $T$-invariant ergodic probability measure on $X$. This Lyapunov spectrum is closely related to the joint spectrum, and at least when $S \subset \GL_d(\C)$ generates a subgroup acting irreducibly on $\C^d$, the Lyapunov spectrum $L^+(F)$ will be sandwiched (see Theorem \ref{lyapspec}) between the interior of $J(S)$ and $J(S)$ itself, with $F$ taken to be the time $1$ map on the shift space $X:=S^\N$.

A simple but non-trivial observation of Daubechies and Lagarias \cite[Theorem 3.1.]{daubechies-lagarias} is that the joint spectral radius $R(S)$ can always be realized along a fixed sequence in $S^\N$. Namely there is $b=(b_1,b_2,\ldots) \in S^\N$ such that $R(S)=\lim_{n \to +\infty} \|b_1\cdot \ldots \cdot b_n\|^{\frac{1}{n}}$. We will see that analogously every point of the joint spectrum $J(S)$ is a limit of Cartan projections along a fixed sequence. Nevertheless,  $L^+(F)$ is not always closed and not every point of $J(S)$ can always be realized as the Lyapunov vector of some ergodic process on $S^\N$ (see \cite[Remark 1.2]{bochi-morris}).

Finally we will investigate the continuity properties of the joint spectrum, in the spirit of the article \cite{bochi-morris} by Bochi and Morris. We will show that under a domination condition, namely if $J(S)$ lies in the open chamber  $\{x \in \R^d: x_1>\ldots>x_d\}$, the joint spectrum is continuous under small perturbations of $S$ in the Hausdorff metric on compact subsets. In fact, in this case the Lyapunov vector $\vec{\lambda}(\nu)$ depends continuously on the measure $\nu$.

We will derive the above results in a more general setting, where the ambient group $\GL_d(\C)$ is replaced by an arbitrary reductive Lie group $G$ and $\Gamma=\langle S \rangle$ is assumed to be Zariski-dense in $G$. This covers the situation above because, modulo some finite index issues, the irreducibility assumption in Theorem \ref{main1} forces the Zariski closure of $\Gamma$ to be a reductive Lie group. We will not tackle here the interesting problem of removing irreducibility (or rather complete reducibility) in our main theorems and leave it for a forthcoming work. However we indicate two things about this issue: the first is that the existence of the limit holds when the irreducibility assumption is replaced by a \emph{domination condition} requiring that the renormalized Cartan projection stay away from the walls of the Weyl chamber of some reductive subgroup containing $S$ (see Theorem \ref{dominated-theorem}). The second is that the limit of Cartan vectors may in general (for non-reductive Zariski-closure) be strictly larger than the limit of Jordan vectors (see Example \ref{JneqK} below). Finally we assume throughout that our matrices are invertible. This allows us to use the well developed representation theory of reductive groups. Although it is likely that part of the above (in particular Theorem \ref{main1}) remains true for general non necessarily invertible matrices, and this is not difficult to confirm say when $d=2$, new ideas and  additional technical tools are required to handle the general case.

\bigskip
		
\begin{center}
* 
\nopagebreak
*    *
\end{center}

We now assume that $G$ is a connected reductive real Lie group, namely the connected component of the real points of a linear reductive algebraic group defined over $\R$. We refer the reader to \S \ref{reductive} below for a  review of the standard terminology on reductive groups needed to formulate our results. We recall here that a subgroup of $G$ is called Zariski-dense if it is not contained in any proper algebraic subgroup of $G$. Furthermore $G$ has a Cartan decomposition $G=K\exp(\a^{+})K$, where $K$ is a maximal compact subgroup of $G$ and $\a$ is the Lie algebra of a  maximal $\R$-split torus $A$, with $\a^+$ a positive (closed) Weyl chamber, i.e. the cone in $\a$ defined by the requirement that all positive roots be non-negative. This yields the so-called \emph{Cartan projection}:
$$\kappa: G \to \a^+.$$
At the same time we may define the \emph{Jordan projection}:
$$\lambda: G \to \a^+,$$ where $\lambda(g)$ is the unique element of $\a^+$ such that $\exp(\lambda(g))$ is conjugate to the hyperbolic (or polar) part of $g$ in the Jordan decomposition of $g$ as a commuting product of a unipotent element, an elliptic element and a hyperbolic element (i.e. an element with a conjugate in $A$). We will also consider the open Weyl chamber $\a^{++}$, which is the part of $\a^+$ where all positive roots are strictly positive.

We are ready to state our main results. The first statement encompasses Theorems \ref{main1} and \ref{main2} above.

\begin{theorem}[the joint spectrum] \label{joint} Let $G$ be a connected reductive Lie group and $S \subset G$ a compact subset generating a Zariski-dense subgroup. Then the sequences
$$\frac{1}{n}\kappa(S^n), \frac{1}{n}\lambda(S^n)$$
converge in Hausdorff metric  as $n\to+\infty$ to the same compact subset of $\a^+$, which we denote by $J(S)$ and call the \emph{joint spectrum of $S$}. Moreover for every $x \in J(S)$ there is a sequence $b=(b_1,b_2,\ldots) \in S^\N$ such that \begin{equation}\label{convergence-sequence}\lim_{n \to +\infty} \frac{1}{n}\kappa(b_1\cdot\ldots\cdot b_n) = x.\end{equation}
\end{theorem}

Recall that the Benoist cone $\mathcal{BC}(\Gamma)$ of a semi-group $\Gamma$ is the closure in  $\a^+$ of the positive linear combinations of $\lambda(g)$, $g \in \Gamma$. Obviously the cone spanned by $0$ and $J(S)$ is exactly the Benoist cone $\mathcal{BC}(\Gamma)$, where $\Gamma$ is the semigroup generated by $S$. Benoist proved in \cite{benoist-asymptotic1} that $\mathcal{BC}(\Gamma)$ is convex. When $G$ is semisimple he further showed that it has non-empty interior in $\a^+$. In fact $J(S)$ already has these properties:

\begin{theorem}[convex body]\label{body}  Under the same assumptions, the joint spectrum $J(S)$ is a closed convex subset of $\a^+$. Its affine hull contains a translate of $\a_S=Lie(A \cap [G,G])$. Moreover if $S$ is not contained in a coset of a proper closed connected Lie subgroup of $G$ containing $[G,G]$, then the affine hull of $J(S)$ is all of $\a$ and thus $J(S)$ has non-empty interior in $\a^+$.
\end{theorem}

For example if $G=\GL_d(\R)$, $S\subset G$ generates a Zariski-dense subgroup and $\{|\det(g)|: g \in S\}$ has more than one element, then $J(S)$ has non-empty interior in $\R^d$. 

\begin{figure}[H]
\begin{tikzpicture}

\coordinate (O) at (0, 0) {};
\coordinate (A) at (0.48, 1.8) {};
\coordinate (B) at (0.98, 3.2) {};
\coordinate (C) at (1.4, 3.62) {};
\coordinate (D) at (1.8, 3.74) {};
\coordinate (E) at (2.1, 3.7) {};
\coordinate (F) at (2.3, 3.6) {};
\coordinate (G) at (2.5, 3.3) {};
\coordinate (H) at (2.55, 2.75) {};
\coordinate (I) at (2.2, 2) {};
\coordinate (J) at (1.5, 1.2) {};

\draw[semithick, ->] (0,0) -- (0,5);
\draw[semithick, ->] (0,0) -- (4.86,2.7);
\draw[dashed] (0,0) -- (1.2, 4.5);

\draw[dashed] (0,0) -- (5, 4);
\draw[dashed] (0,0) -- (5, 4);
\draw[dashed] (0,0) -- (5, 4);

\fill[shade, top color=gray!10, bottom color=gray!70]
(0,0) -- (1.2, 4.5) -- (5.62,4.5) -- cycle;

\draw[gray] (0,3.74) -- (3.9, 3.74);
\draw[gray] (0,3.81) -- (0.07,3.81);
\draw[gray] (0.07,3.81) -- (0.07,3.74);
\draw[gray] (3.34,1.84) -- (1.7,4.7);
\draw[gray] (0.7,4.13) -- (3.9, 3.16);
\draw[gray] (0.7,4.814) -- (3.9, 2.414);
\draw[gray] (3.34,1.84) -- (1.7,4.7);
\draw[gray] (3.41,1.88) -- (3.37,1.95);
\draw[gray] (3.37,1.95) -- (3.3, 1.91) ;

\filldraw[draw=black, fill=red]
(O) -- (A) -- (B) -- (C) -- (D) -- (E) -- (F) -- (G) -- (H) -- (I) -- (J) -- cycle;

\draw (3.3,4.25) node {\small Benoist cone};
\draw (1.6,2.8) node {\small Joint};
\draw (1.6, 2.4) node  {\small spectrum};
\draw (1.2, 5) node  {$\mathfrak{a}^+$};

\end{tikzpicture}
\caption{Example of the joint spectrum of a compact set $S \subset \SL_3(\R)$ containing $1$, and the Benoist cone of the semigroup it generates inside the Weyl chamber $\a^+$.} \label{fig.benoist}
\end{figure}
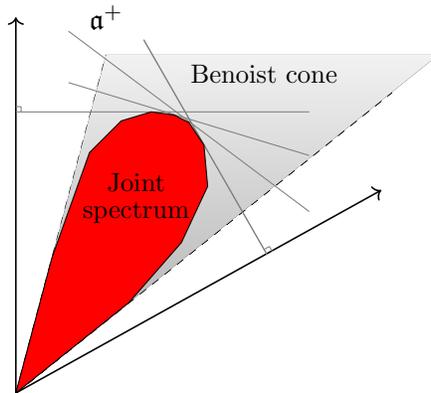

It is known \cite{benoist-asymptotic1} that every convex cone in $\a^+$ with non-empty interior can be realized as the Benoist cone of some Zariski-dense semigroup. The analogous property also holds for the joint spectrum.

\begin{theorem}[realization]\label{realization} Every convex body  $\mathcal{C}$ (i.e. closed convex subset with non-empty interior) in $\a^+$  can be realized as the joint spectrum $\mathcal{C}=J(S)$ for some compact subset $S$ generating a Zariski-dense subgroup of $G$. Furthermore if additionally $\mathcal{C}$ is a polyhedron with finitely many vertices, then $S$ can be chosen to be finite.
\end{theorem}

 A compact subset $S \subset \GL_d(\R)$ is said to be $1$-dominated if the ratio between the first two singular values of elements of $S^n$ grows exponentially fast in $n$ at some fixed rate. This notion and its variants have been studied a lot in dynamics in recent years \cite{bonatti-diaz-viana, avila-bochi-yoccoz,yoccoz,bochi-gourmelon,bochi-morris,bochi-icm}. An important result due to Bochi-Gourmelon \cite[Theorem B]{bochi-gourmelon} (see also \cite{bochi-sambarino-potrie,q-t-z}) asserts that $1$-dominated families are precisely those families satisfying a cone condition: there is a closed salient projective cone in $\mathbb{P}(\R^n)$, which is mapped into its interior by all elements of $S$. We will show in Proposition \ref{prop.domination.and.schottky} below that this is also equivalent to all large enough powers of $S$ forming a so-called Schottky family in the sense of \cite{benoist-asymptotic1}.
 
Analogously to this definition, we will say that a compact subset $S$ of a connected reductive Lie group $G$ is \emph{$G$-dominated} if $\kappa(S^n)/n$ remains in a compact subset of the open Weyl chamber $\a^{++}$ for all large enough $n$. When $G=\GL_d(\R)$ this is equivalent to asking that the wedge powers $\Lambda^k(S)$ are $1$-dominated for every $k=1,\ldots,d-1$. The following is a simple consequence of the proof of Theorem \ref{joint}.

\begin{theorem}[joint spectrum under a domination condition] \label{dominated-theorem} Let $G$ be a connected real reductive Lie group and $S \subset G$ a $G$-dominated compact set. Then $\frac{1}{n}\kappa(S^n)$ and $\frac{1}{n}\lambda(S^n)$ both converge to the same limit, $J(S)$, which is a convex body lying in $\a^{++}$. Moreover for every $x \in J(S)$ there is $b \in S^\N$ such that $(\ref{convergence-sequence})$ holds.
\end{theorem}

The only difference with Theorem \ref{joint} is that we no longer assume that $S$ generates a Zariski-dense subgroup, but assume instead that $S$ is $G$-dominated. Note further that if $S$ generates a Zariski-dense subgroup, and thus the joint spectrum $J(S)$ is well-defined by Theorem \ref{joint}, then $S$ is $G$-dominated if and only if $J(S) \subset \a^{++}$.

Under a $G$-domination condition we have the following useful continuity property.

\begin{theorem}[continuity]\label{continuity} Keep the assumptions of Theorem \ref{dominated-theorem}. For every $\eps>0$ there is $\delta>0$ such that if $d(S,S')<\delta$, then $d(J(S),J(S'))<\eps$.
\end{theorem}

Here $d(.,.)$ is the Hausdorff distance on compact subsets of $G$ and $\a$ respectively.  If $J(S)$ touches the walls of the Weyl chamber, then continuity will typically fail (see Figure \ref{fig.discontinuity}). Bochi and Morris \cite{bochi-morris} studied the continuity properties of the lower spectral radius (which can be obtained from $J(S)$ by the formula $\min\{\chi(x), x \in J(S)\}$, where $\chi$ is the highest weight of $G \subset \GL_n(\C)$). The situation for $\GL_2(\R)$ has been entirely analysed there and there is a complete description of the degeneracies: in this case $\a^+$ is a half plane $\{(x,y)\in \R^2: x\geq 0\}$ and if $J(S)$ touches the wall $\{x=0\}$, then under small perturbations of $S$ one can reach all points in the convex hull of the union of $J(S)$ and its reflection across the wall, see \S \ref{discon}. It is likely that the same phenomenon happens in general, cf. the notion of Morse spectrum in \cite{bochi-icm}.

In \cite{OregonReyes} and \cite{breuillard-fujiwara}  a geometric analogue to the notion of joint spectral radius was studied for subsets of isometries of an arbitrary metric space and a geometric Berger-Wang identity was obtained. In symmetric spaces, this can also be derived from the joint spectrum, see \S \ref{asymp} below. One can also consider arbitrary word metrics on reductive groups, in the spirit of the work of Abels and Margulis \cite{abels-margulis} and deduce a Berger-Wang identity for those, see \S \ref{abmarpar}. In fact the Abels-Margulis norm-like metric associated to a generating set $S$ can be read off  the joint spectrum:

\begin{theorem}[Word balls in reductive groups] \label{word-balls} Let $G$ be a connected reductive Lie group and $S$ a compact neighborhood of the identity with $S^{-1}=S$. Then there is $c \in \N$ such that for every $n \geq c$,
$$S^{n-c} \subset \{g \in G : \kappa(g) \in nJ(S) \} \subset S^{n+c}.$$
\end{theorem}

We now move on to describe the role of random matrix products in the study of the joint spectrum. Let $(g_n)_{n \geq 1}$ be an arbitrary ergodic stationary process with law $\mu$ on the shift space $S^{\N}$. Then by Kingman's subadditive ergodic theorem the following law of large numbers \cite{furstenberg-kesten, guivarch-ldgn, benoist-quint-book} holds almost surely:
$$\frac{1}{n}\kappa(g_1\cdot \ldots \cdot g_n)  \underset{n \to +\infty}{\longrightarrow} \vec{\lambda}_\mu.$$ The right hand side is called the \emph{Lyapunov vector} of the ergodic process and clearly lies in $J(S)$. In \cite{sert-cras,Sert.LDP} the second named author proves a large deviation principle for the Cartan projection of i.i.d.\ random walks supported on $S$. He shows the existence of a convex rate function such that $(\ref{LDP})$ holds for every open set $U \subset \a^+$. He moreover establishes that the support of the rate function coincides with the joint spectrum $J(S)$. See \cite[Theorem 5.1.]{Sert.LDP}. Here we show:

\begin{theorem}[Lyapunov lies in the interior]\label{interior} Let $(g_n)_{n\geq 1}$ be a sequence of i.i.d.\ random variables in $G$ and $S \subset G$ the support of their common law $\mu$. Assume that $S$ generates a Zariski-dense subgroup of $G$. Then the Lyapunov vector $\vec{\lambda}_\mu$ lies in the relative interior of the joint spectrum $J(S)$, and in particular in the interior of the Benoist cone of the semigroup generated by $S$.
\end{theorem}

By \emph{relative interior} we mean the interior of $J(S)$ inside its affine span, which by Theorem \ref{body} is all of $\a$ provided $S$ does not lie in a fixed coset of a proper closed connected Lie subgroup of $G$ containing $[G,G]$. The result was already mentioned in \cite[Theorem 5.1.]{Sert.LDP} with the proof being deferred to the present article. Theorem \ref{interior} refines another landmark theorem \cite{guivarch-raugi, goldsheid-margulis} \cite[Ch. 9]{benoist-quint-book} in the theory of random matrix products, namely the so-called \emph{simplicity of the Lyapunov spectrum}, which in our terminology means that $\vec{\lambda}_\mu$ belongs to the open Weyl chamber $\a^{++}$.  Our proof also gives that if $\mu$ has a finite second moment (and is not necessarily compactly supported), then the Lyapunov vector $\vec{\lambda}_\mu$  lies inside the Benoist cone, see Proposition \ref{prop.lyapunov.in.Benoist}.

The proof follows easily from the non-degeneracy of the normal law in the central limit theorem for the Cartan projection established in \cite{goldsheid-guivarch,guivarch} and studied further in \cite{benoist-quint-central} and \cite[Ch. 13]{benoist-quint-book}. Interestingly the proof of the non-degeneracy given in \cite{guivarch, benoist-quint-central} or \cite[Ch. 13]{benoist-quint-book} relies crucially on the following aperiodicity property of the Jordan vectors of a Zariski-dense subgroup: the closure in $\a$ of the additive group generated by the Jordan vectors $\lambda(g)$, $g \in \Gamma$, contains the entire semisimple part $\a_S:=\a \cap Lie([G,G])$ (see \cite{benoist-asymptotic1, quint-comptage} \cite[Ch. 7]{benoist-quint-book} for two different proofs)\footnote{It is even conjectured on the last page of \cite{benoist-asymptotic2} that this subgroup has a finitely generated dense subgroup and this was later confirmed by Prasad and Rapinchuk \cite{prasad-rapinchuk-schanuel} modulo the Schanuel conjecture from transcendental number theory.}. 

In fact in certain cases we show that  $\vec{\lambda}_\mu$ stays away from the boundary of $J(S)$ as $\mu$ varies among probability measures supported on $S$ (see also Remark \ref{rk.example.not.attained}). 

\begin{proposition}\label{confined} Let $S=\{a,b\}$ and $\mu_p:=p\delta_a + (1-p) \delta_b$ for $p\in (0,1)$, where $$a=\begin{pmatrix}
1 & 1\\
0 &1
\end{pmatrix}, b=\begin{pmatrix}
1 & 0 \\
1 & 1
\end{pmatrix}.$$ Then $J(S)=[0,\log R(S)]$ is a closed inverval of positive length, but there is $R'(S)>0$ with $R'(S)<R(S)$ such that $(0,\log R'(S)]$ is the set of values attained by the Lyapunov exponent $\vec{\lambda}_{\mu_p}=\lambda_1(\mu_p)$ as $p$ varies in $(0,1)$.
\end{proposition}

It would be interesting to determine how general a phenomenon this is. It does seem at least compatible with the idea from \cite{bousch-mairesse} that only sequences of low complexity can ever realize the joint spectral radius. This occurs in particular for every point on the \emph{positive boundary} of $J(S)$ (the points $x$ on the boundary of $J(S)$ for which there is a linear form $\ell$ on $\a$, which is a positive linear combination of fundamental weights, such that $\ell(y)\leq \ell(x)$ for every $y \in J(S)$).

If we relax the i.i.d.\ condition and allow ourselves to consider all stationary ergodic processes on the entire shift space $S^{\N}$, then it turns out that every point in the relative interior of the joint spectrum $J(S)$ can be realized as a Lyapunov vector:

\begin{theorem}[Lyapunov spectrum vs.\ joint spectrum]\label{lyapspec} Keep the assumptions of Theorem \ref{joint}. Let $x$ be a point in the relative interior of $J(S)$. Then there is an ergodic stationary process $(g_n)_{n \geq 1}$ on the shift space $S^{\N}$ such that almost surely
$$\frac{1}{n}\kappa(g_1\cdot \ldots \cdot g_n) \underset{n \to +\infty}{\longrightarrow} x.$$
\end{theorem}

In \cite{bochi-icm} Bochi calls the set of all such limits the Lyapunov spectrum of $S$. So we see that Bochi's Lyapunov spectrum lies in between the joint spectrum and its interior. Note that the Lyapunov spectrum is not always closed as not every point on the boundary of $J(S)$ belongs to it, see e.g. \cite[Remark 1.13]{bochi-morris} (and \S \ref{subsub.extremal} for an explicit example). 

In the following result, using subadditive ergodic theory, we show that every exposed point on the positive boundary of $J(S)$ lies in the Lyapunov spectrum.

\begin{proposition}[Exposed points and Lyapunov spectrum]\label{prop.extremal.points}
Let $G$ be a connected real reductive group.\\[1pt]
(1) Let $S\subset G$ be a $G$-dominated compact set. Then, every extremal point of $J(S)$ belongs to the Lyapunov spectrum.\\[2pt]
(2) Let $S \subset G$ be a compact subset generating a Zariski-dense semigroup in $G$. Then, every exposed point of $J(S)$ lying on the positive boundary belongs to the Lyapunov spectrum.
\end{proposition}

In Section \ref{section.non.polygonal.jointspectrum} we will discuss some concrete examples in $G=\SL_2(\R) \times \SL_2(\R)$. In particular, making use of the work of Bousch-Mairesse \cite{bousch-mairesse}, Jenkinson-Pollicott \cite{jenkinson-pollicott}, Morris-Sidorov \cite{morris-sidorov} and  Morris and Oregon-Reyes \cite{OregonReyes} on counter-examples to the so-called Lagarias-Wang finiteness conjecture, we give an example of a finite set $T \subset G$ for which the joint spectrum in $\R_+^2$ is a convex body with non-differentiable boundary, namely:

\begin{proposition}\label{notpoly} Let $G=\SL_2(\R) \times \SL_2(\R)$.  There is a finite set $T \subset G$  generating a Zariski-dense semi-group, such that its joint spectrum $J(T)$ is a convex body whose boundary is not piecewise $C^1$, and in particular not a polygon.
\end{proposition}

In fact, part of the boundary in this example will have a dense set of points of non-differentiability. See Figure \ref{fig.non-diff}. The Lagarias-Wang finiteness conjecture asked whether the joint spectral radius $R(S)$ of a finite set of matrices $S$ can always be realized by a periodic sequence. Computer simulations indicate that finiteness does occur empirically for most choices of $S$. However the conjecture was refuted by Bousch-Mairesse in \cite{bousch-mairesse} using the notions of Sturmian sequences and measures, with counter-examples in $\SL_2(\R)$ given later by I. Morris, see \cite{OregonReyes}.  The analogue of the finiteness conjecture in the setting of the joint spectrum would be to stipulate that $J(S)$ is always a polyhedron in $\a^+$ when $S$ is finite. And computer simulations seemed to confirm this. However the above proposition refutes it.\\

\noindent \emph{Remark.} Most of the techniques of the present paper are generalizable to the case when the field $\R$ or $\C$ is replaced by a non-archimedean local field $k$, in the spirit of \cite{benoist-asymptotic1,Quint.cones}. Theorems \ref{main1}, \ref{main2}, \ref{joint} hold without change. In Theorem \ref{body} convexity holds, but the joint spectrum can lie inside a wall of the Weyl chamber and be of empty interior. The other results require appropriate adjustments, which we will not discuss in this paper.

\bigskip

The paper is organized as follows. In Section \ref{section.preliminaries} we recall some preliminary material regarding reductive groups,  proximal transformations and the work of Abels-Margulis-Soifer and Benoist. In Section \ref{limit-section} we discuss the definition of the joint spectrum and prove Theorem \ref{joint} as well as Theorems \ref{main1} and \ref{main2} from this introduction. In Section \ref{properties} we complete the proofs of Theorems  \ref{body}, \ref{realization}, \ref{dominated-theorem}, \ref{continuity} and \ref{word-balls}. Section \ref{section5} is devoted to random matrix products and the proof of Theorems  \ref{interior}, \ref{lyapspec} and Propositions \ref{confined}, \ref{prop.extremal.points}. In Section \ref{section.non.polygonal.jointspectrum}, we describe some concrete examples in $\SL_2(\R)\times \SL_2(\R)$, discuss the finiteness conjecture and prove Proposition \ref{notpoly}. In the final section, we discuss some further directions of research.

\subsection*{Acknowledgements}
The authors are thankful to Jairo Bochi, Koji Fujiwara, Yves Guivarc'h and Carlos Matheus for several useful discussions and comments, to Weikun He and Oliver Sargent for interesting computer simulations and to an anonymous referee for a careful reading and useful remarks. They also thank the Isaac Newton Institute and WWU M\"{u}nster for their support while part of this work was conducted.  E.B. acknowledges support from ERC grant GeTeMo no.617129. C.S. acknowledges support from SNF grant 200021-152819.

\setcounter{tocdepth}{1}
\tableofcontents

\section{Proximal transformations}\label{section.preliminaries}

\subsection{Dynamics of projective linear maps}\label{dynamics}

Let $V$ be a finite dimensional  real vector space and $\mathbb{P}(V)$ its projective space.  We will work with the \emph{standard distance} on $\P(V)$: $d([x],[y]):=\frac{||x \wedge y||}{||x||.||y||}$, where $x,y \in V\setminus \{0\}$ and  $[x],[y]$ are the lines in $\P(V)$ induced by $x,y$. Here  we have fixed a Euclidean norm $||.||$ on $V$, which naturally induces a Euclidean norm on $\bigwedge^{2} V$. Recall the following definition (see \cite{tits,  benoist-asymptotic1, breuillard-gelander, abels-proximal}):

\begin{definition}[proximal map] A linear map $g \in Hom(V,V)$ is said to be proximal if it has a unique eigenvalue $\alpha$ of maximal modulus. Denote by $v_{g}^{+}$, the line in $\P(V)$ corresponding to the one dimensional eigenspace of $\alpha$ and $H_{g}^{<}$ the supplementary $g$-invariant hyperplane in $\P(V)$.
\end{definition}

We also recall a classical quantification of this definition (see \cite{AMS,Benoist1, breuillard-gelander}).

\begin{definition}[$(r,\eps)$-proximal map]
Let $0 < \eps \leq r$. A proximal element $ g \in End(V) $ is said to be $(r, \eps)$-proximal, if $d(v_{g}^{+},H_{g}^{<})\geq 2r$ and $d(gx,gy) \leq \eps d(x,y)$ for every $x,y \in \P(V)$ such that $d(x,H_g^{<}) \geq \eps$ and $d(y,H_g^{<}) \geq \eps$.
\end{definition}

\begin{remark}\label{Lipschitz.action.lemma} We note that  every $g \in \GL(V)$ gives rise to a Lipschitz transformation of $\mathbb{P}(V)$  with Lipschitz constant $Lip(g) \leq \|\Lambda^{2}g\|.\|g^{-1}\|^{2}.$ This is immediate from the definition of the standard distance.
\end{remark}

In his proof of the Tits alternative \cite[\S 3.8]{tits} Tits gave a simple criterion for a projective transformation to be proximal. The following version of this criterion is from \cite[Lemme 1.2]{benoist-asymptotic2}.

\begin{lemma}[Tits proximality criterion]\label{Tits.criterion.sublemma}
Let $g \in \GL(V)$, $H\subset \mathbb{P}(V)$ be the projective image of an hyperplane, and $x\in \mathbb{P}(V)$. Let $r\geq \eta >0$ be given and denote $B_{H}^{\eta}=\{z \in \mathbb{P}(V)\, | \, d(z,H)\geq  \eta \}$ and $b_{x}^{\eta}=\{z \in \mathbb{P}(V)\, | \, d(z,x)\leq  \eta \}$. Suppose that $d(x,H)\geq 6r$, $gB_{H}^{\eta}\subseteq b_{x}^{\eta}$ and that the restriction of $g$ to $B_{H}^{\eta}$ is $\eta$-Lipschitz. Then $g$ is $(2r,2\eta)$-proximal. Moreover $v^+_g \in b_x^{\eta}$ and $H_g^{<} \subset \P(V)\setminus B_H^{\eta}$.  
\end{lemma}
\begin{proof}
Since $r\geq \eta$ and $1 \geq d(x,H)\geq 6r$, it follows from the assumptions that $b_{x}^{\eta} \subset B_{H}^{\eta}$ and hence $gb_{x}^{\eta} \subseteq b_{x}^{\eta}$. Since $g$ is moreover $\eta$-Lipschitz on $b_{x}^{\eta}$ and $1>\eta$, $g$ admits a unique fixed point in $b_{x}^{\eta}$. Denote it by $v_{g}^{+}$. It follows that $g$ is proximal and $v_{g}^{+}$ is the unique attracting direction of $g$. Let $H_{g}^{<}$ be the repelling hyperplane of $g$. Clearly, $H_{g}^{<} \subset (B_{H}^{\eta})^{c}$ and hence $d(v_{g}^{+},H_{g}^{<})\geq 6r -2\eta \geq 4r$. For the same reason, $B_{H_g^<}^{2\eta}\subset B_{H}^{\eta}$ and therefore we have $gB_{H_g^<}^{2\eta} \subseteq gB_{H}^{\eta} \subseteq b_{x}^{\eta} \subseteq b_{v_g^+}^{2\eta}$ and the claim follows.
\end{proof}

When $\eps$ is very small, an $(r,\eps)$-proximal map resembles a rank $1$ linear map. This observation is behind the proof of the two statements below. Let $\lambda_1(g)$ be the spectral radius of $g \in Hom(V,V)$, namely the largest modulus of its eigenvalues. 

\begin{lemma}\label{proximal.implies1} For $0 < \eps \leq r$,  there exists a constant $c_{r,\eps} \in (0,1)$ such that, for each $r > 0$, $ \underset{\eps \rightarrow 0}{\lim} \, c_{r,\eps} = 2r$, and for every $(r,\eps)$-proximal endomorphism $g$ of $V$, we have 
\begin{equation*}
c_{r,\eps}\|g\| \leq \lambda_{1}(g) \leq \|g\|.
\end{equation*}
\end{lemma}

\begin{proof} By scaling we may assume $\|g\|=1$. If  $(g_{k})_{k \in \mathbb{N}}$ is a sequence of $(r, \eps_{k})$-proximal transformations with $\|g_k\|=1$ and $\eps_{k} \to 0$, then $\lim_{k \rightarrow \infty}g_{k}=\pi$, a rank $1$ map with of $d(\Ima \pi, \ker \pi) \geq 2r$ and $\|\pi\|=1$. We need to check that $\lambda_1(\pi) \geq 2r$. Let $e$ be a unit vector orthogonal to $\ker \pi$ and write $e=v +k$ for $v \in \Ima \pi$ and $k \in \ker \pi$, then $v \wedge k=e \wedge k$ and so $$d([v],[k])=\frac{\|v \wedge k\|}{\|v\| \|k\|}=\frac{1}{\|v\|} \geq   d(\Ima \pi, \ker \pi) \geq 2r.$$  But then $\|\pi\|=\|\pi(e)\|=1$, so $\|\pi(v)\|=\lambda_1(\pi)\|v\|=1$. It follows that $\lambda_1(\pi) \geq 2r$ as desired.
\end{proof}

Similarly, and more precisely, one has:

\begin{lemma}\label{lambda-norm} For $0 < \eps \leq r$,  there exist constants $D_{r,\eps}>1$ such that, for each $r > 0$, $ \underset{\eps \rightarrow 0}{\lim} \, D_{r,\eps} = 1$, and for every $(r,\eps)$-proximal endomorphism $g$ of $V$, and every $x \in V\setminus\{0\}$ with $d([x],H_g^{<})\geq r$
\begin{equation*}
D_{r,\eps}^{-1} \leq \frac{\|gx\|}{\|x\|} \frac{d(v_g,H_g^{<})}{d([x],H_g^{<})} \frac{1}{\lambda_1(g)} \leq D_{r,\eps}.
\end{equation*}
\end{lemma}
\begin{proof} Again by compactness this reduces to the case when $g$ is a rank one matrix, where the above expression is then equal to $1$ for all $x \notin \ker g$.
\end{proof}

One can iterate proximal transformations and keep track of the spectral radius of the product provided the attracting points and repelling hyperplanes are well separated. One has the following proposition borrowed from \cite[Prop.\ 1.4]{benoist-asymptotic2}.

\begin{proposition} \label{spectralcontrolproximal}
For  $0<\eps \leq r$, there exists a positive constant $D_{r,\eps}>1$ with the property that for each $r > 0 $, we have $\lim_{\eps \rightarrow 0}D_{r, \eps} =1$ and such that if $g_{1}, \ldots g_{\ell}$ are $(r,\eps)$-proximal linear transformations of $V$ satisfying $d(v_{g_{\ell}}^{+},H_{g_{1}}^{<}) \geq 6r $ and $d(v_{g_{j}}^{+},H_{g_{j+1}}^{<}) \geq 6r $, for all $j=1, \ldots \ell-1$, we have that for all integers $n_{1}, \ldots, n_{l} \geq 1$, the linear transformation $g_{\ell}^{n_{\ell}}\ldots g_{1}^{n_{1}}$
is $(2r,2\eps)$-proximal, and
\begin{equation*}
\beta(g_1,\ldots,g_\ell)  D_{r,\eps}^{-\ell} \leq \frac{\lambda_{1}(g_{\ell}^{n_{\ell}}\ldots g_{1}^{n_{1}})}{\lambda_{1}(g_{\ell})^{n_{\ell}} \ldots \lambda_{1}(g_{1})^{n_{1}}} \leq D_{r,\eps}^{\ell} \beta(g_1,\ldots,g_\ell)
\end{equation*}
where 
$$\beta(g_1,\ldots,g_\ell)  := \frac{d(v_{g_\ell}^+,H_{g_1}^{<}) d(v_{g_1}^+,H_{g_2}^{<})\ldots d(v_{g_{\ell-1}}^+,H_{g_\ell}^{<})}{d(v_{g_1}^+,H_{g_1}^{<})\ldots d(v_{g_\ell}^+,H_{g_\ell}^{<})}.$$
\end{proposition}

\begin{proof} 
Let $h:=g_\ell^{n_\ell}\ldots g_1^{n_1}$. Note that $h$ satisfies the assumptions of Lemma \ref{Tits.criterion.sublemma} with $H=H_{g_1}^{<}$, $x=v^+_{g_\ell}$ and $\eta=\eps$. In particular $h$ is proximal and $d(v_h^+,v_{g_\ell}^+)\leq \eps$.  Now apply Lemma \ref{lambda-norm} $\ell$ times starting with $x=v_h$ and $g=g_1^{n_1}$, then successively $x=g_i^{n_i}\ldots g_1^{n_1}v_h$ and $g=g_{i+1}^{n_{i+1}}$ for $i=1,\ldots,\ell-1$. The assumptions guarantee that it is legitimate to do so. Now take the product, and observe that $\|hv_h\|=\lambda_1(h)\|v_h\|$. We then obtain the desired inequalities with each $d(v_{g_i}^+,H_{g_{i+1}}^{<})$ in the numerator of $\beta$ replaced by $d(g_i^{n_i}\ldots g_1^{n_1}v_{h}^+,H_{g_{i+1}}^{<})$. However the ratio of the two is at most $\eps/6r$ close to $1$, because $d(g_i^{n_i}\ldots g_1^{n_1}v_{h}^+,v_{g_i}^+)\leq \eps$ and $d(v_{g_{i}}^{+},H_{g_{i+1}}^{<}) \geq 6r $. Hence the result. 
\end{proof}

This proposition motivates the following terminology, also borrowed from  \cite[Def. 1.7]{benoist-asymptotic2}, which will be of important use to us in the sequel:

\begin{definition}[Schottky family] \label{defnarepsSchottky1}
 A subset $E$ of $GL(V)$ is called an $(r,\eps)$-Schottky family if \\[3pt]
a. For all $\gamma \in E$, $\gamma$ is $(r,\eps)$-proximal, and\\[3pt]
b. $d(v_{\gamma}^{+},H_{\gamma'}^{<}) \geq 6r$, for all $\gamma, \gamma' \in E$.
\end{definition}

\subsection{Connected reductive groups}\label{reductive} Let $G \subset \GL_n(\R)$ be a real reductive linear Lie group. Reductive means that it has no unipotent normal subgroup. For background on reductive Lie groups, we refer the reader to \cite{knapp, borel-tits, Mostow2} for example. Let $\g=Lie(G)$ be the Lie algebra of $G$ and $Ad:G \to \GL(\g)$ be its adjoint representation. We let $A$ be a maximal $\R$-split torus. Recall that $A$ is a closed connected Lie subgroup isomorphic to $(\R_+^*)^d$, for an integer $d$ called the rank of $G$.  The Lie algebra $\g$ decomposes as a direct sum of joint eigenspaces for the action of $Ad(A)$, the root spaces $\g_\alpha$, where $\alpha: A \to \R_+^*$ are the weights of $A$, i.e. $Ad(a)v=\alpha(a)v$ for all $a \in A$, $v \in \g_\alpha$. If $\a=Lie(A)$ denotes the Lie algebra of $A$, then we set $\overline{\alpha} \in Hom(\a,\R)$ to be the linear form defined by $\exp(\overline{\alpha}(x))=\alpha(\exp(x))$ for $x \in \a$. The non-zero $\overline{\alpha}$'s  appearing this way form a root system denoted by $\Sigma$. We can choose a base of simple roots $\Pi=\{\overline{\alpha}_1,\ldots,\overline{\alpha}_{d_S}\}$ for $\Sigma$, so that $\Sigma$ splits as a union of positive roots $\Sigma^+$ (the non-negative linear combinations of simple roots) and negative roots $-\Sigma^+$. The integer $d_S$ is called the semisimple rank of $G$ and $d-d_S$ is the dimension of the center of $G$. For example, when $G=\GL_d(\C)$ viewed as a real group, then $d=d_S+1$, $A=(\R_+^*)^d$ is the subgroup of diagonal matrices with real positive entries $(\lambda_1,\ldots,\lambda_d)$, and the roots  are the linear forms $\overline{\alpha}_{i,j}:=\log \lambda_i - \log \lambda_j$ for $1\leq i,j \leq d$ and a base of simple roots is formed by the $\overline{\alpha}_{i,i+1}$, $i=1,\ldots,d-1$.

\subsubsection{Cartan projection} Recall further the Cartan decomposition of $G$, namely $G=KAK$. Here $K$ is a maximal compact subgroup of $G$, whose Lie algebra is orthogonal to $\a$ for the Killing form $\tr(ad(x)ad(y))$ on $\g$. In this decomposition, the $A$-component of an element $g \in G$ is not uniquely defined, but it has a unique representative in the multiplicative cone $A^+:=\exp(\a^+)$, where $\a^+$ is the (closed) Weyl chamber $\a^+:=\{x \in \a : \overline{\alpha}_i(x) \geq 0 \textnormal{ for all } i=1,\ldots,d_S\}$. The Weyl group $W=N_G(A)/Z_G(A)$ is a finite group which acts by conjugation on $A$, and hence on $\a$ and permutes the roots in $\Sigma$. The Weyl chamber $\a^+$ is a fundamental domain of the action of $W$ on $\a$: the $W$-orbit of any $x \in \a$ intersects $\a^+$ in a unique point. This allows to define the \emph{Cartan projection}:
$$\kappa: G \to \a^+$$
which sends $g$ to the unique $x \in \a^+$ such that $g \in K\exp(x)K$. In fact two elements in $A$ are conjugate in $G$ if and only if they are conjugate by an element of $W$. In the case $G=\GL_d(\C)$, $K=U_d(\C)$ is the unitary group and $\kappa(g)$ is the vector of logarithms of singular values of $g$, which we have considered earlier, and $\a^+=\{x \in \R^d: x_1 \geq \ldots \geq x_d\}$.

We will also consider the open Weyl chamber:
\begin{equation}\label{openweyl}\a^{++}:=\{x \in \a : \overline{\alpha}_i(x) > 0 \textnormal{ for all } i=1,\ldots,d_S\}.\end{equation}

We note that the Lie algebra $\a$ is the orthogonal direct sum \begin{equation}\label{torus}\a=\a_Z\oplus \a_S\end{equation} of the Lie subalgebra $\a_Z$ of $Z(G)\cap A$, where $Z(G)$ is the center of $G$ and the Lie subalgebra $\a_S$ of $A \cap [G,G]$, where $[G,G]$ is the (closed) commutator subgroup of $G$, which is a semisimple Lie subgroup. Note that $G=[G,G]$ if and only if $Z(G)$ is finite. The Weyl chamber $\a^+$ is the direct sum of a salient cone $\a^+ \cap \a_S$ with the vector subspace $\a_Z:=\Lie(Z(G)\cap A)$.

\subsubsection{Jordan projection} Every element $g$ of $G$ admits a Jordan decomposition $g=g_sg_u$, where the two factors $g_s$ and $g_u$ commute and $Ad(g_s)$ is semisimple in $\GL(\g)$ and $Ad(g_u)$ unipotent. The factors are uniquely defined and are called respectively the semisimple and the unipotent parts of $g$. The semisimple part $g_s$ further decomposes as $g_s=g_eg_h$, where $g_e$ and $g_h$ commute, $g_h$ is conjugate to an element in $A$ and $g_e$ is elliptic, i.e. lies in a compact subgroup. Again the elliptic part $g_e$ and the hyperbolic part $g_h$ are uniquely defined. This allows to define the \emph{Jordan projection}:
$$\lambda: G \to \a^+$$
which sends $g$ to the unique element $\lambda(g) \in \a^+$ such that $g_h$ is conjugate to $\exp(\lambda(g))$. In the case when $G=\GL_d(\C)$, $\lambda(g)$ is the vector of logarithms of the moduli of the eigenvalues of $g$, ordered in decreasing order.

\subsubsection{Zariski topology} Connected reductive linear Lie groups are also linear algebraic groups, in the sense that they are the identity connected component $G=\G(\R)^\circ$ of the group of real points $\G(\R)$ of a reductive linear algebraic group $\G$ defined over $\R$. The group $\G$ admits a faithful linear representation $\G \to \GL_n$ and can therefore be seen as the set of zeroes of polynomial maps in $\R[x_{ij},\det x^{-1}]$, where the indeterminates are the entries of $\M_n(\R)$ together with the inverse determinant function. This interpretation allows to speak of the Zariski topology on $G$. This topology is weaker (has fewer closed sets) than the usual Lie group topology on $G$: it is not even Hausdorff. We will not make extensive use of it, but we only recall here some well-known facts. The Zariski closure of a subset of $G$ is by definition the intersection of all algebraic subsets  in $\GL_n(\R)$ (i.e.  the locus of vanishing of a family of polynomial maps) containing that subset.  It is itself an algebraic subset, i.e. it is Zariski-closed. This notion is independent of the choice of the embedding $\G \to \GL_n$. Furthermore, the Zariski-closure of a semigroup in $\G$ is always a Zariski-closed subgroup, in particular it has the same Zariski-closure as the subgroup it generates. Finally a subgroup of $G$ is Zariski-dense if and only if it acts irreducibly on all irreducible linear representations of $G$. We refer the reader to \cite{humphreys,borel, borel-tits} for this background.

\subsubsection{Representations}\label{rep}
For this paragraph we refer the reader to \cite[\S 6]{AMS}, \cite[\S 12]{borel-tits} and \cite{Benoist1} and \cite[Ch. 8]{benoist-quint-book}. Let $(V,\rho)$ be a finite dimensional linear representation of $G$ over the reals. The weights of $(V,\rho)$ are the characters $\chi: A \to \R_+^*$ such that the associated weight space $V_{\chi}=\{v\in V \; | \; \forall a \in A, \rho(a)v=\chi(a)v\}$ is non-trivial. Characters of $A$ form a free abelian group of rank $d$.  We write $\overline{\chi} \in Hom(\a,\R)$ the differential of a character $\chi: A \to \R_+^*$, i.e. $\chi(\exp(x))=\exp(\overline{\chi}(x))$ for all $x \in\a$, and also call it character, or weight, by abuse of language.  A character $\chi$ is said to be \emph{dominant} if $\overline{\chi}(x) \geq 0$ for all $x \in \a^+$. If $(V,\rho)$ is irreducible, then its set of weights admits a maximal element $\chi_\rho$ (for the partial order $\chi_1 \leq \chi_2 \Leftrightarrow \chi_1(x)\leq \chi_2(x)$ for all $x \in \a^{+}$), which is dominant and called the highest weight of $(V,\rho)$.  A representation $(V,\rho)$ is said to be \emph{proximal} if $\dim(V_{\chi_{\rho}})=1$.

We fix as scalar product $\langle \cdot, \cdot \rangle$ on $\a$ whose restriction to $\a_S$ is a multiple of the Killing form and is otherwise arbitrary on $\a_Z$, only requiring the direct sum in $(\ref{torus})$ to be orthogonal. To every simple root $\overline{\alpha} \in \Pi$ one can associate a non-zero dominant weight $\overline{\omega}_\alpha \in Hom(\a,\R)$ such that  $\overline{\omega}_\alpha(\a_Z)=0$ and $\langle\overline{\omega}_{\alpha},\overline{\beta}\rangle =0$ if $\beta \in \Pi\setminus\{\alpha\}$. Every non-negative integer linear combination of the $\overline{\omega}_{\alpha}$'s is the highest weight of some (absolutely) irreducible representation of $G$. These representations are strongly rational over $\R$ in the terminology of \cite[\S 12]{borel-tits}. Moreover these representations are all proximal (see e.g. \cite[Thm 6.3]{AMS}). We may complete the $\overline{\omega}_{\alpha}$'s, who form a basis of $\Hom(\a_S,\R)$, into a basis of $\Hom(\a,\R)$ by adding $d-d_S$ characters of $G/[G,G]$.  We call the representations thus obtained the \emph{distinguished representations} $\rho_1,\ldots,\rho_d$. For $G=\GL_d(\mathbb{C})$, these are given by the exterior power representations $\wedge^i$ for $i=1,\ldots,d$. 

As is well-known \cite{Mostow1} \cite[\S 2]{Mostow2}, on every irreducible representation $(V,\rho)$ we may find a Euclidean norm invariant under $\rho(K)$ such that $\rho(G)$ is stable under adjoint and $\rho(A)$ is diagonalizable in an orthonormal basis. Then clearly for every $g \in G$,
\begin{equation}\label{param}\overline{\chi}_{\rho}(\kappa(g)) = \log \|\rho(g)\|  \textnormal{  and  }\overline{\chi}_{\rho}(\lambda(g)) =  \lambda_1(\rho(g)).\end{equation}
Thus distinguished representations allow us to express the Cartan projection as a vector of norms, and the Jordan projection as a vector of spectral radii. In particular the following is an immediate consequence of the ordinary spectral radius formula:

\begin{lemma}\label{spectral.radius.formula.corollary}
For every $g \in G$, we have $\frac{1}{n} \kappa(g^{n}) \rightarrow \lambda(g)$ as $n \to + \infty$.
\end{lemma}

\subsection{Proximal elements in $G$}

Let as above $G$ be a real reductive group and $\a^{++}$ the open Weyl chamber (\ref{openweyl}).

\begin{definition} An element $g \in $G is said to be $G$-proximal if $\lambda(g) \in \a^{++}.$ 
\end{definition}

This is also known as a \emph{loxodromic element} or an \emph{$\R$-regular element}. By $(\ref{param})$ we see that this is equivalent to asking that $\rho_i(g)$ is a proximal transformation in $\GL(V_{\rho_i})$ for each $i=1,\ldots,d$ (or equivalently $i=1,\ldots,d_S$). Similarly we quantify this notion as follows. Recall that we have fixed Euclidean norms on each $V_{\rho_i}$. 

\begin{definition} Let $r \geq \eps>0$. An element $g \in $G is said to be $(r,\eps)$-proximal in $G$ (or $(r,\eps)$-$G$-proximal) if $\rho_i(g)$ is $(r,\eps)$-proximal for each $i=1,\ldots,d_S$. 
\end{definition}

We also record here  the following simple fact (\cite[\S 4.6]{Benoist1}):

\begin{lemma}[Continuity of Cartan projection] \label{Cartan.stability}
Let $G$ be a connected reductive Lie group. There exists a number $C_G>0$ such that every $g,h_1,\ldots,h_n \in G$, we have \begin{align*}\|\kappa(h_1\cdot \ldots \cdot h_n) \| &\leq C_G(\|\kappa(h_1)\|+\ldots+\|\kappa(h_n)\|)\\ \|\kappa(h_1gh_2) -\kappa(g)\| &\leq C_G(\|\kappa(h_1)\|+\|\kappa(h_2)\|).\end{align*}
\end{lemma}
\begin{proof}
This is immediate from $(\ref{param})$ and the submultiplicativity of operator norms in each distinguished representation. 
\end{proof}

For the same reason, we also obtain the multidimensional counterparts to the above Lemma \ref{proximal.implies1} and Proposition \ref{spectralcontrolproximal} also from \cite{Benoist1}. 

\begin{lemma}\label{loxodromy.implies.Cartan.close.to.Jordan}
Let $G$ be a real reductive group and $r>0$ a number. Then, there exists a constant  $C_{r}>0$ such that if $g\in G$ is $(r,\eps)$-proximal in $G$ for some $\eps \in (0,r)$, then $\|\lambda(g)-\kappa(g)\|\le C_r$.
\end{lemma}



\begin{proposition}\label{Best}Let $G$ be a real reductive group. For every $r > 0 $, there exists a constant $C_r>0$ such that for every $\ell \in \N$ if $g_{1},\ldots, g_{\ell}$ are elements of $G$ that are $(r,\eps)$-proximal in $G$ and have the property that $d(v_{\rho_{i}(g_{j})}^{+}, H_{\rho_{i}(g_{j+1})}^{<}) \geq 6r$ for all $j=0,\ldots, \ell-1$ (where we wrote $g_0:=g_\ell$) and for all $i=1,\ldots,d_S$, then for all $n_{1}, \ldots, n_{\ell} \in \N$, the product $g_{\ell}^{n_{\ell}}\ldots g_{1}^{n_{1}}$ is $(2r,2\eps)$-proximal in $G$, and satisfies
\begin{equation*}
\|\lambda(g_{\ell}^{n_{\ell}}\ldots g_{1}^{n_{1}}) - \sum_{i=1}^{\ell} n_{i}\lambda(g_{i})\| \leq C_r \ell.
\end{equation*}
\end{proposition}

\begin{proof}This is immediate from Proposition \ref{spectralcontrolproximal} and from $(\ref{param})$ after observing that in each representation $\rho_i$ the term $|\log \beta(g_1,\ldots,g_\ell)|$ is bounded by $\ell |\log 6r |$.
\end{proof}

Analogously to Definition \ref{defnarepsSchottky1} we have:

\begin{definition}[Schottky family in $G$]\label{defnarepsSchottky2}
Let $G$ be as above, $r \geq \eps >0$ be given constants. A subset $E$ of $G$ is said to be an $(r,\eps)$-Schottky family in $G$ (or $(r,\eps)$-$G$-Schottky family), if for each $i=1,\ldots,d_S$ the set $\rho_{i}(E)$ acting on $V_{\rho_i}$ is an $(r,\eps)$-Schottky family.
\end{definition}

\subsection{Abels-Margulis-Soifer}


An important and non-trivial fact about Zariski-dense semigroups of reductive Lie groups is that every such contains a $G$-proximal (or $\R$-regular) element, see \cite{benoist-labourie, guivarch-raugi,goldsheid-margulis,Prasad}. Such elements are even Zariski-dense in $G$ \cite{benoist-asymptotic1}. 

The following result of Abels-Margulis-Soifer \cite[Thm. 6.8]{AMS} is an important refinement of these facts. It will be of crucial use in our considerations.  

\begin{theorem}[Abels-Margulis-Soifer \cite{AMS}]\label{AMS} 
Let $G$ be a connected reductive real Lie group  and $\Gamma$ a Zariski-dense subsemigroup. Then, there exists $0<r=r(\Gamma)$ such that for all $0<\eps \leq r$, there exists a finite subset $F=F(r,\eps,\Gamma)$ of $\Gamma$ with the property that for every $\gamma \in G$, there exists $f \in F $ such that $\gamma f$ is $(r,\eps)$-proximal in $G$.
\end{theorem}

An important feature is the Zariski-connectedness of the ambient group. For further use we record the following lemma (also used in \cite{AMS}), where this feature is key.

\begin{lemma}[simultaneous transversality]\label{dispersion.lemma}
Let $\Gamma$ be a Zariski-dense subsemigroup of $G$. Let $k \in \mathbb{N}$, for $i=1,\ldots,k$, let $(V_i,\pi_i)$ be a given irreducible linear representation of $G$ and for $t \in \mathbb{N}$ and $j=1,\ldots,t$, let $v_i^j$ and $H_i^j$ denote respectively a point and a hyperplane in the projective space $\P(V_i)$. Then there exists $\gamma \in \Gamma$ such that $\pi_{i}(\gamma)v_{i}^{j}\notin H_i^j$ for all $i=1,\ldots,d_{S}$ and $j=1,\ldots,t$.
\end{lemma}

\begin{proof}
 Indeed, for each $i,j$ the subsets $\{ g \in G\,|\,\pi_{i}(\gamma)v_{i}^{j} \notin H_i^j\}$ are non-empty Zariski-open subsets of $G$. By connectedness their intersection is still non-empty and Zariski open. Hence it intersects $\Gamma$.
\end{proof}

\begin{remark}\label{unif-rem}
By compactness of projective spaces, we can conclude the stronger statement that $d(\pi_i(\gamma)v_i^j,H_i^j)>r$ for all $i,j$ and that $\gamma$ can be chosen from a finite set $F=F(t,\Gamma)$ such that $r=r(t,\Gamma)>0$ and $F$ are both independent of the $v_i^j$'s and $H_i^j$'s. See \cite[Lemma 4.3]{breuillard-gelander-1}.
\end{remark}

\section{Existence of limits}\label{limit-section}

The aim of this section is to prove  Theorems  \ref{main1}, \ref{main2} and \ref{joint} from the introduction. We begin with Theorem \ref{joint}.


We need to show that  certain sequences of compact subsets of the  Cartan Lie algebra $\a$  of $G$ converge in Hausdorff metric. To this end  we  first recall the following well known fact (see \cite[Section 4.4]{ambrosio-tilli}).

\begin{lemma}\label{Hausdorff.convergence.lemma}
Let $(X,d)$ be a compact metric space. A sequence $(K_{n})_{n \in \mathbb{N}}$ of compact subsets of $X$ converges in Hausdorff metric to  a compact subset of $X$  if and only if for every $\delta>0$ there is $n_0 \in \N$ such that for all $n\geq n_0$ and every $x \in K_n$  we have $\limsup_{m \to +\infty} d(x,K_m) \leq \delta$.
\end{lemma}

\subsection{Convergence of Cartan projections}\label{cartan-proof}

Here we prove the part of Theorem \ref{joint} regarding the convergence of $\frac{1}{n}\kappa(S^n)$. Note first that for every $m\geq 1 $ and every $h \in S^m$ we have \begin{equation}\label{supsup}\frac{1}{m}\|\kappa(h)\|\leq C_G M_S,\end{equation}
where $M_S:=\sup_{s \in S} \|\kappa(s)\|$ as follows from Lemma \ref{Cartan.stability}.  We now verify that the conditions for convergence given in Lemma \ref{Hausdorff.convergence.lemma} are met. Let $\Gamma$ be the semigroup generated by $S$. Since $\Gamma$ is assumed Zariski-dense in $G$, we may apply Theorem \ref{AMS} and let $r=r(\Gamma)>0$ be the constant given by this theorem. Fix numbers $\delta>0$, $\eps \in (0,r)$ and  let $F=F(\eps)$ be the finite subset of  $\Gamma$ given also by Theorem \ref{AMS}. Given an integer $n \geq 1$ pick $x \in \frac{1}{n}\kappa(S^n)$ and $g\in S^n$ such that $x=\kappa(g)/n$. This gives $f \in F$ such that $gf$ is $(r,\eps)$-proximal in $G$. Then so will be every positive power $(gf)^m$, $m \geq 1$. Hence by Lemma \ref{loxodromy.implies.Cartan.close.to.Jordan} we have for all $m\geq1$
$$\|\lambda((gf)^m) - \kappa((gf)^m)\| \leq C_r.$$
However $\lambda((gf)^m) = m \lambda(gf)$ and by Lemma  \ref{Cartan.stability}, $\|\kappa(gf) - \kappa(g)\| \leq C_G\|\kappa(f)\|$. So for all $m\geq 1$,
$$\|\kappa(g) - \frac{1}{m}\kappa((gf)^m)\| \leq 2C_G\|\kappa(f)\| + C_r(1+\frac{1}{m}).$$
Note that $(gf)^m \in S^{m(n+n_f)}$, where $n_f$ is an integer such that $f \in S^{n_f}$. Now pick some element $h_0 \in S$ and consider for any $k \geq1$ the element  $g_k:=h_0^j(gf)^m$, where $m$ and $j$ are the quotient and remainder of the Euclidean division of $k$ by $n+n_f$. Clearly $g_k \in S^k$. By the triangle inequality:
$$\|x - \frac{1}{k}\kappa(g_k)\| \leq \frac{1}{n}\|\kappa(g) - \frac{1}{m}\kappa((gf)^m)\| + |\frac{1}{mn}-\frac{1}{k}|\|\kappa((gf)^m)\|+\frac{1}{k}\|\kappa(h_0^j(gf)^m)- \kappa((gf)^m)\|.$$
Making use of Lemma   \ref{Cartan.stability} again we get:
$$\limsup_{k \to+\infty} \|x - \frac{1}{k}\kappa(g_k)\| \leq \frac{1}{n}(2C_G\|\kappa(f)\| + C_r) + C_G M_S \frac{n_f}{n}.$$ The result follows by taking $n$ large enough.

\subsection{Convergence of Jordan projections}\label{Jordan.limit.subsection} We now turn to the second part of Theorem \ref{joint} regarding convergence of Jordan projections. Again we will show that the conditions of Lemma \ref{Hausdorff.convergence.lemma} are met.

We first make the following initial remark:

\begin{proposition} If $G$ is a connected reductive Lie group and $S \subset G$ a compact subset with $1 \in S$, then the sequence $\frac{1}{n}\lambda(S^n) \subset \a^+$ converges in Hausdorff metric.
\end{proposition}

\begin{proof} Indeed if $x= \frac{1}{n}\lambda(g)$ for some $g \in S^n$ and some $n$, then for every $k \in \N$ the fact that $1 \in S$ implies that $g^m \in S^k$, where $k=mn+j$ is the Euclidean division of $k$ by $n$. Moreover $x=\frac{1}{mn}\lambda(g^m)$, so $$\|x- \frac{1}{k}\lambda(g^m)\| = |\frac{1}{nm}-\frac{1}{k}| \|\lambda(g)\|$$ and the right hand side tends to $0$ as $k$ tends to $+\infty$.\end{proof}

The second remark is that convergence may fail in general if $1 \notin S$ and $G$ is not connected, as shown in the following example:

\bigskip

\begin{Example}\label{ex-non-conv}
Here we present an example of a set $S \subset \PSL_2(\R)$ for which the sequence $\frac{1}{n}\lambda(S^{n})$ does not converge. Let $\alpha>1$ and set $a:= \begin{pmatrix}
\alpha & \\
& \alpha^{-1}
\end{pmatrix}$, $u:=\begin{pmatrix}
& 1 \\
-1 &
\end{pmatrix}$ and take the subset $S:=\{a u, u\}$. Let $\lambda:\PSL_2(\R) \rightarrow [0, \infty[$ denote the Jordan projection, associating to an element the logarithm of its spectral radius. Then, we claim that $\frac{1}{2n}\lambda(S^{2n}) \underset{n \rightarrow \infty}{\longrightarrow} [0, \frac{1}{2}\log \alpha]$ and $\frac{1}{2n+1}\lambda(S^{2n+1})=\{0\}$ for each $n \geq 0$. Indeed every element in $S^{2n+1}$ has order $2$, whereas $\lambda(a^{n}) =n \log \alpha$ for every $n \geq 1$ and $a^n \in S^{2n}$.


Note that in this example, the Zariski-closure $G$ of the semigroup generated by $S$ is not $\PSL_2(\R)$ but the proper subgroup consisting of diagonal and antidiagonal elements. This is a non-connected reductive group. \hfill{$\diamond$}
\end{Example}

We are now ready for the proof of the convergence of $\frac{1}{n}\lambda(S^n)$ under the assumptions of Theorem \ref{joint}. Fix $\delta>0$. Given an integer $n\geq1$ and $g \in S^{n}$, set $x=\frac{1}{n}\lambda(g)$. By spectral radius formula, there exists $\ell \in \mathbb{N}$ such that $$\|\frac{1}{\ell}\kappa(g^{\ell})-\lambda(g)\|<\delta.$$
As before, let $\Gamma$ be semigroup generated by $S$ and let $r=r(\Gamma)>0$ the constant given by Theorem \ref{AMS}.  Fix some $\eps \in (0,r)$ and let $F$ be the finite subset of $\Gamma$ given by Theorem \ref{AMS}. By Theorem \ref{AMS}, there exists $f\in F$ such that $g^{\ell}f$ is $(r,\eps)$-proximal in $G$. By Lemma \ref{Cartan.stability} and $(\ref{supsup})$, writing $M_S:=\sup_{s \in S} \|\kappa(s)\|$ $$\|\kappa(g^{\ell}f)-\kappa(g^{\ell})\|\leq C_G n_f M_S,$$ where $n_{f} \in \N$ such that $f \in S^{n_{f}}$. Now by Lemma  \ref{loxodromy.implies.Cartan.close.to.Jordan}  for all $m\geq 1$
$$\|\lambda((g^{\ell}f)^{m})-\kappa((g^{\ell}f)^{m})\|\leq C_r.$$

Note that for each $i=1,\ldots,d_S$, the attracting points and repelling hyperplanes of $\rho_{i}((g^{\ell}f)^{m})$ are the same for all $m\geq 1$; denote them, respectively, by $v_{i}^{+}$ and $H_{i}^{<}$. Now, fix an arbitrary $h_0 \in S$. By Lemma \ref{dispersion.lemma} that there exists $\gamma \in \Gamma$ such that for all $i=1,\ldots,d_{S}$ and $j=1,\ldots,n\ell+n_{f}$, we have $\rho_{i}(\gamma)\rho_{i}(h_0^{j})v_{i}^{+}\notin H_{i}^{<}$.

\begin{lemma}\label{left.multiply.proximal.power.corollary}
Let $h\in G$ be $(r,\eps)$-proximal in $G$ and  $T\subset G$  a finite subset  such that for every $t\in T$ and every $i=1,\ldots,d_S$ we have $\rho_{i}(t)v^{+}_{\rho_{i}(h)} \notin H_{\rho_{i}(h)}^{<}$. Then, there exists $\hat{r}>0$ such that for every $\hat{\eps} \in (0,\hat{r})$ and every $t \in T$ and every $m \in \N$, $th^{m}$ is $(\hat{r},\hat{\eps})$-proximal in $G$, provided $m$ is larger than some $m_0(\hat{\eps}) \in \N$.
\end{lemma}

\begin{proof}
Recall that $h^{m}$ is $(r,\eps_{m})$-proximal in $G$ for some sequence $\eps_{m} \to 0$ as $m \to \infty$. 
Let $\hat{r}:=\frac{1}{3} \underset{t \in T} {\min} \underset{i=1,\ldots,d_S}{\min} d(\rho_{i}(t)v^{+}_{\rho_{i}(h)}, H_{\rho_{i}(h)}^{<})>0$. Set $Lip_G(T):=\max_{t \in T}Lip_G(t)$, where $Lip_G(t)$ is the maximum of the Lipschitz constants of $\rho_i(t)$ for $i=1,\ldots,d_S$.  Given $\hat{\eps}\in (0,\hat{r})$ let $m_0(\hat{\eps})\in \mathbb{N}$ be such that $h^{m}$ is $(r,\frac{\hat{\eps}}{2Lip_G(T)})$-proximal  for all $m \geq m_0(\hat{\eps})$. 

Using same notation as in the Tits proximality criterion (Lemma \ref{Tits.criterion.sublemma}), for each $i \in \{1,\ldots,d_{S}\}$, we thus have with $\delta:=\frac{\hat{\eps}}{2Lip_G(T)}$
$$
\rho_{i}(th^m)B_{\rho_{i}(h)}^{\delta} \subseteq \rho_{i}(t)b_{\rho_{i}(h)}^{\delta} \subseteq B(\rho_{i}(t)v_{\rho_{i}(h)}^{+},Lip_G(t)\delta).
$$ Therefore for each such $i$, the Tits proximality criterion applies to $\rho_{i}(th^m)$ with $x=\rho_{i}(t)v_{\rho_{i}(h)}^{+}$, $H=V_{\rho_{i}(h)}^{<}$, $r=\frac{\hat{r}}{2}$ and $\eta:=\frac{\hat{\eps}}{2}$, whence the lemma.
\end{proof}

Set $n_{\gamma} \in \mathbb{N}$ such that $\gamma \in S^{n_{\gamma}}$ and put $T:=\{\gamma h_0^{j}\, |\, j=1,\ldots,n\ell+n_{f}\}$. Applying Lemma \ref{left.multiply.proximal.power.corollary}, there exists $\hat{r}>0$ such that for every $\hat{r}\geq \hat{\eps}>0$, $\gamma h_0^{j}(g^{\ell}f)^{m}$ is $(\hat{r},\hat{\eps})$-proximal in $G$ for all $m \geq m_0({\hat{\eps}})$ for some $m_0({\hat{\eps}}) \in \mathbb{N}$ and $j=1,\ldots,n\ell+n_{f}$.

Using Lemma \ref{Cartan.stability} again, for all  $1 \leq j \leq n\ell+n_{f}$ and $m \geq 1$, we have $$
\|\kappa(\gamma h_0^{j}(g^{\ell}f)^{m})-\kappa((g^{\ell}f)^{m})\|\leq C_G \max_{t \in T} \|\kappa(t)\|.
$$
Now  by Lemma  \ref{loxodromy.implies.Cartan.close.to.Jordan} again we have 
$$
\|\lambda(\gamma h_0^{j}(g^{\ell}f)^{m})-\kappa(\gamma h_0^{j}(g^{\ell}f)^{m})\| \leq  C_{\hat{r}}$$ for all  $m \geq m_{\hat{\eps}}$ and all $1 \leq j \leq n\ell+n_{f}$.

Now given an integer $k \geq n_\gamma$, write $k-n_\gamma=m(n\ell+n_f) + j$ the Euclidean division of $k-n_\gamma$ by $n\ell+n_f$, and set $g_k=\gamma h_0^{j}(g^{\ell}f)^{m}$. Note that $g_k \in S^k$. Then combining the above inequalities we obtain:
$$
\limsup_{k \to +\infty} \| x - \frac{1}{k}\lambda(g_k)\| \leq \frac{1}{n} 3C_G n_f M_S + \frac{\delta}{n}.
$$ Taking $n$ large enough, this ends the proof.

\subsection{Coincidence of limits}
Continuing with the proof of Theorem \ref{joint} we now establish that the Hausdorff limits of the Cartan and the Jordan projections coincide. We denote the common limit by $J(S)$ and call it the \emph{joint spectrum} of $S$.

Write $J_c(S)$, resp. $J_j(S)$, for the limit of Cartan projections, resp. Jordan projections. The inclusion $J_j(S) \subset J_c(S)$ follows directly from the spectral radius formula, namely $\lim_{n \to \infty}\frac{1}{n}\kappa(g^{n})=\lambda(g)$ for all $g \in G$, see Corollary \ref{spectral.radius.formula.corollary}.

The other inclusion follows from the proof in \S \ref{cartan-proof}. Namely there is $n_0=n_0(S) \in \N$ such that for any $n \in \N$ and $g \in S^n$ there is $f \in S^i$ for some $i\leq n_0$ such that $\|\kappa(g)- \lambda(gf)\|$ is bounded by a constant depending on $S$ only. This follows by combining Theorem \ref{AMS} and  Lemmas \ref{Cartan.stability} and  \ref{loxodromy.implies.Cartan.close.to.Jordan} as done in \S \ref{cartan-proof}. The inclusion $J_c(S) \subset J_j(S)$ follows easily.

We end this section by an example, showing that when $G$ is not reductive, or equivalently if the group generated by $S$ in Theorem \ref{main1} is not irreducible, the limit of Cartan projections and Jordan projections may differ.

\begin{Example}\label{JneqK} Consider the following two matrices in $\GL_2(\R)$
$$a=\begin{pmatrix}
2 & 0\\
0 &1
\end{pmatrix},\textnormal{     } b=\begin{pmatrix}
2 & 1 \\
0 & 1
\end{pmatrix},$$
and set $S:=\{a,a^{-1},b\}$. The eigenvalues of an element $w\in S^n$ are $2^k$ and $1$, where $k=n_b+n_a$, where $n_b$ is the number of letters $b$ appearing in $w$, expressed as a word in $a^{\pm 1},b$, and $n_a$ the signed sum of the exponents of $a$ appearing in the word $w$. From this description, using the $(x,y)$-coordinates so that $y=\log |\det g|$ and $x=\log \lambda_1(g/\sqrt{|\det g|})$, we see that the Jordan projections $\frac{1}{n}\lambda(S^n)$ converge in Hausdorff distance to the union $J_j(S)$ of the two segments $\{(\frac{x}{2},x) \in \R^2 : x \in [0,\log 2]\}$ and $\{(\frac{x}{2},-x) \in \R^2 : x \in [0,\log 2]\}$.

On the other hand
$$b^na^{-n}=\begin{pmatrix}
1 & 2^n-1\\
0 &1
\end{pmatrix},$$ which implies that the Cartan projection of  the word $b^{n/2}a^{-n/2} \in S^n$ (say $n$ even) converges to the point $(\frac{1}{2}\log 2,0)$, which lies outside $J_j(S)$. A straightforward analysis using words of the form $a^kb^\ell a^{-\ell}$ yields that every point in the triangle $J_c(S):=\{(x,y)\in \R^2 : -2x \leq y \leq 2x, \, 0 \leq x \leq \frac{1}{2}\log 2\}$ is a limit point of $\frac{1}{n}\kappa(S^n)$ and in fact that $\frac{1}{n}\kappa(S^n)$ converges to $J_c(S)$.
\end{Example}

 \subsection{Realization by sequences} In this paragraph we establish the last claim of Theorem \ref{joint}, namely $(\ref{convergence-sequence})$. We shall be brief as the argument is very similar to the previous ones. Fix $x \in J(S)$. By the first part of Theorem \ref{joint}, we can find a sequence $\eps_n \to 0$ and elements $a_n \in S^n$ such that $\frac{1}{n}\kappa(a_n) = x+ \eps_n$. Next we find $r>\eps>0$ and a finite set $F \subset \Gamma$ such that Theorem \ref{AMS} holds together with Lemma \ref{dispersion.lemma} and Remark \ref{unif-rem}. And we fix some $h\in \Gamma$, which we assume $(r,\eps)$-proximal in $G$. We can then find, for each large enough $n$ (say $n\geq n_0$), elements $f_n,f_n' \in F$  such that the transformations $\{g_n\}_{n \geq n_0}$, where $g_n:=hf_na_nf_{n}'h$ form an $(r,\eps)$-Schottky family in $G$ whose attracting points and repelling hyperplanes are close to that of $h$. Denote by $|g_n|$ the number of letters of $S$ used in the expression $hf_na_nf_n'h \in S^{|g_n|}$. Note that $|g_n|=n+O(1)$. Next we choose a fast growing sequence of integers  $\ell_n \in \N$ so  that  $\sum_{i=0}^{n-1} i \ell_i = o(\ell_n)$. The sequence $b=(b_1,\ldots,b_n,\ldots) \in S^\N$ is now chosen so that the letters $b_i \in S$ are the letters read from left to right in the concatenation of the infinite word $g_{n_0}^{\ell_{n_0}}\cdot \ldots \cdot g_n^{\ell_n}\cdot \ldots$.  Thus every finite word $b_1\cdot \ldots \cdot b_k$ is of the form $g_{n_0}^{\ell_{n_0}}\cdot \ldots g_n^{\ell_n} \cdot g_{n+1}^{\ell}\cdot \bar{g}$, where $\ell \leq \ell_{n+1}$ and $\bar{g}$ is a prefix of $g_{n+1}$. Note then that 
\begin{equation}\label{leng} k=\sum_{i=n_0}^n |g_i| \ell_i + |g_{n+1}| \ell + O(n) = |g_n|\ell_n + |g_{n+1}|\ell + o(\ell_n).\end{equation}
Since $\|\kappa(g_{n+1})\|=O(n)$, by Lemma \ref{Cartan.stability} 
$$ \kappa(b_1\cdot \ldots \cdot b_k) = \kappa(g_{n_0}^{\ell_{n_0}}\cdot \ldots g_n^{\ell_n} \cdot g_{n+1}^{\ell}) + O(n)$$ while Lemma \ref{loxodromy.implies.Cartan.close.to.Jordan} and Proposition \ref{Best} in turn imply that 
$$ \kappa(g_{n_0}^{\ell_{n_0}}\cdot \ldots g_n^{\ell_n} \cdot g_{n+1}^{\ell})  =  \lambda(g_{n_0}^{\ell_{n_0}}\cdot \ldots g_n^{\ell_n} \cdot g_{n+1}^{\ell}) +O(1)= \sum_{i=n_0}^n \ell_i \lambda(g_i) + \ell\lambda(g_{n+1}) + O(n).$$
Finally since $\lambda(g_i)=\kappa(g_i) + O(1) = \kappa(a_i)+O(1) = i(x+\eps_i) + O(1)=|g_i|(x+\eps_i')$ where $\eps_i' \to 0$ as $i \to \infty$,
$$ \kappa(b_1\cdot \ldots \cdot b_k) = |g_n|\ell_n(x+\eps_n')+ |g_{n+1}|\ell  (x+\eps_{n+1}') + o(\ell_n),$$ which, in view of $(\ref{leng})$ implies that
$$\frac{1}{k}\kappa(b_1\cdot \ldots \cdot b_k) \to_{k \to +\infty} x.$$

\subsection{Proof of Theorems \ref{main1} and \ref{main2}} We view $\GL_d(\C)$ as a real reductive group (by restriction of scalars from $\C$ to $\R$) and $D_d \simeq (\R_+^*)^d$ be the group of diagonal matrices with real positive entries, which is a maximal split torus of $\GL_d(\C)$. Let $\G$ be the Zariski closure of the semigroup $\Gamma:=\langle S \rangle$ generated by $S$. Let $G$ be the connected component of identity in $\G(\R)$. The algebraic group $\G$ is reductive, for otherwise it would have a non-trivial unipotent radical, whose subspace of fixed vectors in $\C^d$ would be invariant under $\Gamma$, contradicting the assumption that $\Gamma$ acts irreducibly on $\C^d$. 

We now describe the relationship between the Cartan projection in $G$ and singular values of a matrix in $\GL_d(\C)$. Let $K$ be a maximal compact subgroup of $G$, $A$ a maximal real split torus associated to $K$ so that $G=KAK$ is the Cartan decomposition of $G$ as in \S \ref{reductive}. By \cite{Mostow1} there is a matrix $\gamma_0 \in \GL_d(\C)$ such that $\gamma_0G\gamma_0^{-1}$ is stable under adjoint, $\gamma_0K\gamma_0^{-1} \subset U_d(\C)$ and $\gamma_0A\gamma_0^{-1} \subset D_d$. At the Lie algebra level this embedding induces a linear embedding:
\begin{align*}\iota: \a &\to \R^d\\ x &\mapsto Ad(\gamma_0)x \end{align*}
such that $\gamma_0\exp(x)\gamma_0^{-1}=\exp(\iota(x)) \in D_d$ for $x \in \a:=\Lie(A)$. The positive Weyl chamber in $\Lie(D_d)=\R^d$ is $\mathcal{C}_d:=\{x=(x_1,\ldots,x_d) \in \R^d ; x_1\geq \ldots \geq x_d\}.$
To avoid confusion we denote by $\kappa_G(g)$ the Cartan projection of an element $g \in G$ and $\kappa(\gamma)$ the vector of logarithms of singular values of a matrix $\gamma \in \GL_d(\C)$. The folding map:
\begin{align*}\pi: \R^d &\to \mathcal{C}_d\\ x &\mapsto \kappa(\exp(x))\end{align*}
is the piecewise linear map which reorders the coordinates of $x$ in decreasing order.

For every $g \in G$, by definition of the Cartan projection we have $g \in K\exp(\kappa_G(g))K$. Therefore it follows from the above that 
\begin{equation}\label{kappaGGL}\kappa(\gamma_0g\gamma_0^{-1}) = \pi \circ \iota(\kappa_G(g)).\end{equation}


Now recall that $\G(\R)$ has finitely many connected components and hence $G$ has finite index in $\G(\R)$ and $\Gamma$ meets every coset of $G$ in $\G(\R)$. So there is $n_0 \in \N$ such that for every $g \in \Gamma$ there is $i \le n_0$ and $f \in S^i$ with $gf \in G$.  On the other hand, by Lemma \ref{Cartan.stability} applied to $\GL_d(\C)$, for every $g \in \GL_d(\C)$,
$$\|\kappa(\gamma_0g\gamma_0^{-1}) - \kappa(g)\|\leq C \|\kappa(\gamma_0)\|,$$
and 
$$\|\kappa(g)-\kappa(gf)\|\le C$$  for some constant $C=C(d,S)>0$. Hence, as before, we can apply Lemma \ref{Hausdorff.convergence.lemma} to conclude that $\frac{1}{n}\kappa(S^n) \subset \mathcal{C}_d$ converges in Hausdorff metric: indeed for some large $n$ pick $g \in S^n$, then $f\in S^i$, $i\leq n_0$ with $gf \in G$ and approximate $x=\frac{1}{n}\kappa(g)$ by $\pi \circ \iota ( \frac{1}{n}\kappa_G(gf))$. From the proof of the convergence of Cartan projections in \S \ref{cartan-proof}, we conclude that the sequence $\frac{1}{n}\kappa_G(S^n \cap G)$ converges to a compact set $J_G(S)$ in $\a^+$ and that $\frac{1}{n}\kappa(S^n) \subset \mathcal{C}_d$ converges to:
\begin{equation}J(S)=\pi \circ \iota(J_G(S)).\end{equation}
Note that the map $\pi \circ \iota:\a \to \mathcal{C}_d$ is continuous and piecewise linear with at most $d!$ pieces. This is the \emph{folding map} alluded to in the introduction. This ends the proof of Theorem \ref{main1}. 

 To complete the proof of Theorem \ref{main2} observe first that the inclusion $\frac{1}{n}\lambda(S^n) \subset J(S)$ follows immediately from the spectral radius formula $\lim_{n \to +\infty} \frac{1}{n}\kappa(g^n) = \lambda(g)$. For $(\ref{closure})$ simply note that if $x \in J_G(S) \in \a^+$ and $\eps>0$, there are arbitrarily large $n\geq 1$ and $g \in S^n \cap G$ with $\|x-\kappa(g)/n\|<\eps$. Then by the proof of the convergence of Cartan projections in \S \ref{cartan-proof}, there is some $f \in S^i$, $i\le n_0$ such that $\|\kappa(g)-\lambda(gf)\|$ is bounded independently of $n$. This yields $(\ref{closure})$ as $n$ is arbitrarily large. The assertion about the Benoist cone is clear (see also \S  \ref{Benoist.cone.subsection}). Theorem \ref{main2} is now proved.

\subsubsection{Convergence of Jordan projections in the non-connected case} \label{nonconnected} We end this section with a brief explanation of the remark following the statement of Theorem \ref{main2} pertaining to the convergence of $\frac{1}{n}\lambda(S^n)$. As shown in Example \ref{ex-non-conv} convergence may fail if the Zariski closure of the subgroup $\langle S \rangle$ is not connected. We show now that under the assumptions of Theorem \ref{main2} convergence does occur if $S$ is not contained in a single coset $aH$ of a subgroup $H\leq \langle S \rangle$ of finite index containing the derived group of $\langle S \rangle$. 

Let $G^0$ be the connected component of the Zariski closure of $\langle S \rangle$. Note that  $G^0 \cap \langle S \rangle$ has finite index in  $\langle S \rangle$ and so there is $k\in \N$ such that $g^k\in G^0$ for every $g \in \langle S \rangle$, and moreover $\lambda(g)=\frac{1}{k}\lambda(g^k)$. Now observe from the proof of the Hausdorff convergence of Jordan projections in Theorem \ref{joint} given in \S \ref{Jordan.limit.subsection} that the exact same argument will work in $G^0 \cap \langle S \rangle$ by choosing, in place of $h_0^j \in S^j$ different elements $h_j \in S^j \cap G^0$. The only thing one needs to check is that the set of integers $j$ such that $S^j \cap G^0 \neq \emptyset$ contains all large enough integers. This is easy to see under the assumption we made on $S$. Indeed working in the finite group $F= \langle S \rangle/(\langle S \rangle \cap G^0)$ things boil down to the following lemma:

\begin{lemma} Let $F$ be a finite group and $\Sigma\subset F$ a generating set. Then there is $N\in \N$ such that $\Sigma^N=F$ unless $\Sigma$ is contained in a single coset of some subgroup containing $[F,F]$. 
\end{lemma}

\begin{proof}Note that $n \mapsto |\Sigma^n|$ is  non-decreasing (since left multiplication is injective), so there is $N$ such that $|\Sigma^n|=|\Sigma^N|$ for all $n \geq N$. It is straightforward to check that $H:=\Sigma^N (\Sigma^N)^{-1}$ is a normal subgroup and $\Sigma^N$, hence $\Sigma$, is contained in a single coset of $H$. Since $\Sigma$ generates $F$, $F/H$ is a cyclic group, so $[F,F]\leq H$.
\end{proof}


Note that the conclusion of this lemma implies that $\Sigma^n=F$ for large enough $n$.

\section{Properties of the joint spectrum}\label{properties}

The aim of this section is to prove Theorems \ref{body}, \ref{realization}, \ref{dominated-theorem}, \ref{continuity} and \ref{word-balls} from the introduction.

\subsection{Joint spectrum and Benoist cone}\label{Benoist.cone.subsection} 
Recall (see \cite{benoist-asymptotic1}) that if $\Gamma$ is a Zariski dense semigroup in a reductive real algebraic group $G$, we may define its Benoist limit cone $\mathcal{BC}(\Gamma) \subset \a^+$ as the closed cone based at the origin spanned by all Jordan projections $\lambda(\gamma)$, $\gamma \in \Gamma$. This cone is invariant under conjugation. If $S$ is a compact generating set of $\Gamma$, then we clearly have:

\begin{proposition}\label{prop.js.generates.benoist} The joint spectrum $J(S)$ spans the Benoist cone $\mathcal{BC}(\Gamma)$, i.e. $\mathcal{BC}(\Gamma)=\{tx : t\geq 0, x \in J(S)\}$. 
\end{proposition}

Benoist showed that his cone is convex and has non-empty interior provided the semigroup is Zariski-dense in $G$. We will now show that these properties hold already at the level of the joint spectrum.

\subsection{Convexity} Here we prove the first assertion of Theorem \ref{body}, i.e. that $J(S)$ is convex in $\a^+$. The proof is closely related to Benoist's proof of the convexity of the limit cone $\mathcal{BC}(\Gamma)$ (see e.g. \cite[\S 6.2, 6.3]{benoist-quint-book}).

We keep the notation and terminology of Section \ref{section.preliminaries}. Let $G$ be a connected reductive group as before and $\Gamma$ be a Zariski dense semigroup in $G$. As before $\lambda: G \to \a^+$ denotes the Jordan projection. Recall that by Theorem \ref{AMS}, $\Gamma$ contains many $G$-proximal elements.

\begin{lemma}\label{trace.angle.lemma}
Let $g,h$ be two $G$-proximal elements of $\Gamma$. There exist $u \in \Gamma$ and a constant $M_{1} \geq 0$ such that for all $k \geq 0$ 
$$
||\lambda(g^{k}uh^{k}u)-k\lambda(g)-k\lambda(h)||\leq M_{1}.
$$ 
\end{lemma}

\begin{proof}
Recall from \S \ref{rep} that $G$ has $d$ distinguished irreducible representations $\rho_1,\ldots,\rho_d$ with $\rho_i$ being a character of $G/[G,G]$ when $i> d_S$.  We need to check that for some $u \in \Gamma$, $\overline{\chi}_{\rho_i}(\lambda(g^{k}uh^{k}u))-k\overline{\chi}_{\rho_i}(\lambda(g))-k\overline{\chi}_{\rho_i}(\lambda(h))$ is uniformly bounded for all $k \geq 0$ for each $i$. If $i>d_S$ this quantity is independent of $k$ and equals $2\overline{\chi}_{\rho_i}(\lambda(u))$. When $i \le d_S$, by $(\ref{param})$ we have $$\overline{\chi}_{\rho_i}(\lambda(g^{k}uh^{k}u))-k\overline{\chi}_{\rho_i}(\lambda(g))-k\overline{\chi}_{\rho_i}(\lambda(h))=\log \frac{\lambda_{1}(\rho_{i}(g^{k}uh^{k}u))}{\lambda_{1}(\rho_{i}(g))^{k}\lambda_{1}(\rho_{i}(h))^{k}}.$$
However, since $\rho_i(g)$ is proximal $\frac{\rho_i(g)^k}{\lambda_1(\rho_i(g))^k}$ tends to $\pi_{\rho_i(g)}$, a rank $1$ matrix with image $v_{\rho_i(g)}^+$ and kernel $H_{\rho_i(g)}^{<}$. And the same holds for $h$. So the above quantity converges to $\log |\lambda_1(\pi_{\rho_{i}(g)}\rho_{i}(u)\pi_{\rho_{i}(h)}\rho_{i}(u))|$ as $k$ goes to $\infty$. Therefore, we only need to find some $u\in \Gamma$ for which the rank $1$ map $\pi_{\rho_{i}(g)}\rho_{i}(u)\pi_{\rho_{i}(h)}\rho_{i}(u)$ is not nilpotent for each $i=1,\ldots,d_S$. The existence of such a $u \in \Gamma$ follows from Lemma \ref{dispersion.lemma}.
\end{proof}

We are now in a position to complete the proof of the convexity statement in Theorem \ref{body}: let $x,y \in J(S)$. We need to show that $\frac{1}{2}(x+y)\in J(S)$. This will follow easily from the above lemma, once we approximate $x$ and $y$ by the Jordan projection of suitable $G$-proximal elements. Note that by Theorem \ref{joint} for every $\delta>0$ and all $n \in \mathbb{N}$ large enough, we can find $g \in S^n$ and $h\in S^n$ such that $\|x-\frac{\kappa(g)}{n}\|\leq \delta$ and  $\|y-\frac{\kappa(h)}{n}\|\leq \delta$. And by Theorem \ref{AMS} there is a finite set  $F\subset \Gamma$ depending only on $\Gamma$ such that $gf_{g}$ and $hf_{h}$ are $G$-proximal for some $f_{g}, f_{h} \in F$. Furthermore, using Lemmas \ref{Cartan.stability} and \ref{loxodromy.implies.Cartan.close.to.Jordan}, we have
\begin{equation}\label{eq3}
\|\lambda(gf_{g})-\kappa(g)\|\leq M_{0}
\end{equation} for some $M_0$ depending on $\Gamma$ only, and similarly for $hf_{h}$ and $h$. On the other hand, Lemma \ref{trace.angle.lemma} implies that for some $u \in \Gamma$, there exists a constant $M_{1}\geq 0$, such that for all $k \geq 1$, we have
\begin{equation}\label{eq4}
\|\lambda((gf_{g})^{k}u(hf_{h})^{k}u)-k\lambda(gf_{g})-k\lambda(hf_{h})\|\leq M_{1}
\end{equation}
Now let $n_{f_{g}},n_{f_{h}},n_{u} \in \mathbb{N}$ be such that $f_{g} \in S^{n_{f_{g}}}$, $f_{h} \in S^{n_{f_{h}}}$ and $u \in S^{n_{u}}$ so that $(gf_{g})^{k}u(hf_{h})^{k}u \in S^{2nk+k(n_{f_{g}}+n_{f_{h}})+2n_{u}}$. Combining (\ref{eq3}) and (\ref{eq4}), it follows from the triangle inequality that if $n$ is large enough, we have
$$\|\frac{\lambda((gf_{g})^{k}u(hf_{h})^{k}u)}{2nk+k(n_{f_{g}}+n_{f_{h}})+2n_{u}}-\frac{x+y}{2}\| \leq 3\delta
$$
for all large enough $k$. This ends the proof.

\begin{remark} We note that the convexity of the joint spectrum can also be deduced from the probabilistic characterization as the essential support of the rate function of an i.i.d.\ random walk supported on $S$. See \cite{Sert.LDP} and \S \ref{subsec.LDP} below.
\end{remark}

\subsection{Dominated families and Schottky families}

In this paragraph we clarify the relationship between the concept of Schottky family introduced in Definition \ref{defnarepsSchottky1} and the concept of dominated family. The latter notion is widely used in dynamics, see in particular \cite{bonatti-diaz-viana, yoccoz, avila-bochi-yoccoz, bochi-gourmelon}. It will be crucial for the continuity properties of the joint spectrum established in the next subsection. 

\begin{definition}[Dominated family] A relatively compact subset $S \subset \GL_n(\R)$ is said to be $1$-dominated if there is $\eps>0$ such that for every large enough $n \in \N$ and every $g \in S^n$,
$$\frac{a_2(g)}{a_1(g)} \leq (1-\eps)^n.$$
\end{definition}

There is also a notion of $k$-dominated family, requiring that each ratio $a_{k+1}(g)/a_k(g)$ be bounded by $(1-\eps)^n$. An important result of Bochi and Gourmelon \cite{bochi-gourmelon} (see also  \cite{bochi-sambarino-potrie} and for a recent quantitative version \cite{q-t-z}) asserts that every $1$-dominated compact set $S$ preserves a salient projective cone. This is key to the following

\begin{proposition}[Schottky families versus $1$-dominated families]\label{prop.domination.and.schottky} Let $S$ be a relatively compact subset of $\GL_d(\R)$.
\begin{enumerate}
\item (From Schottky to domination) Suppose that there exist $n \in \mathbb{N}$ and constants $r\geq \eps>0$ such that $S^n$ is an $(r,\eps)$-Schottky family. Then $S$ is 1-dominated.
\item (From domination to Schottky) Suppose $S$ is $1$-dominated. Then there exists $r>0$ such that for every $\eps\in (0,r)$, there is $n_{\eps}\in \mathbb{N}$ such that the semigroup $\cup_{n \geq n_\eps}S^n$ is an $(r,\eps)$-Schottky family.
\end{enumerate}
\end{proposition}

Recall the notation from \S \ref{dynamics}. For a proximal $g \in \Endo(\mathbb{R}^d)$, we denote respectively by $v_g^+$ and $H_g^<$, its attracting line and the supplementary $g$-invariant hyperplane in $\mathbb{P}(\mathbb{R}^d)$. Similarly, we set $B_g^\eta:=\{x \in \mathbb{P}(\mathbb{R}^d)\,|\, d(x,H_g^<)\geq \eta\}$ and $b_g^{\eta}:=\{x \in \mathbb{P}(\mathbb{R}^d)\,|\,d(x,v_g^+)\leq \eta\}$. 
We will need the following simple lemma from \cite[Lemmas 3.3, 3.4]{breuillard-gelander}.

\begin{lemma}\label{bg-lemma} Let $g\in \GL_d(\R)$ be such that $g$ is $\eps$-Lipschitz on some open subset of $\P(\R^d)$ for some $\eps\in (0,1)$. Then $\frac{a_2(g)}{a_1(g)} \leq \eps/\sqrt{1-\eps^2}$.
\end{lemma}


\begin{proof}[Proof of Proposition \ref{prop.domination.and.schottky}] Proof of $(i)$. Using Lemma \ref{Cartan.stability} we easily see that if $S^n$ is 1-dominated for some integer $n \geq 1$, then $S$ itself is $1$-dominated. Therefore, we can assume that $S$ itself is an $(r,\eps)$-Schottky family. Then for any $n \geq 1$ every $g \in S^n$ is $\eps^n$-Lipschitz on some neighborhood of $v_s^+$ for any $s \in S$. In view of Lemma \ref{bg-lemma} this implies that $\frac{a_2(g)}{a_1(g)}\leq 2\eps^n$ as desired.
\bigskip

Proof of $(ii)$. For the other direction, we rely on a crucial result of Bochi-Gourmelon according to which if $S$ is $1$-dominated, then there exists a \emph{dominated splitting} for the full shift. Namely by \cite[Theorem A]{bochi-gourmelon} there exist continuous maps $E^u:S^\mathbb{Z} \to \mathbb{P}(\mathbb{R}^d)$ and $E^s:S^\mathbb{Z} \to \Gr(d-1,\mathbb{R}^d)$, where $\Gr(d-1,\mathbb{R}^d)$ denotes the grassmannian of projective hyperplanes in $\P(\R^d)$, such that $E^u(Tx)=x_0E^u(x)$, $E^s(Tx)=x_0E^s(x)$ and $E^u(x)$ is disjoint from $E^s(x)$ for all $x \in S^\Z$, and such that
\begin{equation}\label{dom}\|g_n{|E^s(x)}\| \leq C \tau^n \|g_n{|E^u(x)}\|\end{equation}
for all $n\geq 0$ and for some constants $C>0$ and $\tau \in (0,1)$, where $g_n:=x_{n-1}\cdot\ldots\cdot x_0$. As noted in \cite{bochi-gourmelon}, by construction, $E^s(x)$  depends only on the future $x_0,x_1,\ldots$, while $E^u(x)$ depends only on the past $x_{-1},x_{-2},\ldots$.  Consequently, as noted there too, $E^s(x)$ is disjoint from $E^u(y)$ for all $x,y \in S^\Z$. Indeed $E^s(x)=E^s(z)$ and $E^u(y)=E^u(z)$, where $z\in S^\Z$ is defined to have the same future as $y$ and the same past as $x$. By continuity and compactness this means that there is $\eta>0$ such that $d(E^u(x),E^s(y))>\eta$ for all $x,y \in S^\Z$.

Now note that every element of a $1$-dominated family in $\GL_d(\mathbb{R}) $ is necessarily proximal transformation of $\R^d$. This follows immediately from the spectral radius formula applied to the exterior power representation $\wedge^2$ of $\GL_d(\mathbb{R})$. If $g\in S^n$, we may write $g=x_{n-1}\ldots x_0$ for some $x_i$'s in $S$. Say $x \in S^\Z$ is defined by $x_i=x_r$, where $r\in \{0,\ldots,n-1\}$ is the class of $i$ modulo $n$. Then $T^nx=x$ and we conclude that $E^u(x)$ and $E^s(x)$ are $g$-invariant and hence sums of generalized eigenspaces. In view of $(\ref{dom})$ and since $g$ is proximal, we conclude that $E^u(x)=v_g^+$ and $E^s(x)=H_g^{<}$. Therefore $d(v_g^+,H_h^<)\geq \eta>0$ for all $g,h \in \cup_{n \geq 1} S^n$. 

To get a Schottky family it remains to find $r>0$ and for every $\eps>0$ an integer $n_\eps$ such that every $g \in S^n$ for $n\geq n_\eps$ is $(r,\eps)$-proximal. Pick $r=\eta/12$ and assume that $v$ is a unit vector in $\R^d$ with $d([v],H^{<}_g)>\eps$. Write $v=u+w$, where $w \in H_g^{<}$ and $u \in  v_g^+$. Then 
$$d([v],H^{<}_g) = \inf_{x \in H^{<}_g} \frac{\|v \wedge x\|}{\|v\|\|x\|} = \|u\| d(v^+_g,H_g^{<}),$$
and hence $\eps\le \|u\|\le 1/r$ and $\|w\| \leq 1+1/r$. On the other hand
 $$d(gv,v^+_g)=\frac{\|gw \wedge u\|}{\|gv\|\|u\|} \leq \frac{\|gw\|}{\|gv\|} \leq \frac{\|gw\|}{\|gu\|-\|gw\|}$$ 
but $\|gw\|\leq \|g|H_g^{<}\|\|w\|$ and $\|gu\|= \|g|v^+_g\|\|u\|$, hence by $(\ref{dom})$
$$\frac{\|gw\|}{\|gu\|} \leq C \tau^n \frac{\|w\|}{\|u\|} \leq C \tau^n \frac{1}{\eps}(1+\frac{1}{r}).$$
Combining the last two displayed equations, we see that if $n$ is large enough $d(g[v],v_g^+) \leq \eps$ for all $[v] \in \P(\R^d)$ with $d([v],H_g^{<})>\eps$, confirming that $g$ is $(r,\eps)$-proximal.
\end{proof}

We will see in the next paragraph that Proposition \ref{prop.domination.and.schottky} is useful for deducing various continuity results. The most immediate one is the following well-known fact \cite[B.1]{bonatti-diaz-viana}.

\begin{corollary}
Let $\Sigma$ be a compact $k$-dominated subset of $\GL_d(\mathbb{R})$. Every compact $\Sigma'$ sufficiently close to $\Sigma$ is also $k$-dominated. 
\end{corollary}
\begin{proof}
By passing to $k$-th exterior power representation $\wedge^k$, $\Sigma$ is $1$-dominated. The desired result is then immediate by definition of an $(r,\eps)$-Schottky family and the previous proposition.
\end{proof}

Let now $G$ be a connected reductive real Lie group of rank $d$. Recall that in \S \ref{reductive} we have defined $d$ distinguished linear representations of $G$; $\rho_1,\ldots, \rho_d$ with $\rho_i$ trivial on $[G,G]$ if and only if $i>d_S$, where $d_S$ is the semisimple rank of $G$.

\begin{definition}[$G$-dominated family]\label{Gdom} A compact subset  $S\subset G$ is called $G$-dominated if $\rho_i(S)$ is $1$-dominated for every $i=1,\ldots,d_S$.
\end{definition}

It is easy to see that this definition is independent of the choice of the distinguished representations. In \S \ref{reductive} we also defined the notion of $(r,\eps)$-Schottky family in $G$. By virtue of Proposition \ref{prop.domination.and.schottky} the following is clear:

\begin{proposition} Let $S\subset G$ be a compact subset. 
\begin{enumerate}
\item If there is $n \in N$ and $r \geq \eps>0$ such that $S^n$ is an $(r,\eps)$-Schottky family in $G$, then $S$ is $G$-dominated.
\item if $S$ is $G$-dominated, then there is $r>0$ such that for every $\eps>0$ there is $n_\eps$ such that $\cup_{n \geq n_\eps} S^n$ is an $(r,\eps)$-Schottky family in $G$.
\end{enumerate}
\end{proposition}

In view of Theorem \ref{joint} we also have the following charaterization of $G$-domination:

\begin{proposition} Let $S\subset G$ be a compact subset generating a Zariski-dense subgroup. Then $S$ is $G$-dominated if and only if the joint spectrum $J(S)$ belongs to the open Weyl chamber $\a^{++}$.
\end{proposition}

\begin{proof}By Theorem \ref{joint}, $\frac{\kappa(S^n)}{n}$ converges to $J(S)$ in Hausdorff metric as $n$ tends to infinity. In view of $(\ref{param})$, if $J(S) \subset \a^{++}$ then there is $\eps>0$ and $n_0\geq 1$ such that for each $i=1,\ldots,d_S$, we have $\frac{1}{n}\log \frac{a_1(\rho_i(g))}{a_2(\rho_i(g))} \geq \eps$ for all $g \in S^n$. So $\rho_i(S)$ is $1$-dominated, and conversely.
\end{proof}

In fact if $S$ is not assumed to be Zariski-dense, but is $G$-dominated, then the joint spectrum is well-defined and we have Hausdorff convergence of both $\frac{1}{n}\kappa(S^n)$ and $\frac{1}{n}\lambda(S^n)$ towards the same limit. We recorded this statement in the introduction as Theorem \ref{dominated-theorem}.

\begin{proof}[Proof of Theorem \ref{dominated-theorem}]
By Proposition \ref{prop.domination.and.schottky}, we may assume that $\sup_{n \geq n_0} S^n$ is an $(r,\eps)$-Schottky family in $G$ for some $r>\eps>0$. The result then follows from Proposition \ref{Best} and Lemma \ref{loxodromy.implies.Cartan.close.to.Jordan} exactly as in the proof of Theorem \ref{joint}.
\end{proof}

\subsection{Continuity properties of the joint spectrum}\label{subsection.continuity.properties}

\subsubsection{Continuity under domination}

Here we give a direct proof of Theorem \ref{continuity}. In Section 5 we will give an independent proof of a stronger result, namely the continuity of the Lyapunov spectrum, see Proposition \ref{prop.lyapunov.continuous}.

\begin{proof}[Proof of Theorem \ref{continuity}] In view of Proposition \ref{prop.domination.and.schottky} we may assume that $S$ is an $(r,\eps)$-Schottky family in $G$. Set $r_1=r, \eps_1=\eps$. By continuity of eigenvalues and eigendirections, given $r_1>r_2>\eps_2>\eps_1>0$, there exists $\delta_1>0$, such that if $d(S,S')<\delta_1$, then $S'$ is an $(r_2,\eps_2)$-Schottky family in $G$. Let $n_0 \in \mathbb{N}$ be such that $\frac{C_{r_1}+C_{r_2}+1}{n_0}<\eps$, where $C_r$ is the constant from Lemma \ref{loxodromy.implies.Cartan.close.to.Jordan}. By continuity, there exists $\delta_2>0$ such that if $d(S,S')<\delta_2$, then $\|\lambda(g_1)-\lambda(g_2)\|<1$ for all $g_1 \in S^{n_0}$ and $g_2 \in (S')^{n_0}$. It follows by Proposition \ref{Best} that given $m=n_0q \in \mathbb{N}$ and elements $g_1^{(1)}\cdot \ldots \cdot g^{(1)}_m \in S^m$ and $g_1^{(2)}\cdot\ldots\cdot g^{(2)}_m \in (S')^m$, denoting by $h^{i}_{j}=g^{(i)}_{jn_{0}+1}\cdot\ldots\cdot g^{(i)}_{(j+1)n_{0}}$ for $i=1,2$ and $j=0,\ldots,q-1$, we have 
$$
||\lambda(g_1^{(i)}\cdot\ldots\cdot g^{(i)}_m)-\sum_{j=1}^{q} \lambda(h_j^{(i)})|| \leq q C_{r_i}
$$ for $i=1,2$.

Therefore, if $d(S,S')<\min \{\delta_1,\delta_2\}=:\delta(S)$, for any $m=n_0q \in \mathbb{N}$, we have
$$
\|\frac{1}{m} \lambda(g_1^{(i)}\cdot\ldots\cdot g^{(i)}_m) -\frac{1}{m} \lambda(g_1^{(i)}\cdot\ldots \cdot g^{(i)}_m) \| < \frac{q(C_{r_1}+C_{r_2})+q}{m}=\frac{C_{r_1}+C_{r_2}+1}{n_0}\leq \eps
$$
Since $\frac{1}{n}\lambda(S^n)$ converges to $J(S)$ for the Hausdorff metric by Theorem \ref{joint}, the desired continuity of $J(S)$ follows.
\end{proof}

\subsubsection{Discontinuities of the joint spectrum}\label{discon}

The joint spectrum is not continuous in general. The continuity properties of the lower joint spectral radius $R_{sub}$ have been studied studied by Bochi-Morris in \cite{bochi-morris}, who gave a description of the discontinuities of $R_{sub}$. However, all the examples of discontinuity points of $R_{sub}$ in that work occur at families stabilizing a finite union of proper subspaces. The following simple example inspired from Bochi-Morris \cite[\S 1.4]{bochi-morris} confirms their insight on the existence of a strongly irreducible set which is a discontinuity point for $R_{sub}$ (and consequently, for the joint spectrum). A variation of this example also yields an irreducible set on which $R_{sub}$ is continuous but the joint spectrum is not.

\begin{Example}\label{ex.discon}
Let $r_n$ be the rotation matrix in $\SO(2,\mathbb{R})$ of angle $\frac{\pi}{2n}$. Let $a,b \in \SL(2,\mathbb{R})$ be two hyperbolic elements with equal spectral radii $\lambda_1(a)=\lambda_1(b)>1$. Suppose that their fixed points on the boundary $\partial D$ of the Poincar\'{e} disc are cyclically ordered as $[x_b^{-},x_b^{+},x_a^{+},x_a^{-}]$, where $x_i^{+}$ and $x_i^{-}$ stand, respectively, for the attracting and repelling fixed points of the element $i \in \{a,b\}$. Assume further that $b^{-1}=r_1ar_1$. Finally, let $\alpha>1$ be a real number. Consider the subsets $S_0=\{\alpha \id,a,b\} $ and for $n \geq 1$, $S_n=\{\alpha r_n,a,b\}$ of $\GL_2(\mathbb{R})$. 

The joint spectrum $J(S_0)$ of $S_0$ is the red region in Figure \ref{fig.discontinuity} (for a discussion concerning the precise values in this figure see Subsection \ref{subsec.ingredients.hyp.geo}). Similarly, for every $n \geq 1$, the joint spectrum $J(S_n)$ of $S_n$ is the union of the dark gray and red regions, which is obtained as the intersection of the convex hull of the union of $J(S)$ and its reflection (light-red colored area) along the wall of the Weyl chamber corresponding to $y$-axis, with $\mathfrak{a}^+$.
On the other hand, clearly, $S_n \to S_0$ in Hausdorff metric, as $n \to +\infty$, showing the discontinuity claim. 

In this example, by our choice of the coordinate system (namely $\log \det(g)$ on $y$-axis and $\log \lambda_1(\frac{g}{\det(g)^{\frac{1}{2}}})$ on the $x$-axis), the lower joint spectral radius $R_{sub}$ of $S_k$ $k \in \mathbb{N}$ is given by $\exp(\min_{x \in J(S_k)}\ell(x))$, where $\ell$ is the linear form given by $(x,y) \mapsto x+\frac{y}{2}$. Therefore $S_0$ is also a discontinuity point for $R_{sub}$.

\begin{figure}[H]
\begin{flushleft}
\begin{minipage}{0.3\textwidth}
\begin{flushleft}
\begin{tikzpicture}[scale=0.90]

\coordinate (A) at (1.7,0) {};
\coordinate (B) at (4,0) {};
\coordinate (C) at (0,2) {};
\coordinate (D) at (-4,0) {};

\draw (4,1.5) node {$\mathfrak{a}^+$};
\draw (0.54,2.12) node {$\log \alpha$};
\draw (1.4,-0.3) node {$\log \lambda_1(a)$};
\draw (3.7,-0.4) node {$\frac{\log  \lambda_1(ab)}{2} $};

\draw[dashed] (0,2) -- (-1.7,0);
\draw[dashed, ->] (0,0) -- (-4.7,0);

\filldraw[draw=black, fill=red]
(A) -- (C) -- (B) -- cycle;

\filldraw[draw=black, fill=gray]
(0,0) -- (A) -- (C) -- cycle;

\filldraw[draw=black, dashed, fill=gray!38]
(0,0) -- (D) -- (C);

\filldraw[dashed,fill=red!45]
(-1.7,0) -- (C) -- (D) -- cycle;

\draw (2,0.5) node {$J(S_0)$};

\draw[semithick, ->] (0,0) -- (4.7,0);
\draw[semithick, ->] (0,-2.2) -- (0,2.8);
\draw[dashed] (C) -- (-1.7,0);
\draw[fill] (A) circle [radius=0.05];
\draw[fill] (C) circle [radius=0.05];
\draw[fill] (B) circle [radius=0.05];

\end{tikzpicture}
\end{flushleft}
\caption{} \label{fig.discontinuity}
\end{minipage}
\begin{minipage}{0.33\textwidth}
${}$
\end{minipage}
\begin{minipage}{0.35\textwidth}
\vspace{-0.3cm}
On the other hand if we take instead $\alpha \in (0,1)$, one gets the same discontinuity phenomenon for the joint spectrum and get the same picture as in Figure \ref{fig.discontinuity} but reflected around the $x$-axis. However, in that case, $R_{sub}(S_n)$ is constant and equal to $\alpha$ for every $n \in \mathbb{N}$. In fact, it follows from \cite[Corollary 1.7, Corollary 1.9]{bochi-morris} that $S_0$ is then a continuity point for $R_{sub}$.
\end{minipage}
\end{flushleft}
\end{figure}
\end{Example}

The underlying mechanism in the previous example can be generalized to give a description of discontinuity points of joint spectrum, see items 4. and 5. in Section \ref{sec.further}.

\subsection{Prescribed joint spectrum}\label{subsection.prescribed.joint.spectrum}

In this paragraph we prove Theorem \ref{realization} from  Introduction showing that every convex body $K$ in the (closed) Weyl chamber can be realized as the joint spectrum of a set $S$ generating a Zariski-dense semigroup.
The idea is to start with a subset of simultaneously diagonalizable elements in $A$ whose joint spectrum is already $K$, select a point in the interior of $K$ and add a finite number of elements in a small neighborhood of that point in order to generate a Zariski-dense subgroup. We have to make sure that this  procedure does not alter the joint spectrum. To effect this strategy we need to control the spectral radius of a product, much like in Proposition \ref{spectralcontrolproximal}, except that our elements need not be $(r,\eps)$-proximal for small $\eps$. The next two lemmas make up for this lack of contraction.

\begin{lemma}\label{diagonal} Let $d \in \N$ and $\eta\in(0,1)$. For every small enough $\eps>0$ there is a neighborhood $V_\eps$ of the identity in $\GL_d(\R)$ such that for every $u \in V_\eps$ and every diagonal matrix $g \in \GL_d(\R)$, $g=\diag(\mu_1,\ldots,\mu_d)$ such that $\max_{i>1}|\mu_i(g)|\leq (1-2\eta)|\mu_1(g)|$, the element $gu$ preserves the ball $B([e_1],\eps)$ around $[e_1]\in \P(V)$, acts as a $(1-\eta)$-Lipschitz map on this ball and satisfies $\lambda_2(gu)\leq (1-\eta) \lambda_1(gu)$ and $|\log \lambda_1(gu) - \log \lambda_1(g)|\leq \eps$.
\end{lemma}

\begin{proof} First note that if $ \frac{1-2\eta}{1-\eps^2} < 1-\frac{3}{2}\eta$, then for all diagonal matrices satisfying  $\max_{i>1}|\mu_i(g)|\leq (1-2\eta)|\mu_1(g)|$ we have:
$$d(g[x],g[y]) \leq (1-\frac{3}{2}\eta) d([x],[y])$$
for all $[x],[y] \in B([e_1],\eps)$. To see this, simply compute using the formula for the standard distance:  note that $\|gx \wedge gy\| \leq |\mu_1(g)\mu_2(g)|\|x \wedge y\|$ and $\|gx\| \geq |\mu_1(g)||x_1|$ and $\|gy\| \geq |\mu_1(g)||y_1|$ and observe that $d([x],[e_1]) \leq \eps$ implies $|x_1|^2 \geq (1-\eps^2)\|x\|^2$ and similarly for $y$.

Then pick $V_\eps$ small enough so that every $u \in V_\eps$ is $(1+\eps)$-Lipschitz on all of $\P(\R^d)$ and $d(u[x],[x]) \leq \frac{3}{2}\eta \eps$ for all $[x] \in \P(\R^d)$. If $(1+\eps)(1-\frac{3}{2}\eta) < 1-\eta$ it will follow that all $gu$, $g$ diagonal as above, $u \in V_\eps$, will preserve $B([e_1],\eps)$  and will be $(1-\eta)$-Lipschitz on it. Choosing $V_\eps$ even smaller if necessary the last two inequalities will hold, by compactness and continuity of matrix eigenvalues.
\end{proof}

\begin{lemma}\label{produ} Let $\eta \in (0,1)$. For every $\eps>0$ there is $d_{\eta,\eps}>0$ with $\lim_{\eps \to 0} d_{\eta,\eps} =1$ such that for any point  $p \in \P(V)$ and any choice of $\ell \geq 1$ elements $g_1,\ldots,g_\ell$ in $\GL(V)$ which preserve the ball $B(p,\eps)$ in $\P(V)$ and act as $(1-\eta)$-Lipschitz maps on this ball, the following holds:
$$ d_{\eta,\eps}^{-\ell} \leq \frac{\lambda_{1}(g_\ell\ldots g_1)}{\lambda_1(g_\ell) \ldots\lambda_1(g_1)}\leq d_{\eta,\eps}^\ell.$$
\end{lemma}

\begin{proof}First we claim that there is $d_{\eta,\eps}>0$ as desired such that for every point $p$ and  ball $B(p,\eps)$ in $\P(V)$ and for every $g \in \GL(V)$ which preserves the ball and is $(1-\eta)$-Lipschitz on it, we have for all $x \in V\setminus \{0\}$ with $[x] \in B(p,\eps)$
\begin{equation}\label{eq.cone.estimate1}
d_{\eta,\eps}^{-1} \leq \frac{\|gx\|}{\|x\|.|\lambda_1(g)|} \leq d_{\eta,\eps}
\end{equation}
Indeed note that the assumption implies that $g$ is proximal with attracting direction $[v_g] \in B(p,\eps)$, and thus $\|gv_g\|=\lambda_1(g)\|v_g\|$. Without loss of generality one may assume that $\|g\|=1$. The desired inequality then follows by compactness by observing that any limit of the above expression must equal $1$.

Now let $h=g_\ell\ldots g_1$. Note that $h$ preserves $B(p,\eps)$ and acts as a $(1-\eta)^\ell$-Lipschitz map on it. Since $\ell\geq 1$ this implies that $h$ admits a unique fixed point $[v_h] \in B(p,\eps)$, for some $v_h \in V\setminus\{0\}$, and that $h$ is proximal with attracting direction $[v_h]$. In particular $\|hv_h\|=\lambda_1(h)\|v_h\|$. Now the lemma follows by applying the estimate $(\ref{eq.cone.estimate1})$ to each $g_i$ with $x=g_{i-1}\ldots g_1v_h$ and taking the product.
\end{proof}

\begin{proof}[Proof of Theorem \ref{realization}]
Let $w_0$ be a vector lying in the interior of $K$ in $\a^+$ and $a_0:=\exp(w_0)$. We will show that the set $S:=\exp(K) \cup \{a_0F\}$, for some finite set $F\subset V_\eps$, where $V_\eps$ is a sufficiently small neighborhood of $1$ in $G$, satisfies the conclusion of the theorem. Clearly $K \subset J(S)$, so we need to show the reverse inclusion. 

If $g=s_1\ldots s_n \in S^n$, we let $\ell$ be the number of indices $i \in \{1,\ldots,n\}$ such that $s_i \in a_0V_\eps$. So $g$ can be written $g=a_1a_0v_0\cdot\ldots\cdot a_\ell a_0v_\ell  a_{\ell+1}$, where $a_j \in \exp(\a^+)$ for each $j=1,\ldots, \ell$ and $v_j \in V_\eps$. We may assume $\ell \geq 1$. Since $\lambda(g)=\lambda(a_{\ell+1}ga_{\ell+1}^{-1})$, without loss of generality we may assume that $a_{\ell+1}=1$. Set $h_j=a_ja_0$ and  $g_j=h_jv_j$ for all $j=1,\ldots,\ell$.

Now consider the distinguished representations $\rho_i$, $i=1,\ldots, d_S$ defined in \S \ref{rep}. Let $p_i$ be the highest weight eigendirection in $\P(V_{\rho_i})$. Note that for each $j$ 
$$\frac{\lambda_2(\rho_i(h_j))}{\lambda_1(\rho_i(h_j))} = \frac{\lambda_2(\rho_i(a_j))}{\lambda_1(\rho_i(a_j))} \frac{\lambda_2(\rho_i(a_0))}{\lambda_1(\rho_i(a_0))}  \leq \frac{\lambda_2(\rho_i(a_0))}{\lambda_1(\rho_i(a_0))}   \leq 1-2\eta$$
for some $\eta>0$ depending only on the choice of $w_0 \in \a^{++}$. 

We may thus apply Lemma \ref{diagonal} in each $V_{\rho_i}$ (with $[e_1]=p_i$) and conclude that for every small enough $\eps>0$, if the  neighborhood $V_\eps$ is chosen sufficiently small, the $\rho_i(g_j)$'s will satisfy the assumptions of Lemma \ref{produ} and $|\log \lambda_1(\rho_i(g_j)) - \log \lambda_1(\rho_i(h_j))| \leq \eps$. In particular combining the $\rho_i$'s together and using $(\ref{param})$ we get:
$$\lambda(g) = \sum_{j=1}^\ell \lambda(a_ja_0) + O(\eps \ell)$$
where the $O$ depends only on $G$. But since $w_0$ belongs to the interior of the convex set $K$ and $\eps>0$ can be chosen arbitrarily small, we see that $\frac{1}{n}\lambda(g)  \in K$ as desired.

It remains to verify that we can find $v_0$ in such a way that $S$ generates a Zariski-dense semigroup. This is the content of  the following lemma. Theorem \ref{realization} is now established.
\end{proof}

\begin{lemma}Let $G$ be a connected reductive real algebraic group and $O\subset G$ be an open set for the Hausdorff metric. Then there is a finite set $F \subset O$ which generates a Zariski-dense semigroup.
\end{lemma}

\begin{proof}This is standard. Note that the Zariski-closure of a semigroup is a group, so it is enough to find a finite $F$, which generates a Zariski-dense subgroup. Let $G_F$ be the Zariski-closure of the subgroup generated by $F$. If $G_F \neq G$ for all $F$, then there is a Zariski-connected algebraic subgroup $H \leq G$ such that $(G_F)^\circ=H$ for all large enough $F$ and hence $H$ has finite index in $G_F$. Also $H$ is normalized by every element in $O$, hence is normal in $G$. But this implies that every $g \in O$ has finite order in $G/H$, hence that $G/H$ is made exclusively of elements of finite order. But this means that $H$ has finite index in $G$, and hence that $H=G$ by connectedness.
\end{proof}

\subsection{Joint spectrum and joint spectral radius}\label{subsection.joint.spectral}

The notion of joint spectrum encompasses the more classical notion of joint spectral radius. Recall that if $S \subset \GL_d(\C)$ is compact, its \emph{joint spectral radius} is defined as 
$$R(S)=\lim_{n \to +\infty} \max_{g \in S^n} \|g\|^{\frac{1}{n}}.$$
Therefore, if $J(S)$ denotes as before the joint spectrum in the Weyl chamber $\{(x_1,\ldots,x_d) \in \R^d : x_1 \geq \ldots \geq x_d\}$ of $\GL_d(\C)$, the following is clear:
\begin{proposition} We have: $\log R(S) = \max\{x_1 : (x_1,\ldots,x_d) \in J(S)\}$.
\end{proposition}
Similarly, Theorems \ref{main1} and \ref{main2} immediately imply the Berger-Wang identity \cite{berger-wang} 
$$R(S)=\limsup_{n \to +\infty} \max_{g \in S^n} \lambda_1(g)^{\frac{1}{n}},$$ where $\lambda_1(.)$ denotes the spectral radius. The lower joint spectral radius, or joint subradius, can also be directly read off the joint spectrum:
$$R_{sub}(S):=\lim_{n \to +\infty} \min_{g \in S^n} \|g\|^{\frac{1}{n}}= \exp (\min\{x_1 : (x_1,\ldots,x_d) \in J(S)\}).$$

\subsubsection{Other representations}
More generally if $G$ is a real reductive group, $S \subset G$ a compact subset and $\rho: G \to \GL_d(\R)$ a finite dimensional representation with highest weight $\overline{\chi}_{\rho} \in \Hom(\a,\R)$, then the joint spectral radius of $\rho(S)$ can be read off the joint spectrum $J(S)$ as follows:

\begin{proposition}\label{spec-rep} We have: $\log R(\rho(S)) = \max\{\overline{\chi}_\rho(x) : x \in J(S)\}$.
\end{proposition}
Clearly, the same holds for $R_{sub}(\rho(S))$ replacing the maximum with the minimum in the previous expression.

\subsubsection{Outer envelope and positive boundary}\label{positiveboundary}
In the spirit of Kostant's article \cite{kostant} one can introduce a natural partial order on $\a^+$ by setting $x \le y $ if $\theta(x) \le \theta(y)$ for every dominant weight $\theta$. By \cite[Theorem 3.1]{kostant}, $x\leq y$ if and only if $\lambda_1(\rho(\exp x)) \leq \lambda_1(\rho(\exp y)) $ for every finite dimensional linear representation $\rho$ of $G$. It is easy to check that $x\leq y$ if and only if $x$ is contained in the conxex hull in $\a$ of the orbit $Wy$ of the Weyl group $W$. In accordance with Bochi's terminology in \cite{bochi-icm} we define the \emph{outer-envelope} $O(S)$ of $S$ to be the closed convex hull of the Weyl group orbit of $J(S)$ in $\a$, and we define its \emph{positive boundary} to be 
$$\partial O(S) \cap J(S),$$
where $\partial O(S)$ is the boundary of $O(S)$ in $\a$. It is easy to see that this coincides with the definition given in   Introduction. The positive boundary is also related to the joint spectral radius of linear representations. The set of dominant weights of finite dimensional representations of a real reductive group $G$ forms a dense set of directions in the Weyl chamber $\a^+$. Given Proposition \ref{spec-rep} the following is then clear: if $S,S'$ are compact subsets generating  Zariski-dense subgroups of $G$, we have:
\begin{proposition} $O(S) \subset O(S')$ if and only if $R(\rho(S)) \leq R(\rho(S'))$ for every finite dimensional linear representation $\rho$ of $G$. Moreover 
$$O(S) \cap \a^+ = \bigcap_{\rho} \{x \in \a^+ ; \overline{\chi}_{\rho}(x) \leq \log R(\rho(S))\}$$
where the intersection runs over all linear representations $\rho$ of $G$ with highest weight $\overline{\chi}_{\rho}$.
\end{proposition}

\subsubsection{Joint spectrum and asymptotic joint displacement}\label{asymp} In \cite{OregonReyes} and \cite{breuillard-fujiwara}  a geometric analogue to the notion of joint spectral radius was studied for subsets of isometries of an arbitrary metric space $(X,d)$, namely:
\begin{equation}\label{ell}\ell(S):=\lim_{n\to +\infty} \frac{1}{n} \max_{g \in S^n} d(gx,x)\end{equation} when $S \subset Isom(X)$.
This quantity does not depend on the choice of $x$ and was called the \emph{asymptotic joint displacement} in \cite{breuillard-fujiwara}. When $X$ is a symmetric space of non-compact type it turns out that $\ell(S)$ can easily be read off the joint spectrum in $G=Isom(X)^\circ$, which is a semisimple Lie group. Indeed we have:
$$\ell(S)=\max\{\|x\|, x \in J(S)\},$$ where $\|\cdot\|$ is the Euclidean norm on $\a$ induced by the Killing form. A related quantity,  the asymptotic translation length, can be defined by \begin{equation}\label{lamb}\lambda_\infty(S):=\limsup_{n \to +\infty}  \frac{1}{n}\sup_{g \in S^n} \inf_{x \in X} d(gx,x).\end{equation}  In \cite[Theorem 1.2]{breuillard-fujiwara} a geometric Berger-Wang identity was established to the effect that $\ell(S)=\lambda_\infty(S)$ for all $S$. We see now (at least when $S$ generates a Zariski-dense subgroup of $G$ or under $G$-domination assumption) that this also follows from the identity of the two limits in Theorems \ref{joint} and \ref{dominated-theorem}. In the next paragraph we prove this more generally when the Euclidean norm on $\a$ is replaced by an arbitrary $W$-invariant norm.

\subsubsection{Word metrics on reductive groups}\label{abmarpar} In \cite{abels-margulis} Abels and Margulis study word metrics on reductive Lie groups $G$. Their main result asserts that if $\Omega$ is a bounded symmetric compact neighborhood of the identity in $G$ then the word metric $$d_\Omega(x,y):=\inf \{n \in \N : x^{-1}y \in \Omega^n\}$$ is at a bounded distance from a \emph{norm like} pseudometric $d_{N_\Omega}$ on $G$, i.e. 
\begin{equation}\label{abmar} \exists C>0,  \forall x,y \in G ,\, |d_\Omega(x,y) - d_{N_\Omega}(x,y)| \leq C.\end{equation}
Norm like metrics are defined as follows: given a $W$-invariant norm $N_\Omega:\a \to \R_+$ we set $d_{N_\Omega}(x,y)=d_{N_\Omega}(1,x^{-1}y)$ and $d_{N_\Omega}(1,x)=N_\Omega(\kappa(x))$, where $\kappa(x)$ is the Cartan projection defined in \S \ref{reductive}. It is an exercise to prove that $d_N(x,y)$ does indeed satisfy the triangle inequality. 
In view of the above, we can understand the asymptotic joint displacement in the (pseudo-)metric space $(G,\rho)$ and read it off the joint spectrum of $\Omega$, thus establishing a Berger-Wang identity for word metrics on reductive groups. Namely defining $\ell(S)$ and $\lambda_\infty(S)$ as in $(\ref{ell})$ and $(\ref{lamb})$ for $d_\Omega$ we have:

\begin{proposition}[Berger-Wang for word metrics] Let $d_\Omega$ be a word metric on $G$ as defined above and $S \subset G$ be a compact subset generating a Zariski-dense subgroup. Then
\begin{equation}\label{BW-reduc}\ell(S)=\lambda_\infty(S)= \max_{x \in J(S)} N_\Omega(x).\end{equation}
\end{proposition}

\begin{proof} First observe that $\ell(S)\geq \lambda_\infty(S)$ (see also \cite{breuillard-fujiwara}). In view of $(\ref{abmar})$ we only need to check $(\ref{BW-reduc})$ for $d_{N_{\Omega}}$ in place of $d_\Omega$. So $\ell(S)=\lim \frac{1}{n}\max_{g \in S^n} N_\Omega(\kappa(g))$. Similarly, $d_{N_\Omega}(gx,x)\geq \limsup \frac{1}{n}d_{N_\Omega}(1,x^{-1}g^nx) = N_\Omega(\lambda(g))$, where the last equality follows from the spectral radius formula and continuity of $N_{\Omega}$. So $\lambda_\infty(S)\geq \limsup \frac{1}{n}\max_{g \in S^n} N_\Omega(\lambda(g))$. Now the desired identity follows from Theorem \ref{joint}. 
\end{proof}

We pass to the proof of Theorem \ref{word-balls}. 

\begin{lemma} Suppose $S=S^{-1}$ is a compact neighborhood of $1\in G$. Then $J(S)=O(S)\cap \a^+$.
\end{lemma}

\begin{proof}Since $G$ is connected, $S$ generates $G$. By Theorem \ref{body}, $J(S)$ has non empty interior in $\a$. In particular every point $x$ in the interior of $J(S)$ is $\R$-regular in the sense that $\overline{\alpha}_i(x)>0$ for $i=1,\ldots,d_S$. Given an arbitrarily small $\eps>0$ we may thus find $n \in \N$ and $g\in S^n$ such that $\|\lambda(g)/n - x\|<\eps$ and $g$ is $\R$-regular. Let $A_g$ be the unique maximal $\R$-split torus containing $g$ and $W_g$ its Weyl group. For every $w\in W_g$ there is $g_w \in G$ representing $w$. For $k\in \N$ consider the element $h_{k,w}=g_w g^k g_w^{-1}$. Note that, as $S$ generates $G$, there is $n_g \in \N$ such that $g_w \in S^{n_g}$ for all $w \in W_g$.  And $h_{k,w}\in A_g$. Now considering the Jordan projection of products of the form $\prod_{w \in W_g}h_{k_w,w}$ for large integers $k_w$ and passing to the limit as each $k_w\to+\infty$ and as $\eps\to 0$, we see that every point of $\mathfrak{a}^+$ that is in the convex hull of Weyl group orbit of $x$ belongs to $J(S)$. Hence $O(S)\cap \a^+\subset J(S)$.
\end{proof}

\begin{proof}[Proof of Theorem \ref{word-balls}] In view of $(\ref{abmar})$ and the previous lemma the proof follows from the following proposition.
\end{proof}

\begin{proposition} In the setting of Theorem \ref{word-balls}, $O(S)$ is precisely the unit ball of the norm $N_S$ on $\a$.
\end{proposition}

\begin{proof}  Since $S^n=\{g \in G: d_S(1,g)\leq n\}$ it is clear from $(\ref{abmar})$ and Theorem \ref{joint} that $J(S)$ is contained in the unit ball of $N_S$. Conversely since $N_S$ is $W$-invariant, we only need to check that every $x \in \a^+$ with $N_S(x)\leq1$ lies in $J(S)$. But $g_n=\exp(nx)$ has $\kappa(g_n)=nx$, so $N_S(\kappa(g_n))\leq n$ and hence by $(\ref{abmar})$ $g_n \in S^{n+C}$ for some constant $C \in \N$. Therefore as $n \to +\infty$, the limit points of $\frac{1}{n+C}\kappa(g_n)$ will belong to $J(S)$  by Theorem \ref{joint}. Hence $x \in J(S)$.
\end{proof}

\section{The Lyapunov vector for random products}\label{section5}

\subsection{Lyapunov vector and non-empty interior}

In this subsection, we prove Theorem \ref{interior} from the introduction, which also completes the proof of Theorem \ref{body}.

Apart from Abels-Margulis-Soifer's Theorem \ref{AMS} and additivity properties of $(r,\eps)$-Schottky families (Proposition \ref{loxodromy.implies.Cartan.close.to.Jordan}), the main input of our argument will be the \emph{non-degeneracy} of the Gaussian distribution in the  central limit theorem on linear groups. This non-degeneracy result was first proved with an exponential moment hypothesis for $\SL(d,\mathbb{R})$ by Goldsheid-Guivarc'h \cite{goldsheid-guivarch}, then for linear real semisimple groups by Guivarc'h \cite{guivarch} and finally with a second moment hypothesis on reductive groups by Benoist-Quint \cite{benoist-quint-central, benoist-quint-book}. It reads

\begin{theorem}[\cite{benoist-quint-central},\cite{benoist-quint-book}]\label{CLT}
Let $G$ be as in Theorem \ref{joint} and $\mu$ be a probability measure of finite second order moment on $G$. Suppose that the support of $\mu$ generates a Zariski dense semigroup in $G$. Then
$$
\frac{\kappa(Y_{n})-n\vec{\lambda}_{\mu}}{\sqrt{n}} \underset{n \to \infty}{\longrightarrow} \mathcal{N}(0,\sigma_{\mu}) \quad \text{in distribution}
$$
where $Y_n$ is the random walk at time $n$ and $ \mathcal{N}(0,\sigma_{\mu})$ is a Gaussian distribution on $\a$ centered at the origin and whose support is a subspace $\a_\mu$ containing $\mathfrak{a}_{S}:=\Lie([G,G]) \cap \a$.
\end{theorem}



The following lemma gives a criterion for a point to belong to the relative interior of the joint spectrum. As before $S$ is a compact subset of the connected reductive group $G$ and generates a Zariski-dense subgroup. 

\begin{lemma}\label{interior.lemma}
Let $\mathfrak{b}$ be a linear subspace of $\mathfrak{a}$ and let $x \in J(S)$. Assume that for every linear form $\ell$ on $\a$ with $\mathfrak{b} \not\subset \ker(\ell)$ and for every $n \geq 1$, there exists $y_n \in \kappa(S^{n})$ such that $\lim_{n \to +\infty} \ell(y_n- nx) =+\infty$. Then, $x$ belongs to the interior of $J(S)\cap (\mathfrak{b}+x)$ in the affine subspace $\mathfrak{b}+x$. In particular, when $\mathfrak{b}=\mathfrak{a}$, $x \in \inte(J(S))$.
\end{lemma}

\begin{proof}
Since $J(S)$ is convex by the first part of Theorem \ref{body} we only need to show that for every linear form $\ell$ with  $\mathfrak{b} \not\subset \ker(\ell)$, there exists $z\in J(S)$, with $\ell(z)>\ell(x)$. Say $y_n=\kappa(g_n)$ for some $g_n \in S^n$. By Theorem \ref{AMS} there are $r\geq \eps>0$ and a finite set $F \subset \cup_n S^n$ such that for all $n$ there is $f_n \in F$ such that $g_nf_n$ is $(r,\eps)$-proximal in $G$. Set $z_n=\lambda(g_nf_n)/m$, where $m$ is such that $g_nf_n \in S^m$. Note that $m-n$ is bounded. Then clearly $z_n \in J(S)$ and Lemmas \ref{Cartan.stability}, \ref{loxodromy.implies.Cartan.close.to.Jordan} guarantee that $\|nz_n - y_n\|$ remains bounded, say by $C>0$. Then
$$\ell(z_n) \geq \frac{\ell(y_n) - C\|\ell\|}{n} = \ell(x) +  \frac{1}{n}(\ell(y_n - nx) - C\|\ell\|),$$
which becomes $>\ell(x)$ when $n$ is large enough. This ends the proof.
\end{proof}

In the next lemma, we observe that Theorem \ref{CLT} implies that the criterion of the previous lemma is satisfied for the subspace $\b=\mathfrak{a}_\mu$.

\begin{lemma}\label{CLT.lemma}
Let $c>0$ and $\ell$ be a linear form on $\a$ such that $\mathfrak{a}_{\mu} \not\subset \ker(\ell)$. Then for every large enough $n$, there is $y_n \in \kappa(S^n)$ such that $\ell(y_n-n\vec{\lambda}_{\mu}) \geq c \sqrt{n}$.
\end{lemma}
\begin{proof}
Let $\nu_{\mu}$ denote the limit Gaussian distribution on $\a_\mu$ given by Theorem \ref{CLT}.  Since $\nu_\mu$ is non-degenerate and $\ker(\ell) \not\supset \mathfrak{a}_{\mu}$  for any $c>0$, $\nu_{\mu}( \{y \in \mathfrak{a}, \ell(y)=c\})=0$. Therefore by Theorem \ref{CLT}
\begin{align}
\nu_\mu(\{y \in \mathfrak{a}, \ell(y)\geq c\}) =\lim_{n \to \infty}\mathbb{P}(\ell(\kappa(Y_{n} - n\vec{\lambda}_\nu))\geq c \sqrt{n}).
\end{align}
Since the left hand side is strictly positive, the result follows.
\end{proof}

We now complete the proof of Theorem \ref{interior}. We are going to show that the affine span of $J(S)$ is $\vec{\lambda}_{\mu} + \a_\mu$. The previous lemma shows that $\vec{\lambda}_{\mu}$ belongs to the relative interior of $J(S)\cap (\vec{\lambda}_{\mu} + \a_\mu)$ in the affine space $\vec{\lambda}_{\mu} + \a_\mu$. In particular the affine span of $J(S)$ contains $\vec{\lambda}_{\mu} + \a_\mu$. So it only remains to show the opposite inclusion. But the key non-degeneracy claim in Theorem \ref{CLT} says that $\a_S \subset \a_\mu$. Therefore it suffices to look at the projected random walk on the abelian quotient $G/[G,G]$ and everything boils down to the following easy observation:  if $\Sigma$ is a compact subset of $\R^k$ and $\nu$ a probability measure on $\R^k$ with $\Supp(\nu)=\Sigma$, then the vector subspace supporting the affine span of $\Sigma$ in $\R^k$ coincides with the support of the limiting Gaussian distribution in the central limit theorem for sums of i.i.d.\ random variables with common law $\nu$.

This establishes that the affine span of $J(S)$ in $\a$ is indeed $\vec{\lambda}_{\mu} + \a_\mu$. The claim about the Benoist cone is clear. This ends the proof of Theorem \ref{interior}. Moreover note that by the previous observation, since $\a_S \subset \a_\mu$, we have $\a_\mu=\a$ unless $S$ is contained in a coset of a proper closed subgroup of $G$ containing $[G,G]$. This ends the proof of the remaining part of Theorem \ref{body}.


We end this subsection by recording the following consequence of the above:

\begin{proposition}\label{prop.lyapunov.in.Benoist}
Let $G$ be a connected real semisimple linear algebraic group and let $\mu$ be a probability measure on $G$. Suppose that the support of $\mu$ generates a Zariski dense semigroup $\Gamma$ of $G$ with Benoist cone $\mathcal{BC}(\Gamma)$ and that $\mu$ has a finite second order moment. Then, we have $\vec{\lambda}_{\mu} \in \inte(\mathcal{BC}(\Gamma))$. 
\end{proposition}
Indeed, if the support of $\mu$ is bounded, then Theorem \ref{interior} gives a more precise result via Proposition \ref{prop.js.generates.benoist}. If the support of $\mu$ is not bounded, the proof follows mutatis mutandis using the same argument as in Lemmata \ref{interior.lemma} and \ref{CLT.lemma}. We omit the details to avoid repetitions.

\subsection{Continuity of Lyapunov exponents in case of domination}\label{subsec.if.dom.then.lyap.cts}

Continuity (and discontinuity) properties of Lyapunov exponents have    long been studied starting, in the context of linear cocycles, with Ruelle \cite{ruelle.analytic}, Furstenberg-Kifer \cite{furstenberg-kifer}, Le Page \cite{LePage-regularite}, Peres \cite{Peres,peres2}. More recent works in this direction include articles \cite{bochi-viana, bocker-viana} and surveys \cite{bonatti-diaz-viana, duarte-klein, viana-lectures, viana-survey}. In particular, the continuity of Lyapunov exponents under a  domination condition, see Proposition \ref{prop.lyapunov.continuous} below, is well-known for $G=\GL_d(\mathbb{R})$ or $\SL_d(\mathbb{R})$. Its classical proof uses the existence of continuous dominated splittings and the consequent fact that under domination, one can express the Lyapunov vector of a measure as a suitable integral with respect to this measure (see for example \cite[Section 12.1]{bonatti-diaz-viana}, \cite[Corollary 2.4, $(1.7)$]{bochi-rams}). Below, we give a more direct proof of this continuity for general $G$ under a $G$-domination assumption, using properties of $(r,\eps)$-Schottky families. This point of view will be helpful in the proof of Theorem \ref{lyapspec} below.

Let $G$ be a connected linear reductive group and let $S$ be a compact subset of $G$. Denote by $\mathcal{P} (S^{\mathbb{N}})$ the set of shift-invariant probability measures on the compact 
$S^\mathbb{N}$. The set $\mathcal{P} (S^{\mathbb{N}})$ is a compact convex subset of the Banach space of 
finite signed measures on $S^\mathbb{N}$. Let $\mathcal{P}_{erg}(S^\mathbb{N})$ be the set of ergodic probability measures on $S^\mathbb{N}$, i.e. the 
set of extremal points of the convex $\mathcal{P}(S^{\mathbb{N}})$. Recall also that the set $\mathcal{P}_{erg}(S^\mathbb{N})$ is a dense subset of $\mathcal{P}(S^{\mathbb{N}})$ and for each $\beta \in \mathcal{P}_{erg}(S^\mathbb{N})$, there corresponds a Lyapunov vector $\vec{\lambda}(\beta) \in \mathfrak{a}^+$ defined by Kingman's subadditive ergodic theorem for $\beta$-a.s. $b=(b_1,\ldots,) \in S^{\mathbb{N}}$ as the limit of $\frac{1}{n} \kappa(b_1.\ldots.b_n)$.

\begin{proposition}\label{prop.lyapunov.continuous}
Let $S$ be a compact set in $G$ satisfying the full $G$-domination condition (see Definition \ref{Gdom}).  The map $\phi: \mathcal{P}_{erg}(S^\mathbb{N}) \to \mathfrak{a}^+$ defined by $\beta \mapsto \vec{\lambda}(\beta)$ is continuous. 
\end{proposition}

\begin{proof}
Since $S$ is $G$-dominated, thanks to Lemma \ref{loxodromy.implies.Cartan.close.to.Jordan} and Proposition \ref{prop.domination.and.schottky} for $\beta \in \mathcal{P}_{erg}(S^\mathbb{N})$, for $\beta$-almost every $b=(b_1,\ldots,)\in S^\N$ we have $\frac{1}{n} \lambda(b_n.\ldots.b_1) \to \vec{\lambda}(\beta)$ as $n \to +\infty$. 

Now, we need to show that given $\delta>0$, if $\beta_2 \in \mathcal{P}_{erg}(S^\mathbb{N})$ is sufficiently close to a fixed $\beta_1 \in \mathcal{P}_{erg}(S^\mathbb{N})$, then $\|\vec{\lambda}(\beta_1)-\vec{\lambda}(\beta_2)\|<\delta$ for some fixed norm on $\mathfrak{a}$. Fix such a $\delta>0$. By Proposition \ref{prop.domination.and.schottky}, there exists $r>0$ such that fixing some $\eps \in (0,r)$, if $n \in\mathbb{N}$ is large enough (say larger than $n_\eps \in \mathbb{N}$), then $S^{n}$ is an $(r,\eps)$-Schottky family. Let $C_r>0$ be the constant given by Proposition \ref{Best}. Fix some $n_0 \geq n_\eps$ such that $\frac{1}{n_0}(2(C_r +1)) <\frac{\delta}{2}$. Fix also an $\eta>0$ small enough so that if $g,h \in S^{n_0}$ are such that $d_G(g,h) \leq \eta$, then $\|\lambda(g)-\lambda(h)\| \leq 1$, where $d_G$ is a fixed Riemannian  metric on $G$. Let $M>0$ be the diameter of a compact subset $K$ of $\mathfrak{a}^+$ such that $\frac{\lambda(S^n)}{n} \in K$ for every $n \geq 1$. Now, consider a partition of $S^{n_0}$ into $k_\eta$ subsets $B_1,\ldots,B_{k_\eta}$ of diameter at most $\eta$ and such that $\beta_1(\partial B_j)=0$ for $j=1,\ldots,k_\eta$ (one can always guarantee this by a small perturbation of the partition). Finally let $\rho>0$ be small enough constant so that $M k_\eta \rho < \frac{\delta}{2}$. 

Let  $\beta_{1,n_0}= p_{n_0} {}_\ast(\beta_1)$ and $ \beta_{2,n_0} =p_{n_0} {}_\ast(\beta_2)$ be the push-forwards of $\beta_1$ and $\beta_2$ by the map $p_{n_0}: S^\mathbb{N} \to S^{n_0}$ given by $b=(b_1,\ldots) \mapsto b_1\cdot\ldots\cdot b_{n_0}$. By continuity, it follows that if $\beta_2$ is close enough to $\beta_1$, then $|\beta_{1,n_0}(B_i) - \beta_{2,n_0}(B_i)| <\rho$ for each $i=1,\ldots,k_\eta$. So fix such a $\beta_2$. By Birkhoff's ergodic theorem, it is easy to see that for $i \in \{1,2\}$, there exists $X_i \subset S^\mathbb{N}$ with $\beta_i(X_i)=1$ and such that for every $b^{(i)}=(b^{(i)}_1,\ldots,)\in X_i$, we have $\frac{1}{n} \lambda(b^{(i)}_1\cdot \ldots\cdot b^{(i)}_n) \to \vec{\lambda}(\beta_i)$ as $n \to +\infty$ and for $m=qn_0$ and $\ell=1,\ldots,q \in \mathbb{N}$, denoting the element $b_{n_0(\ell-1)+1}^{(i)}\cdot\ldots\cdot b_{n_0\ell}^{(i)}$ of $S^{n_0}$ by $g^{(i)}_\ell$ and for $j=1,\ldots,k_\eta$, setting $\alpha_{j,q}^{i}=\frac{1}{q}\#\{g_\ell \in B_j \, |\, \ell=1,\ldots,q\}$, we have $\alpha_{j,q}^i \to \beta_{i,n_0}(B_j)$ as $q \to +\infty$. 

On the other hand, by Proposition \ref{Best}, for $i=1,2$, we have 
\begin{equation}\label{eq.lyapcts1}
\|\frac{1}{m}\lambda(b_1^{(i)}\cdot\ldots\cdot b_m^{(i)})- \frac{1}{n_0}\frac{1}{q} \sum_{\ell=1}^q \lambda(g_\ell^{(i)})\|\leq  \frac{C_r}{n_0}.
\end{equation} 
Furthermore, for $j=1,\ldots,k_\eta$, fixing some element $h_j \in B_j$,  for $i=1,2$, we have 
\begin{equation}\label{eq.lyapcts2}
\|\frac{1}{q}\sum_{\ell=1}^q \lambda(g_\ell^{(i)})-\sum_{j=1}^{k_\eta} \alpha_{j,q}^{i} \lambda(h_j)\|\leq 1.
\end{equation}
Putting the inequalities $(\ref{eq.lyapcts1})$ and $(\ref{eq.lyapcts2})$  together, we get
\begin{equation*}
\begin{aligned}
\|\vec{\lambda}(\beta_1)-\vec{\lambda}(\beta_2)\|&=\limsup_{q \to + \infty} \|\frac{1}{qn_0}\lambda(b_1^{(1)}\cdot\ldots\cdot b_{qn_0}^{(1)})-\frac{1}{qn_0}\lambda(b_1^{(2)}\cdot\ldots\cdot b_{qn_0}^{(2)})\|\\ & \leq 2 \frac{C_r}{n_0}+\frac{2}{n_0} + \rho k_\eta M <\delta,
\end{aligned}
\end{equation*}
as desired.
\end{proof}

\begin{remark}
It follows from the above proof that if a set $S$ of $\GL_d(\mathbb{R})$ is $k$-dominated, then the sum of first $k$ Lyapunov exponents depends continuously on the ergodic measures on $S^\mathbb{N}$. In particular, if $S$ is $i$-dominated for $i=1,\ldots,k$, all first $k$ Lyapunov exponents are continuous. 
\end{remark}

\begin{remark}
As is well-known the dependence of Lyapunov exponents to the ergodic measures on $S^{\mathbb{N}}$ for an arbitrary compact set $S$ is not continuous in general. As a simple example, let $S=\{a,r\}$ where $a \in \SL(2,\mathbb{R})$ is an hyperbolic element and $r \in SO(2,\mathbb{R})$ such that $rar=a^{-1}$. For $s \in [0,1]$ define $\mu_s=s \delta_a +(1-s)\delta_r$. It is immediate to see that for every $s<1$, $\vec{\lambda}(\mu_s)=0$ and $\vec{\lambda}(\mu_1)=\lambda_1(a)>0$.
\end{remark}

\subsection{Lyapunov spectrum}

Denote by $Lyap(S)$ the \emph{Lyapunov spectrum} of $S$, i.e.  the set of all possible Lyapunov vectors of ergodic measures on $S^\mathbb{N}$, that is $Lyap(S)=\{\vec{\lambda}(\beta)  \, |\, \beta \in P_{erg}(S^{\mathbb{N}})\} \subset \mathfrak{a}^+$. 

Different properties of the Lyapunov spectrum for more general linear cocycles, such as properties of extremal orbits and ergodic measures, density of Lyapunov vectors of periodic measures, have been studied in detail in the context of non-commutative ergodic optimization \cite{bochi-icm,jenkinson-survey,jenkinson-survey0},  for example by Feng \cite{feng1,feng2}, Morris \cite{morris-mather-sets}, Kalinin \cite{kalinin}, 
Bochi-Rams \cite{bochi-rams}, and more recently, by Bochi \cite{bochi-icm} and Park \cite{park}.

Below, we prove Theorem \ref{lyapspec} from the introduction and show in fact that each point in the relative interior of joint spectrum can be realized as the Lyapunov vector of some shift invariant ergodic probability measure. Furthermore, recall that by Theorem \ref{interior}, the Lyapunov vectors of i.i.d.\ measures belong to the relative interior of the joint spectrum. We also give an example in Proposition \ref{confined}, where the set of Lyapunov vectors of i.i.d.\ measures (i.e.\ $0$-step Markov measures) does not fill the interior of joint spectrum and is a proper compact subset therein.

\subsubsection{Proof of Theorem \ref{lyapspec}}\label{subsub.lyapspec}

We start with a lemma which enables us to approximate any given finite subset of the joint spectrum using a Schottky family. A similar result appears in \cite{sert-growth}.

\begin{lemma}\label{lemma.render.schottky.family}
Let $S$ be a compact subset  generating a Zariski-dense semigroup $\Gamma$ in a connected reductive linear group $G$. For all $t \in \mathbb{N}$ and $\eta>0$, given any $t$ points $x_1,\ldots,x_t$ in the relative interior of the joint spectrum $J(S)$ of $S$, for $i=1,\ldots,t$, there exist $n_i \in \mathbb{N}$, $g_i \in S^{n_i}$ with the property that 
$\|\frac{\lambda(g_i)}{n_i} -x_i\|<\eta$ and $\{g_1,\ldots,g_t\}$ is an $(r,\eps)$-Schottky family in $G$ for some $r\geq \eps>0$.
\end{lemma}
\begin{proof} The proof is similar to many of our earlier arguments involving Theorem \ref{AMS}. We spell out the details for the reader's convenience. We first show that there exist elements $h_1,\ldots,h_t$ in $\Gamma$ that are  $(r,\eps)$-proximal in $G$ and have $\|\frac{\lambda(h_i)}{n_i} -x_i\|<\eta$. There are $r > \eps>0$ and a finite subset $F$ in  $\Gamma$ such that the conclusion of  Theorem \ref{AMS} holds. For each $f \in F$, fix $n_f \in \mathbb{N}$ with $f \in S^{n_f}$. Now by Theorem \ref{joint} there exist $m'_i \in \mathbb{N}$, $h'_i \in S^{m'_i}$ such that $\|\frac{1}{m'_i}\kappa(h'_i)-x_i\|<\frac{\eta}{3}$. Note that here $m'_i$'s can be chosen arbitrarily large. Now, for each $h'_i$, there exists $f_i \in F$ such that $h'_if_i$ is $(r,\eps)$-proximal in $G$ and $\|\kappa(h'_if_i)-\kappa(h'_i)\| \leq M$ by Lemma \ref{Cartan.stability} for some $M=M(F)$. Moreover, by Lemma \ref{loxodromy.implies.Cartan.close.to.Jordan}, $\|\lambda(h'_if_i)-\kappa(h'_i)\|<M+C_r$. It follows that if $m'_i$'s are large enough (depending on $M+C_r$ and $\eta$), setting $m_i=m'_i+n_{f_i}$ and $h_i=h'_if_i$, we have $h_i \in S^{m_i}$, $\|\frac{1}{m_i}\lambda(h_i)-x_i\|\leq \frac{2\eta}{3}$ for $i=1,\ldots,t$ and thus $\{h_1,\ldots,h_t\}$ satisfies are requirements.

Now let us prove the full claim. As in Lemma \ref{dispersion.lemma}, it follows by Zariski connectedness and irreducibility of distinguished representations $(\rho_1,V_1),\ldots,(\rho_{d_S},V_{d_S})$ that there exists $\gamma \in \Gamma$, say $\gamma \in S^{n_\gamma}$, with the property that for every $j=1,\ldots,d_S$ and $i_1,i_2 \in \{1,\ldots,t\}$, we have $\rho_j(\gamma)v^{j,+}_{h_{i_1}} \notin H^{j,<}_{h_{i_2}}$ where $v^{j,+}_{h_{i_1}} \in \mathbb{P}(V_j)$ and $H^{j,<}_{h_{i_2}}\subset \mathbb{P}(V_j)$ are respectively the attracting point and the repelling hyperplane of the proximal elements $\rho_j(h_{i_1})$ and $\rho_j(h_{i_2})$. By taking a large enough power by $p_i \in \mathbb{N}$ of $h_i$ (which does not modify $v^{j,+}_{h_{i}}$'s and $H^{j,<}_{h_{i}}$'s), by Lemma \ref{left.multiply.proximal.power.corollary}, we get that $\gamma h_i^{p_i}$'s are $(\hat{r},\hat{\eps})$-proximal in $G$ for some $\hat{r}> \hat{\eps}>0$. Moreover, since $\rho_j(\gamma)$'s act as Lipschitz transformations on $\mathbb{P}(V_j)$ for $j=1,\ldots,d_S$ and since $d(\rho_j(\gamma)v^{j,+}_{h_{i_1}}, H^{j,<}_{h_{i_2}})$'s are uniformly bounded from below by a positive constant, up to enlarging $p_i$'s, by \cite[Lemma 4.6]{Sert.LDP}, we conclude that there exists $\rho>0$ such that $d(v^{j,+}_{\gamma h_{i_1}^{p_{i_1}}},H^{j,<}_{\gamma h_{i_2}^{p_{i_2}}}) \geq \rho>0$ for every $j=1,\ldots,d_S$ and $i_1,i_2 \in \{1,\ldots,t\}$. 

Now, as before, by Lemma \ref{Cartan.stability} and Lemma \ref{loxodromy.implies.Cartan.close.to.Jordan}, if $p_i$'s are large enough (depending on $\gamma$ and $\eta$), setting $n'_i=n_\gamma +p_i m_i$ and $g'_i=\gamma h_i^{p_i}$, we see that $g'_i \in S^{n'_i}$, $||\frac{1}{n'_i}\lambda(g'_i)-x_i||\leq \eta$. Now replacing $g'_i$'s by their large enough (depending on $\rho$ and $\hat{\eps}$) powers $g_i$'s, the desired conclusion is obtained for $\{g_1,\ldots,g_t\}$. 
\end{proof}

We are now ready to complete the proof of Theorem \ref{lyapspec}. Brouwer's fixed point theorem is used at the end to show the claimed surjectivity. A similar idea was used by Thi Dang in her thesis \cite{thi-dang}. 

\begin{proof}[Proof of Theorem \ref{lyapspec}]
Let $x \in \mathfrak{a}^{++}$ be in the relative interior of $J(S)$. Choose $t=\rk(G)+1$ points $x_1,\ldots,x_t$ in the relative interior of $J(S)$ such that $x$ belongs to the interior of the convex hull of the $x_i$'s. Choose elements $g_i \in S^{n_i}$ satisfying the conclusions of  Lemma \ref{lemma.render.schottky.family}. Note that replacing $n_i$'s by $pn_i$'s for some large $p \in \mathbb{N}$, i.e. passing to $p^{th}$-powers of $g_i$'s (this does not modify the conclusions of Lemma \ref{lemma.render.schottky.family}), we can suppose that $n_i \geq M$ for each $i=1,\ldots,t$, where the large constant $M$ is to be specified later. For $i=1,\ldots,t$, denote by $C_i$ an $n_i$-tuple $(b_1^{(i)},\ldots,b_{n_i}^{(i)})$ in the $n_i^{th}$-cartesian power of $S$ such that $g_i=b_1^{(i)}.\ldots.b_{n_i}^{(i)}$. Denote by $\Sigma$ the set $\{1,\ldots,t\}$ seen as an alphabet of $t$ letters. Denote by $\sigma$ the shift map on $\Sigma^{\mathbb{Z}}$. Let $T_t$ be the $t-1$-dimensional simplex of Bernoulli probability measures on $\Sigma^{\mathbb{Z}}$. For $s=(s_1,\ldots,s_t) \in T_t$ (i.e. $s_i \geq 0$ and $\sum s_i=1$), denote the corresponding Bernoulli measure on $\Sigma^{\mathbb{Z}}$ by $\beta_s$. Now define a roof function $r:\Sigma \to \mathbb{N}_+$, by setting $r(i)=n_i$ for $i=1,\ldots,t$. Consider the associated discrete suspension/Kakutani skyscraper $(\hat{\Sigma},\hat{\beta}_s,\hat{\sigma})$. This system is ergodic. Define the map $\Psi:\hat{\Sigma} \to S^\mathbb{Z}$ by setting for $\underline{j}:=(\ldots,j_{-1},j_0,j_1,\ldots)\in \Sigma^{\mathbb{Z}}$ and $k\in [0,r(j_0)-1]$, $\psi((\underline{j},k))$ to be the sequence $\underline{b} \in S^{\mathbb{N}}$ given by the concatenation of finite words $(\ldots, C_{j_{-1}},C_{j_0},C_{j_1},\ldots)$ such that the $k^{th}$-letter of $C_{j_0}$ is at the $0^{th}$-position. Let $\hat{\theta}_s$ denote the push-forward $\Psi_{\ast}\hat{\beta}_s$ on $S^{\mathbb{Z}}$ and $\theta_s$ the push-forward of $\hat{\theta}_s$ on $S^{\mathbb{N}}$. Denote also by $\sigma$ the shift on $S^{\mathbb{N}}$. The system $(S^{\mathbb{N}},\theta_s,\sigma)$ is by construction a factor  of the suspension $(\hat{\Sigma},\hat{\beta}_s,\hat{\sigma})$ and thus it is automatically ergodic since its extension is. In other words we have the commuting system of measure preserving maps as below:
$$ \begin{tikzcd}
(\hat{\Sigma},\hat{\beta}_s) \arrow{r}{\hat{\sigma}} \arrow[swap]{d}{\Psi} & (\hat{\Sigma},\hat{\beta}_s) \arrow{d}{\Psi} \\%
(S^{\mathbb{N}},\theta_s) \arrow{r}{\sigma}& (S^{\mathbb{N}},\theta_s)
\end{tikzcd}
$$
Since the set $\{g_1,\ldots,g_t\}$ is an $(r,\eps)$-Schottky family in $G$, it follows from Proposition \ref{prop.domination.and.schottky} that $S$ satisfies the $G$-domination condition of Definition \ref{Gdom}. This in turn allows us to use Proposition \ref{prop.lyapunov.continuous} to conclude that the map $\phi: T_t \to J(S)$  defined by $\phi(s)=\vec{\lambda}(\theta_s)$ is continuous. 

Now, it follows by construction of the ergodic measure $\theta_s$ for $s=(s_1,\ldots,s_t) \in T_t$ that along $\theta_s$-almost every orbit $b=(b_1,b_2,\ldots) \in S^\mathbb{N}$  what we see after an initial segment of uniformly bounded length ($\le \max_{i=1,\ldots,t} n_i$) is a concatenation of finite words corresponding to $C_i$'s for $i=1,\ldots,t$. By ergodicity (i.e. Birkhoff's ergodic theorem), there is a shift-invariant $\theta_s$-full measure set $(S^\mathbb{N})'$ such that  the asymptotic proportion of appearances of a $C_i$ in an element $b=(b_1,\ldots,) \in (S^\mathbb{N})'$ is equal to $s_i$ for $i=1,\ldots,t$. Clearly, we can also suppose that for every $b \in (S^\mathbb{N})'$, we have $\lim_k \frac{1}{k}\kappa(b_1\cdot \ldots \cdot b_k)=\vec{\lambda}(\theta_s)$. By shift-invariance, choose $b \in (S^\mathbb{N})'$ such that $b$ is given by a concatenation $(C_{j_1},C_{j_2},\ldots)$ for some word $\underline{j}=(j_1,j_2,\ldots) \in \Sigma^\mathbb{N}$. Let $k$ be an integer such that $k=n_{j_1}+\ldots+n_{j_\ell}$ for some large $\ell \in \mathbb{N}$. Since $\{g_1,\ldots,g_t\}$ is an $(r,\eps)$-Schottky family in $G$,  by Lemma \ref{loxodromy.implies.Cartan.close.to.Jordan} and Proposition \ref{Best}, it follows that
\begin{equation}\label{eq.realizing1}
||\frac{1}{k}\lambda(b_1\cdot\ldots\cdot b_k)-\frac{1}{k}\sum_{i=1}^{\ell}\lambda(g_{j_i})||\leq \frac{\ell}{k}(C_r +1)
\end{equation}
Since for each such $\ell$ and $k$, we have $\frac{\ell}{k}\leq \frac{1}{M}\leq \frac{1}{\min_{i=1,\ldots,t} n_{i}}$, given any $\delta>0$, we can 
choose $M$ large enough (depending only on $C_r$ and $\delta$) so as to deduce from (\ref{eq.realizing1}) that
$$
||\vec{\lambda}(\theta_s)-\sum_{i=1}^t s_i \frac{\lambda(g_i)}{n_i}||<\delta
$$
As a consequence, by pre-composing the map $\phi: T_t \to J(S)$ by a suitable linear automorphism, it follows easily from Brouwer's fixed point theorem (in the form given say in \cite[Lemma 7.23]{rudin}) that there exists $s \in T_t$ with $\vec{\lambda}(\theta_s)=x$, proving the theorem.
\end{proof}

\begin{remark}\label{remark.Gibbs}
One easily sees that the supports of ergodic measures $\hat{\theta}_s$ on $S^\mathbb{Z}$ constructed in the previous proof are renewal systems. These are particular types of sofic subshifts (see \cite{lind-marcus, adler-walters}). Determining what points of the joint spectrum can be realized as Lyapunov vectors of processes associated to Gibbs measures remains an interesting problem.
\end{remark}


\subsubsection{Exposed points of the joint spectrum}\label{subsub.extremal} Here we prove Proposition \ref{prop.extremal.points} from Introduction and give an example of a set $S$ whose joint spectrum has an extremal point that does not belong to the Lyapunov spectrum of $S$.

\begin{proof}[Proof of Proposition \ref{prop.extremal.points}]
 We begin with part (1) of the proposition. Let $S \subset G$ be a compact $G$-dominated set. For $\mu \in \mathcal{P}_{erg}(S^\mathbb{N})$, the vector $\lim_n \frac{1}{n} \mathbb{E}_\mu[\kappa(b_n\ldots b_1)]$  can be expressed as an integral of a continuous function on $S^\N$ with respect to $\mu$ (see e.g. \cite[$(1.7)$]{bochi-rams}). As a consequence, the map $\mathcal{P}(S^\mathbb{N}) \ni \mu \mapsto \lim_n \frac{1}{n} \mathbb{E}_\mu[\kappa(b_n\ldots b_1)]$ is continuous and the restriction of a linear map (on the Banach space of finite signed measures on $S^\N$). In particular its image is a convex compact subset of $\mathfrak{a}^+$. By Theorem \ref{lyapspec} its image must be equal to $J(S)$ and by ergodic decomposition, an extremal point of $J(S)$ is attained by an element of $\mathcal{P}_{erg}(S^\mathbb{N})$.
 
 We now move to part (2) and assume that $S$ is a compact subset of $G$ generating a Zariski dense semigroup. Let $x \in \mathfrak{a}^+$ be an exposed point of $J(S)$ lying on the positive boundary (see \S \ref{positiveboundary}). Then there is a linear form $\ell$ on $\a$, which is a positive linear combination of fundamental weights, such that for every $y\in J(S)$, we have $\ell(y)\leq \ell(x)$, with equality if and only if  $y=x$. 
 For positive $n \in \mathbb{N}$, let $f_n$ be the function defined on $S^\N$ by $f_n(b)=\ell(\kappa(b_n\ldots b_1))$. Since $\ell$ is a positive linear combination of fundamental weights, $a\mapsto \ell(\kappa(a))$ satisfies the triangle inequality $\ell(\kappa(ab)) \leq \ell(\kappa(a))+\ell(\kappa(b))$. Hence $f_n$ is a continuous subadditive  cocycle. 

By Theorem \ref{joint}, there is $b \in S^\N$ be such that $\frac{1}{n} \kappa(b_n \ldots b_1) \to x$. In particular, $\lim_n \frac{1}{n}f_n(b)=\ell(x)$. It follows from a classical observation (see e.g. \cite[Lemma A.6.]{morris-mather-sets}) that there exists a shift-invariant probability measure $\tilde{\beta} \in \mathcal{P}(S^\N)$ such that $\lim_n \frac{1}{n} \mathbb{E}_{\tilde{\beta}}[\ell(\kappa(b_n \ldots b_1))] \geq \ell(x)$. By ergodic decomposition, there exists $\beta \in  \mathcal{P}_{erg}(S^\N)$ satisfying the same inequality. In particular, $\ell(\vec{\lambda}(\beta)) \geq \ell(x)$, and hence $x=\vec{\lambda}(\beta)$, as desired. 
\end{proof}

In the absence of $G$-domination, the exposed extremal points on the boundary of the joint spectrum are not necessarily realized as Lyapunov vectors (see \cite{bochi-morris} Theorem 1.11 and Remark 1.13). We now give another explicit example of this phenomenon in our context.

\begin{Example}
For $n \in \mathbb{N}^{\ast}$, consider the set $S_n \subset \GL(2,\R)$ as in Example \ref{ex.discon}. 
The joint spectrum $J(S_n)$ is given by the triangle in the Weyl chamber $\mathfrak{a}^+$ of $\GL(2,\R)$ with vertices $(0,0)$, $(0,\log \alpha)$ and $(\frac{\log \lambda_1(ab)}{2},0)$ (see Figure \ref{fig.discontinuity}). It is clear that any probability measure $\mu \in \mathcal{P}_{erg}(S_n^\mathbb{N})$ giving positive mass to the cylinder set corresponding to the element $\alpha r_n$ has its Lyapunov vector away from the $x$-axis. On the other hand, if such a $\mu$ gives no mass to that cylinder, its Lyapunov vector is contained in the segment $[\log \lambda_1(a),\frac{\log \lambda_1(ab)}{2}]$ of $x$-axis. In particular, the exposed extremal point $(0,0)$ of $J(S_n)$ is never attained by the Lyapunov vector of any $\mu \in \mathcal{P}_{erg}(S_n^\mathbb{N})$.
\end{Example}

\subsubsection{i.i.d.\ Lyapunov spectrum}\label{subsub.iid.Lyap}

We prove Proposition \ref{confined}. The is a simple consequence of the continuity properties of Lyapunov vectors proved by Hennion and Furstenberg-Kifer \cite{hennion, furstenberg-kifer} and of Theorem \ref{interior}.

\begin{proof}[Proof of Proposition \ref{confined}] Clearly, the joint spectrum of $S$ is equal to $[0, \log R(S)]$ where $R(S)>0$ is the joint spectral radius of $S$ (in passing, it is easy to see that $R(S)=\sqrt{\lambda_1(ab)}$, where as usual $\lambda_1(.)$ denotes the largest modulus of eigenvalues). For $p \in [0,1]$, denote the probability measure $p\delta_{a}+(1-p)\delta_{p}$ by $\mu_{p}$. Let $\phi_S:[0,1] \to [0,\log R(S)]$ be the function given by $\phi_S(p)=\vec{\lambda}_{\mu_{p}}$, i.e. $\phi_S(p)=\lim_{n \to \infty}\frac{1}{n}\mathbb{E}_{p}[\log ||Y_{n}||]$. Since for any finite word $\omega$ in the alphabet $S=\{a,b\}$, exchanging $a$ and $b$ does not alter the operator norm of the corresponding products, we get that for each $p \in [0,1]$, one has $\phi_S(p)=\phi_S(1-p)$. By  \cite{hennion, furstenberg-kifer} we know that $\phi_S$ is continuous on $(0,1)$. By Theorem \ref{interior} for $p \in (0,1)$ we have $\phi_S(p) \in (0,\log R(S))$, to show the claim, we only need to show $\lim_{p \to 1^{-}} \phi_S(p)=0$. This is easy: note first that for each $p\in(0,1)$ and every $n \in \mathbb{N}$, we have 
\begin{equation}\label{eq6}
\mathbb{E}_{p}[\log \|Y_{n}\|] \leq  \mathbb{P}_{p}(X_{1}=\ldots=X_{n}=a)\log (\|a^{n}\|) +\log(\max_{g\in S^{n}}\|g\|)(1-p^{n})
\end{equation}
On the other hand, by subadditivity of the sequence $(\mathbb{E}_{\mu_p}[\log ||Y_{n}||])_{n \geq 1}$, one has $\phi_S(p)=\inf_{n \geq 1}\frac{1}{n} \mathbb{E}_{p}[\log ||Y_{n}||]$. Therefore, it follows from (\ref{eq6}) that for each fixed $p$ and every $n \in \mathbb{N}$, we have 
$$
0 < \phi_S(p) \leq p^{n} \frac{1}{n}\log(n+1)+\frac{1}{n}\log(\max_{g\in S^{n}}||g||_{1})(1-p^{n})
$$
The claim follows.
\end{proof}

\begin{remark}\label{rk.example.not.attained}
Note that in the example of Proposition \ref{confined}, the two parabolic isometries are conjugate to each other via the non-trivial element of the Weyl group of the diagonal matrices. It is not hard to see that many similar examples can be constructed, say by pair of unipotents or hyperbolic elements conjugate to each other by the Weyl group element of a split torus. Similarly, one can generalize these examples to higher rank.
\end{remark}

\subsection{Large deviations of random products of matrices}\label{subsec.LDP}
In this paragraph, we briefly discuss another probabilistic aspect of random walks on linear groups where the notion of joint spectrum appears naturally, namely large deviations in connection with the second author's work \cite{Sert.LDP}.

Let $(Y_{n}=X_{n}. \ldots .X_{1})_{n \geq 1}$ be an i.i.d.\ random walk on a connected linear reductive group $G$. Consider the stochastic process given by $(\frac{1}{n}\kappa(Y_{n}))_{n \geq 1}$. It was shown in \cite{Sert.LDP} that if the probability measure $\mu$ governing the random walk on $G$ has a finite exponential moment and if its support generates a Zariski-dense semigroup, then the sequence $(\frac{1}{n}\kappa(Y_{n}))_{n \geq 1}$ satisfies a large deviation principle (LDP) with a proper convex rate function $I_{\mu}:\mathfrak{a}^{+} \to [0,+\infty]$ which admits a unique zero on the Lyapunov vector $\vec{\lambda}_{\mu}$.

If the probability measure $\mu$ on $G$ has bounded support, then the random variables $\frac{1}{n}\kappa(Y_{n})$ remain confined in a bounded region of $\mathfrak{a}^{+}$.  Hence the rate function $I_{\mu}$ is $+\infty$ outside of that bounded region. In this case it is therefore natural to determine the effective support $\{x \in \a^+ : I_{\mu}(x)<+\infty\}$. The following theorem, proven in \cite{Sert.LDP}, relates it to the joint spectrum:

\begin{theorem}[\cite{Sert.LDP}]
Let $G$ be a connected real linear reductive group and let $\mu$ be a probability measure on $G$, whose support generates a Zariski-dense semigroup in $G$. Then,\\[3pt]
1. The effective support $D_{I_{\mu}}=\{x \in \mathfrak{a}^+ \, | \, I_{\mu}(x)<\infty\}$ is a convex subset of $\mathfrak{a}^+$.\\[3pt]
2. If the support $S$ of $\mu$ is a bounded subset of $G$, then $\overline{D}_{I_{\mu}}=J(S)$ and $\ri(D_{I_{\mu}})=\ri(J(S))$, where $\ri(.) $ denotes the relative interior of a subset of $\mathfrak{a}$.\\[3pt]
3. If $S$ is a finite subset of $G$, then $D_{I_{\mu}}=J(S)$.
\end{theorem}


We note that it is an open problem to show that the Jordan vectors $(\frac{1}{n}\lambda(Y_{n}))_{n \geq 1}$ also satisfy an LDP with the same rate function. See \cite[Conjecture 6.2]{Sert.LDP}. Using Proposition \ref{prop.domination.and.schottky}, it is not difficult to check that this is indeed the case  when the support of the probability measure is $G$-dominated (for details, see \cite[Example 6.3.1]{Sert.LDP}).

\section{An example of non-polygonal joint spectrum}\label{section.non.polygonal.jointspectrum}

In this section we prove Proposition \ref{notpoly} from the introduction. We will exhibit a finite subset $T \subset G:=\SL(2,\mathbb{R}) \times \SL(2,\mathbb{R})$, which  generates a Zariski dense semigroup and whose joint spectrum in the Weyl chamber $\R_+ \times \R_+$ is not a polygon (see Fig. \ref{fig.non-diff}). More precisely we will show the following:

\begin{proposition}\label{notpoly-bis} There exists a pair $a,b \in \SL_2(\R)$ of non-commuting hyperbolic matrices with spectral radius $\lambda_1(b) \geq \lambda_1(a)\geq 1$, such that  the joint spectrum $J(T)$ of $T:=\{(1,a), (a,a), (b,b), (b,a)\}$ is given by $$J(T) := \{(x,y) \in \R_+ \times \R_+ : 0\le x \le \log \lambda_1(b), \log \lambda_1(a) \le y \le I(\frac{x}{\log \lambda_1(b)})\},$$ where $I:[0,1] \to [\log \lambda_1(a) ,\log \lambda_1(b)]$ is a surjective strictly concave increasing function, which is differentiable at $x \in (0,1)$ if and only if $x \notin \Q$.
\end{proposition}

The proof will make key use of the notion of Sturmian sequences as in the work of Bousch-Mairesse \cite{bousch-mairesse}, Hare-Morris-Sidorov-Theys \cite{HMST}, Morris-Sidorov \cite{morris-sidorov} and Jenkinson-Pollicott \cite{jenkinson-pollicott}. In fact a concrete  family of examples of pairs $a,b$ for which Proposition \ref{notpoly-bis} holds is as follows:
\begin{equation}\label{expair}a=\begin{pmatrix}
2 & 1\\
3 & 2\\
\end{pmatrix} \textnormal{      and } b=b^{(s)}:=\begin{pmatrix}
2e^{-s} & 3e^s \\
e^{-s} & 2e^s
\end{pmatrix},\end{equation} for any $s\geq 10$. This pair belongs to the class $\mathfrak{E}$ of \cite{jenkinson-pollicott} (see  \cite[\S 2.3]{OregonReyes}) made of full Sturmian pairs.

 Before we pass to the details, we first record some facts about hyperbolic elements in $\SL_2(\R)$ and give two easier examples in $\SL(2,\mathbb{R}) \times \SL(2,\mathbb{R})$ with polygonal joint spectrum.


\subsection{Some ingredients from hyperbolic geometry}\label{subsec.ingredients.hyp.geo}
We recall that $\PSL_2(\R)$ is the group of orientation preserving isometries of the hyperbolic plane. We will work in disc model and let $\mathbb{D}$ be the Poincar\'{e} disc and $\partial \mathbb{D}$ its boundary. A matrix $a \in \SL_2(\R)$ is called hyperbolic if its spectral radius $\lambda_1(a)$ is $>1$. The transformation $a$ has two fixed points on $\partial \mathbb{D}$: an attracting fixed point $x^+_a$ and a repelling one $x^-_a$. The hyperbolic geodesic $(x_a^-,x_a^+)=:\axe(a)$  is called the translation axis of $a$. It is the set of $x \in \mathbb{D}$ such that $d(ax,x)$ is minimal, where $d$ is the hyperbolic metric on $\mathbb{D}$. This minimal quantity is called the translation length of $a$ and we denote it by $\tau_a$. We have $\tau_a=d(ax,x)$ for every $x \in \axe(a)$ and $\tau_a=2\log \lambda_1(a)$.

We shall refer to two hyperbolic elements $a,b \in \SL(2,\mathbb{R})$ with disjoint translation axes as being \emph{in the same direction} if their fixed points on the boundary circle $\partial \mathbb{D}$ are located in this order $(x_a^-,x_a^+,x_b^+,x_b^-)$. 

\begin{lemma}\label{lemma.calculating.product.of.hyperbolics}
Let $a,b \in \SL(2,\mathbb{R})$ be two hyperbolic elements with disjoint translation axes and in the same direction. Let $d > 0 $ denote the distance between their translation axes. We have
$$
\cosh(\frac{\tau_{ab}}{2})=\cosh(d)\sinh(\frac{\tau_a}{2})\sinh(\frac{\tau_b}{2})+\cosh(\frac{\tau_a}{2})\cosh(\frac{\tau_b}{2}),
$$ in particular
\begin{equation}\label{tau+} \tau_a + \tau_b< \tau_{ab} \leq \tau_a + \tau_b + 2d.\end{equation}
\end{lemma}

\begin{proof}(sketch)
One routinely computes the trace of $ab$ in terms of that of $a$, $b$ and the distance $d$. For the inequality on the right hand side note that $$\cosh(\frac{\tau_{ab}}{2}) \leq \cosh(d)\cosh(\frac{\tau_a}{2}+\frac{\tau_b}{2}) \leq \cosh(d+\frac{\tau_a}{2}+\frac{\tau_b}{2}),$$ since $\cosh(x+y)=\sinh(x)\sinh(y)+\cosh(x)\cosh(y)$ and $\cosh(y)\cosh(x) \leq \cosh(x+y)$ for positive $x,y$.
\end{proof}


Given a finite word $w$ in the alphabet $\{0,1\}$, we denote by $w(a,b)$ the element of $\SL(2,\mathbb{R})$ obtained by replacing  $0$ by $a$ and $1$ by $b$, and taking the corresponding product in $\SL(2,\mathbb{R})$. 

\begin{remark}\label{remark.disjoint.axis} Note that if $a,b$ are as in Lemma \ref{lemma.calculating.product.of.hyperbolics} and $w$ a finite word, then $w(a,b)$, too, is hyperbolic and its translation axis lies in between those of $a$ and $b$, i.e. contained in the connected component of $\mathbb{D}\setminus (\axe(a)\cup \axe(b))$ whose boundary contains both $\axe(a)$ and $\axe(b)$.
\end{remark}

\begin{lemma}\label{word-comp}Let $a,b$ be as in Lemma \ref{lemma.calculating.product.of.hyperbolics} and $w,w'$ be two finite words of the same length in the alphabet $\{0,1\}$ that differ by at most $\ell$ different letters. Then
$$|\tau_{w(a,b)} - \tau_{w'(a,b)}| \leq \ell ( 2d + |\tau_b-\tau_a|\}).$$
\end{lemma}

\begin{proof} We may write $w=a_1x_1a_2x_2\ldots a_\ell x_\ell$ and $w=a_1y_1a_2y_2\ldots a_\ell y_\ell$ for some subwords $a_i$'s and letters $x_i,y_i \in \{0,1\}$. Now by $(\ref{tau+})$ and Remark \ref{remark.disjoint.axis} we see that $$|\tau_{w(a,b)} - \sum_i (\tau_{a_i} + \tau_{x_i})| \leq 2d \ell.$$ A similar inequality holds for $\tau_{w'(a,b)}$ and taking the difference yields the lemma.
\end{proof}

\begin{lemma}\label{corollary.substitute.and.increase}
Let $a,b$ be as in Lemma \ref{lemma.calculating.product.of.hyperbolics} and assume that $\tau_b \geq \tau_a + 2d +1$. Let $w$ be a finite word in $\{0,1\}$. If $\tilde{w}$ is a word obtained from $w$ by replacing one occurrence of $0$ by $1$, then $\tau_{\tilde{w}(a,b)}\geq\tau_{w(a,b)} +1$. 
\end{lemma}

\begin{proof}
If $w$ does not contain any occurrence of $0$, there is nothing to prove. Suppose it does. Then we may conjugate it cyclically (this does not change the translation length) and assume that the replaced letter $0$ appears at the end of $w$, i.e. $w=w'0$ and $\tilde{w}=w'1$. The elements $w'(a,b)$ and $a$ as well as $w'(a,b)$ and $b$ satisfy the hypotheses of Lemma \ref{lemma.calculating.product.of.hyperbolics} and the result follows immediately from $(\ref{tau+})$.
\end{proof}

\subsection{Some polygonal examples}\label{subsec.some.polygo}

We say that a bounded set $S \subset \Mat(d,\mathbb{R})$ has (Lagarias-Wang) finiteness length $n \in \mathbb{N}$ for the (upper) joint spectral radius if $n$ is the least natural number such that there exists $g \in S^n$ with $R(S)=\lambda_1(g)^{\frac{1}{n}}$, where we recall that $R(S)$ stands for the joint spectral radius of $S$ and $\lambda_1(g)$ is the spectral radius of the matrix $g$. If there is no such $n$, we say that $S$ is a finiteness counterexample. We define similarly the finiteness length of $S$ for the lower joint spectral radius $R_{sub}(S)$. 

Recall that it had been conjectured in \cite{lagarias-wang} that every finite set of matrices has finite finiteness length for the upper joint spectral radius. This was first refuted by Bousch-Mairesse \cite{bousch-mairesse}. We note the following:

\begin{lemma}\label{finiteness} Let  $a,b$ be are two hyperbolic elements in $\SL_2(\R)$ as in Lemma \ref{corollary.substitute.and.increase}. Then $\{a,b\}$ has finiteness of  length $1$ for both the  upper and lower joint spectral radii.
\end{lemma}

\begin{proof} Let $n \in \N$. Applying Lemma \ref{corollary.substitute.and.increase} repeatedly we see that $\tau_{a^n} \leq \tau_{w(a,b)} \leq \tau_{b^n}$ for every word $w$ of length $n$. The lemma follows. 
\end{proof}

\noindent \textit{Example.} ${}$ Let  $a,b$ be are two hyperbolic elements in $\SL_2(\R)$ as in Lemma \ref{corollary.substitute.and.increase}. Consider the subset $S$ of $G=\SL_2(\R) \times \SL_2(\R)$ given by $S=\{(a,a),(b,b),(a,b)\}$. Using Lemma  \ref{corollary.substitute.and.increase} it is straightforward to verify that the joint spectrum $J(S)$ of $S$ is the red triangle shown in Figure \ref{fig.tri} below.
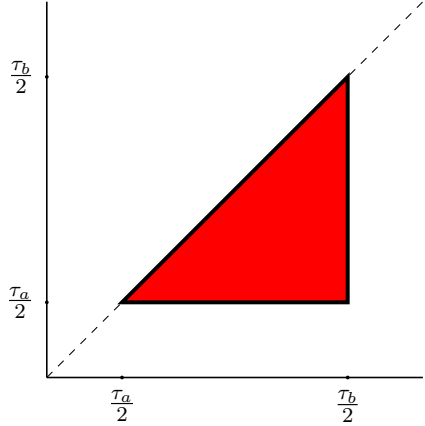
\begin{figure}[H]
\begin{minipage}{0.45\textwidth}
\begin{tikzpicture}
\coordinate (A) at (1, 1) {};
\coordinate (B) at (4, 1) {};
\coordinate (C) at (4,  4) {};
\node[circle,inner sep=0.5pt,fill=black,label=left:{$\frac{\tau_a}{2}$}] at (0,1) {};
\node[circle,inner sep=0.5pt,fill=black,label=below:{$\frac{\tau_a}{2}$}] at (1,0) {};
\node[circle,inner sep=0.5pt,fill=black,label=below:{$\frac{\tau_b}{2}$}] at (4,0) {};
\node[circle,inner sep=0.5pt,fill=black,label=left:{$\frac{\tau_b}{2}$}] at (0,4) {};
\filldraw[draw=black, fill=red, line width=1.5pt]
(A) -- (B) --  (C) --cycle;
\draw[semithick] (0,0) -- (5,0);
\draw[semithick] (0,0) -- (0,5);
\draw[dashed] (0,0)--(5,5);
\end{tikzpicture}
\caption{} \label{fig.tri}
\end{minipage}
\begin{minipage}{0.5\textwidth}
\vspace{-0.5cm}
Indeed replacing each letter $(a,a)$ by $(a,b)$ in a finite word in $S$ does not alter the spectral radius of the first coordinate, but increases that of the second coordinate by Lemma \ref{corollary.substitute.and.increase}.  Observe moreover that the set $S$ generates a Zariski dense semigroup and has \textit{finiteness of length $1$ for the joint spectrum} in the sense that $co(\lambda(S))=J(S)$ where $co(.)$ denotes the convex hull and $\lambda(.)$ is the Jordan projection. 
\end{minipage}
\end{figure}

\vspace{0.5cm}

\noindent \textit{Example.} ${}$
This time let $a$ and $b$ be two hyperbolic elements with disjoint axes and in the same direction satisfying $\tau_a=\tau_b$. It follows from $(\ref{tau+})$ and Remark \ref{remark.disjoint.axis} that the set $\{a,b\}$ has finiteness of length $1$ for the lower joint spectral radius. It is also not difficult to verify (see \cite{breuillard-sert.finiteness} for details) that it has finiteness of length $2$ for the (upper) joint spectral radius. Setting $S_1=\{(a,a),(a,b),(b,a)\}$, it follows that $J(S_1)$ is the square shown in Figure \ref{fig.square} below. 

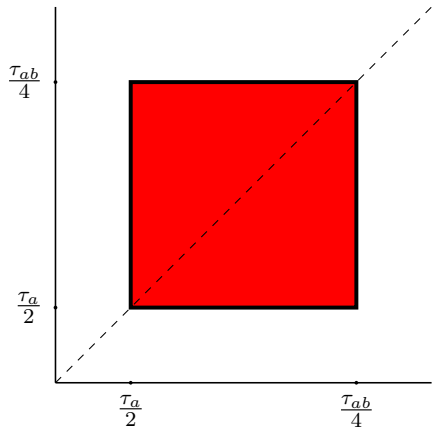
\begin{figure}[H]
\begin{minipage}{0.45\textwidth}
\begin{tikzpicture}
\coordinate (A) at (1, 1) {};
\coordinate (B) at (4, 1) {};
\coordinate (D) at (4, 4) {};
\coordinate (C) at (1, 4) {};
\node[circle,inner sep=0.5pt,fill=black,label=left:{$\frac{\tau_a}{2}$}] at (0,1) {};
\node[circle,inner sep=0.5pt,fill=black,label=below:{$\frac{\tau_a}{2}$}] at (1,0) {};
\node[circle,inner sep=0.5pt,fill=black,label=below:{$\frac{\tau_{ab}}{4}$}] at (4,0) {};
\node[circle,inner sep=0.5pt,fill=black,label=left:{$\frac{\tau_{ab}}{4}$}] at (0,4) {};
 
\filldraw[draw=black, fill=red, line width=1.5pt]
(A) -- (B) -- (D) -- (C) --cycle;
\draw[semithick] (0,0) -- (5,0);
\draw[semithick] (0,0) -- (0,5);
\draw[dashed] (0,0)--(5,5);
\end{tikzpicture}
\caption{} \label{fig.square}
\end{minipage}
\begin{minipage}{0.5\textwidth}
\vspace{-1cm}
Moreover, the set $S$ has finiteness of length $2$ for the joint spectrum. We also note that setting $S_2:=S_1 \cup \{(b,b)\}$, we have $J(S_1)=J(S_2)$. Setting $S_3:=S_2 \cup \{(e,e)\}$, we have $J(S_3)=co(J(S_2) \cup \{(0,0\})$, which is a rhombus with vertices $(0,0),(\frac{\tau_a}{2},\frac{\tau_{ab}}{4}),(\frac{\tau_{ab}}{4},\frac{\tau_a}{2}),(\frac{\tau_{ab}}{4},\frac{\tau_{ab}}{4})$. $\diamond$
\end{minipage} 
\end{figure}

\subsection{A non-polygonal joint spectrum in $\SL(2,\mathbb{R}) \times \SL(2,\mathbb{R})$} \label{subsection.polygonal}
In preparation for the proof of Proposition \ref{notpoly-bis}, we recall some useful results from \cite{bousch-mairesse,HMST, morris-sidorov,jenkinson-pollicott}.

\subsubsection{Sturmian pairs of matrices}

We first recall the notion of \emph{sturmian measure} on the space of one-sided infinite words in two letters $\Omega:=\{0,1\}^{\N}$ endowed with the left shift $\sigma: \Omega \to \Omega$ sending $w:=x_0x_1x_2\ldots$ to $\sigma(w)=x_1x_2\ldots$. If $w \in \{0,1\}^{\mathbb{N}}$ we denote by $w_n$ its prefix of length $n \in \mathbb{N}$, i.e. its first $n$ digits.

For every $\alpha \in [0,1]$, there exists a unique $\sigma$-invariant Borel probability measure $\mu_\alpha$ on $\Omega$ such that for $\mu_\alpha$-almost every $w \in \Omega$ and for every $n \in \N$
$$|-n\alpha + N_1(w_n)|<1,$$
where $N_1(w_n)$ denotes the number of digits $1$ occuring in $w_n$. The measure $\mu_\alpha$ is called \emph{the sturmian measure with parameter $\alpha$} on $\Omega$. It is $\sigma$-ergodic.  Its parameter $\alpha$ is rational if and only if $\mu_\alpha$-a.e. word in $\Omega$ is pre-periodic. We refer the reader to \cite{bousch-mairesse} for this background.

Now let $a,b$ be two matrices in $M_d(\mathbb{R})$ and set $S=\{a,b\}$. According to the Berger-Wang identity \cite{berger-wang} the joint spectral radius $R(S)$ satisfies:
 $$\log R(S)=\limsup_{n} \frac{1}{n}\log \max_{g \in S^n} \lambda_1(g)$$

It is not hard to show, see \cite[Thm. 3.1]{daubechies-lagarias}, that there always exists $w \in \{0,1\}^{\mathbb{N}}$ such that $$\log R(S)=\lim_n \frac{1}{n} \log \|w_n(a,b)\|.$$

We shall say that the (ordered) pair $(a,b)$ is \emph{sturmian} if the following stronger property holds: there is a unique $\alpha \in [0,1]$ such that  for $\mu_\alpha$-almost every infinite word $w$, we have $$\log R(\{a,b\})=\lim_n \frac{1}{n}\log \lambda_1(w_n(a,b)).$$

For a sturmian pair $(a,b)$, we shall refer to the parameter $\alpha \in [0,1]$ as sturmian parameter of the pair. Note then that $(b,a)$ is sturmian with paramater $1-\alpha$.

A pair $(a,b)$ is called \emph{full sturmian} if for every $t>0$, the pair $(a,tb)$ is sturmian. In \cite{jenkinson-pollicott}, Jenkinson-Pollicott exhibited an explicit open subset $\mathfrak{E}$ of $M_2(\mathbb{R})^2$ consisting of full sturmian pairs. Given a full sturmian pair $(a,b)$, we denote by $$p:\mathbb{R}^*_+ \to [0,1]$$ the map sending $t>0$ to  the sturmian parameter $p(t)$ of the pair $(a,tb)$. 

As shown in \cite[Thm. 9]{jenkinson-pollicott} for every pair in $\mathfrak{E}$, the parameter mapping $p$ is continuous, non-decreasing and surjective with the property that $p^{-1}(\alpha)$ is a singleton if $\alpha \in [0,1]$ is irrational and a closed interval of positive length of $\alpha$ is rational.

Note that the above results immediately imply the existence of uncountably many pairs which are finiteness counterexamples for the joint spectral radius: given a full sturmian pair $(a,b)$, for every $t$ such that the value $p(t)$ is  irrational $S_t:=\{a,tb\}$ is a finiteness counterexample. Likewise, the uniqueness of sturmian measures implies the existence of a finiteness counterexample in $\SL(2,\mathbb{R})$ as shown in  \cite[\S 2.3]{OregonReyes}) thanks to a remark of I. Morris that for every $s\geq 1$ the pair $(a,b^{(s)})$ given in $(\ref{expair})$ 
belongs to the aforementioned set $\mathfrak{E}$.

\subsubsection{Ratio-constrained joint spectral radius function}

Let $(u,v)$ be a pair of matrices in $M_d(\mathbb{R})$. We define their \emph{ratio-constrained joint spectral radius function} $I:[0,
1] \to \mathbb{R}_+$ by
$$
I(\alpha)=\inf_{\eps>0} \sup_{n \geq 1} \{ \frac{1}{n} \log \lambda_1(w_n(u,v)) \, | \, \frac{1}{n}N_1(w_n) \in [\alpha-\eps,\alpha+\eps], \, w \in \{0,1\}^{\mathbb{N}}\}
$$
where, as before, for an infinite word $w$ in the alphabet $\{0,1\}$, $w_n$ denotes the prefix of length $n$ and  $N_1(w_n)$ the number of occurrences of $1$ in $w_n$. Also as before $\lambda_1(g)$ is the spectral radius of $g \in M_d(\R)$.

We note that a very similar function has been studied by Morris-Sidorov \cite{morris-sidorov} for a particular pair of unipotent elements of $\SL(2,\mathbb{R})$ and our study in this paragraph is similar to theirs, except we deal with hyperbolic matrices.

\begin{lemma}\label{lemma.precise.at.rationals} Let $(a,b)$ be a pair of hyperbolic matrices in $\SL_2(\R)$ with disjoint axes and in the same direction as in Lemma  \ref{lemma.calculating.product.of.hyperbolics}. The associated function $I(\alpha)$ is continuous and for every $\alpha \in [0,1] \cap \mathbb{Q}$, we have
$$
I(\alpha)=\sup_{\underset{n\alpha \in \mathbb{N}}{n \geq 1}} \frac{1}{n} \{\log \lambda_1(w_n(a,b)) \, |\, \frac{1}{n}N_1(w_n)=\alpha\}
$$
\end{lemma}

\begin{proof}This is a straightforward consequence of Lemma \ref{word-comp} above and the fact that $\tau_x=2\log \lambda_1(x)$ for a hyperbolic matrix $x \in \SL_2(\R)$. We leave the details to the reader. 
\end{proof}

We now assume from now on that the pair $(a,b)$ is as in Lemma \ref{corollary.substitute.and.increase} and is also full sturmian in the sense of Jenkinson-Pollicott as recalled above, i.e. $(a,tb)$ is sturmian for every $t>0$.  We further assume that the parameter mapping $p:\R_+^* \to [0,1]$ is continuous, non-decreasing and surjective with the property that $p^{-1}(\alpha)$ is a singleton if $\alpha \in [0,1]$ is irrational and a closed interval of positive length of $\alpha$ is rational. The pairs $(a,b^{(s)})$ of $(\ref{expair})$ will satisfy these assumptions. We denote by $S$ the set $\{a,b\}$ and for $t>1$, we put $S_t:=\{a,tb\} \subset \GL_2(\R)$.

\begin{proposition}\label{lemma.properties.ratio.constraint.function}
The function $I:[0,1] \to \mathbb{R}_+$ associated to the pair $(a,b)$ is strictly increasing, strictly concave, satisfies $I(0)=\log R_{sub}(S) = \frac{\tau_a}{2}$, $I(1)=\log R(S)=\frac{\tau_b}{2}$ and is differentiable at $x \in [0,1]$ if and only if $x$ is irrational.
\end{proposition}

\begin{proof} Let $\beta>\alpha$. If $w$ is a word of length $n$ in $\{0,1\}$ with $|N_1(w)-\alpha n|<\eps n$, then changing $\lfloor(\beta-\alpha)n\rfloor$ digits $0$ into $1$'s, we obtain a word $w'$ with $|N_1(w')-\beta n|<\eps n+1$. Moreover Lemma \ref{corollary.substitute.and.increase} implies that $\tau_{w'(a,b)} \geq \tau_{w(a,b)} + \lfloor(\beta-\alpha)n\rfloor$. It follows that $I(\beta)\geq I(\alpha) + \beta-\alpha$. In particular the function $I$ is strictly increasing. The assertions about $I(0)$ and $I(1)$ follow from Lemma  \ref{finiteness}.


Let us show that $I$ is strictly concave: let $\alpha \in [0,1] \cap  \mathbb{Q}$. For any word $w$ of length $n$ in the alphabet $\{0,1\}$ satisfying $n\alpha=N_1(w)$ and for any $t>0$, we clearly have
\begin{equation}\label{eqq1}
\lambda_1(w(a,b)) = t^{-\alpha n}\lambda_1(w(a,tb)),
\end{equation}
so that for each such word and $t>0$,
\begin{equation*}
\log \lambda_1(w(a,b)) \leq  n(\log R(S_t) -\alpha  \log t).
\end{equation*}
By Lemma \ref{lemma.precise.at.rationals}, the last inequality implies that for every $t>0$,
\begin{equation}\label{eqq2}
I(\alpha) \leq \log R(S_t) -\alpha \log t.
\end{equation}
Now using the fact that the pair $(a,b)$ is full sturmian and that its associated parameter mapping $p:\mathbb{R}_+ \to [0,1]$ is surjective, choose $t_\alpha \in p^{-1}(\alpha)$. Plugging this value in (\ref{eqq1}) together with a maximizing finite word $w$, we get that 
\begin{equation} \label{eqq3}
n I(\alpha) \geq \lambda(w(a,b)) = n(\log R(S_{t_\alpha})-\alpha  \log t_{\alpha})
\end{equation}
Putting (\ref{eqq2}) and (\ref{eqq3}) together, we obtain that for every $\alpha \in [0,1] \cap \mathbb{Q}$,
\begin{equation} \label{eqq4}
I(\alpha)= \min_{t>0} \{\log R(S_t) - \alpha \log t \}= \log R(S_{t_\alpha})-\alpha  \log t_{\alpha}
\end{equation}
In view of the continuity of $I$, this establishes that $I$ is concave. Furthermore, the continuity of the parameter function $p$ and the joint spectral radius implies that the previous equation holds for every $\alpha \in [0,1]$. it is strictly concave, because $t_\alpha<t_\beta$ if $\alpha<\beta$.

Now recall that, for $\alpha \in [0,1]$, by assumption $p^{-1}(\alpha)=[t^{(1)}_\alpha,t^{(2)}_\alpha]$, where $t^{(1)}_\alpha<t^{(2)}_\alpha$ if $\alpha$ is rational and $t^{(1)}_\alpha=t^{(2)}_\alpha$ if not. Hence, by $(\ref{eqq4})$, $I'(\alpha^{-})=-\log t^{(1)}_\alpha$ and $I'(\alpha^{+})=-\log t^{(2)}_\alpha$. This establishes the claim about the derivative of $I$.
\end{proof}

\subsubsection{Proof of Proposition \ref{notpoly-bis}}

We start by observing that the pairs $(a,b^{(s)})$ defined in $(\ref{expair})$ satisfy the assumptions of Lemma  \ref{corollary.substitute.and.increase}.

\begin{lemma}
For every $s >0$, the pair $(a,b^{(s)})$ defined in $(\ref{expair})$ consists of two hyperbolic elements of $\SL(2,\mathbb{R})$ that have disjoint translation axes and are in the same direction. Moreover the distance between the translation axes of $a$ and $b^{(s)}$ is bounded by some $d>0$ as $s$ varies and $\tau_{b^{(s)}} - \tau_a \geq 2d+1$ for all $s \geq 10$.
\end{lemma}

\begin{proof} Straightforward calculation. 
\end{proof}

Furthermore the pair $(a,b^{(s)})$ belongs to the class $\mathfrak{E}$ of Jenkinson-Pollicott therefore, according to \cite[Thm. 9]{jenkinson-pollicott}, it satisfies the assumptions of Proposition \ref{lemma.properties.ratio.constraint.function}. 

For ease of notation, we set $b=b^{(s)}$ and we let $T:=\{(e,a),(b,b),(a,a),(b,a)\}\subset \SL(2,\mathbb{R}) \times \SL(2,\mathbb{R})$. We now identify $J(T)\subset \R_+\times \R_+$ and show that it is not polygonal.

\begin{figure}[H]
\begin{flushleft}
\begin{minipage}{0.4\textwidth}
\begin{flushleft}
\begin{tikzpicture}

\tikzstyle{every node}=[draw]
\coordinate (A) at (0, 1) {};
\coordinate (B) at (4, 1) {};
\coordinate (C) at (4, 4) {};
\coordinate (X1) at (0.2, 2.5) {};
\coordinate (X2) at (1.1, 3.5) {};
\coordinate (X3) at (2.5, 3.95) {};

\draw[fill] (X1) circle [radius=0.02];
\draw[fill] (X2) circle [radius=0.02];
\draw[fill] (X3) circle [radius=0.02];

\filldraw[draw=black, fill=red, dashed]
(A.center) to[out=90, in=265] (X1.center) to[out=65, in=220] (X2.center) to[out=20, in=190] (X3.center) to[out=2, in=180] (C.center) -- (B) -- cycle;

\node[circle,inner sep=1pt,fill=black,label=left:{$\frac{\tau_a}{2}$}] at (0,1) {};
\node[circle,inner sep=1pt,fill=black,label=below:{$\frac{\tau_b}{2}$}] at (4,0) {};
\node[circle,inner sep=1pt,fill=black,label=left:{$\frac{\tau_b}{2}$}] at (0,4) {};
\draw[draw=black, line width=1.5pt]
(A) -- (B) --  (C);
\draw[semithick, ->] (0,0) -- (5,0);
\draw[semithick, ->] (0,0) -- (0,5);
\draw[dashed] (0,0)--(5,5);
\end{tikzpicture}
\end{flushleft}
\caption{} \label{fig.non-diff}
\end{minipage}
\begin{minipage}{0.08\textwidth}
${}$
\end{minipage}
\begin{minipage}{0.5\textwidth}
\vspace{-1.4cm}
We have already seen (in the first example in Subsection \ref{subsection.polygonal}) that  the subset $T_1:=T \setminus\{(e,a)\}$ has a triangular joint spectrum as in Figure \ref{fig.tri}. Moreover $J(T)$ contains a larger triangle $co(\lambda(T))$ which has vertices $(0,\frac{\tau_a}{2}),(\frac{\tau_b}{2},\frac{\tau_a}{2}),(\frac{\tau_b}{2},\frac{\tau_b}{2})$. We will show that $J(T)$ is yet strictly larger, as in Figure \ref{fig.non-diff} and its boundary curve between the vertices $(0,\frac{\tau_a}{2})$ and $(\frac{\tau_b}{2},\frac{\tau_b}{2})$ will be derived from the graph of the function $I$.
\end{minipage}
\end{flushleft}
\end{figure}



Note first that $J(T)$ is contained in the rectangle $\{(x,y), 0\leq x \leq \frac{\tau_a}{2} \leq y \leq  \frac{\tau_b}{2} \}$.  Since $J(T)$ is a convex body in the plane, for every $\alpha \in [0,1]$, the vertical line $x=\alpha \frac{\tau_b}{2}$ intersects $J(T)$ in a segment $[\frac{\tau_a}{2},f(\alpha)]$, for some concave function $f$. Furthermore, it follows from Lemma \ref{corollary.substitute.and.increase} that $f(0)=\frac{\tau_a}{2}$ and $f(1)=\frac{\tau_b}{2}$. We will now show that $f(\alpha)=I(\alpha)$ for all $\alpha \in [0,1]$, where $I$ is the function discussed in the previous subsection. 

We already know that $f$ is concave and, since $f(1)=\frac{\tau_b}{2}$ we also know that it is non-decreasing on $[0,1]$. Let now $\alpha \in (0,1)$ and $\eps>0$ be given. To ease the notation, let us put $x_0:=(e,a), x_1:=(b,b), x_2:=(a,a)$ and $x_3:=(b,a)$. By definition of the joint spectrum, there exist $n \in \mathbb{N}$, a word $w$ of length $n$ in the alphabet $\{0,1,2,3\}$ such that 
$$
0 \leq f(\alpha)-\frac{1}{n} \log \lambda_1(p_2(w(x_0,x_1,x_2,x_3))) <\eps \quad \text{and} \quad$$ $$|\frac{1}{n}\log \lambda_1(p_1(w(x_0,x_1,x_2,x_3)))-\alpha \frac{\tau_b}{2}|<\eps
$$
where $p_i:G \to \SL(2,\mathbb{R})$ denotes the projection onto $i^{th}$ factor for $i \in \{1,2\}$. Now by Lemma \ref{corollary.substitute.and.increase} if we replace any letter $x_3$ of $w(x_0,x_1,x_2,x_3)$ by $x_1$ we will increase $\lambda_1(p_2(w(x_0,x_1,x_2,x_3)))$ while not affecting $\lambda_1(p_1(w(x_0,x_1,x_2,x_3)))$. Hence
\begin{equation*}
\begin{aligned}
&\lambda_1(p_1(w(x_0,x_1,x_2,x_1)))=\lambda_1(p_1(w(x_0,x_1,x_2,x_3))) \qquad  \text{and} \\ & e^{nf(\alpha)} \geq \lambda_1(p_2(w(x_0,x_1,x_2,x_1))) \geq \lambda_1(p_2(w(x_0,x_1,x_2,x_3))).
\end{aligned}
\end{equation*}
This implies that we can restrict attention to the elements $x_0,x_1,x_2$ while studying the piece of the boundary of $J(T)$ described by the function $f$. 

Now, using Lemma \ref{lemma.calculating.product.of.hyperbolics}, for a finite word $w$ in the alphabet $\{0,1,2\}$, we clearly have 
\begin{equation*}
\begin{aligned}
& \lambda_1(p_1(w(x_0,x_1,x_2))) \geq \lambda_1(p_1(w(x_0,x_1,x_0))) \quad \text{and}\\  & \lambda_1(p_2(w(x_0,x_1,x_2)))=\lambda_1(p_2(w(x_0,x_1,x_0))).
\end{aligned}
\end{equation*}
These, together with the fact that $f$ is non-decreasing, imply that we can only consider the Jordan projections of the products of elements $x_0$ and $x_1$ in order to determine the boundary of $J(T)$ described by the function $f$. But then, since $x_0=(e,a)$ and $x_1=(b,b)$, by definition of the function $f$, we have
$$
f(\alpha)=\inf_{\eps>0}\sup_{n \geq 1} \{\frac{1}{n}\lambda(w_n(x_0,x_1))\, |\, w \in \{0,1\}^{\mathbb{N}}, \frac{1}{n}N_1(w_n)\in [\alpha-\eps,\alpha+\eps]\}=I(\alpha)
$$
establishing the desired equality. Together with Proposition \ref{lemma.properties.ratio.constraint.function} this ends the proof of Proposition \ref{notpoly-bis}.

\begin{remark}\label{remark.GL2.spectrum} This remark is due to Jairo Bochi, who also pointed out to us the reference \cite{morris-sidorov}. 
Using the analogous function $I$ associated to a certain pair of unipotent matrices in $\SL(2,\mathbb{R})$ studied in \cite{morris-sidorov}, one can also construct a non-polygonal joint spectrum in $\GL(2,\mathbb{R})$ having similar differentiability properties to the one constructed here. 
More precisely, it follows easily from \cite{morris-sidorov} that for any real $\alpha \in \mathbb{R}\setminus\{0,1\}$, the joint spectrum of $$\{ \begin{pmatrix} 1 & 1 \\ 
0 & 1
\end{pmatrix},\alpha \begin{pmatrix} 1 & 0\\
1 & 1
\end{pmatrix}\}$$ is a non-polygonal convex subset of non-empty interior in the Weyl chamber $\mathbb{R}\times \mathbb{R}_+$ of $\GL(2,\mathbb{R})$. Note also that for $\alpha=1$, the joint spectrum of the corresponding pair is a non-trivial line segment contained in  the Weyl chamber $\{0\} \times \mathbb{R}_+$  of the subgroup $\SL(2,\mathbb{R})$.
\end{remark}

\begin{remark}\label{remark.elementary.example}
In work in preparation \cite{breuillard-sert.finiteness} we give another proof of Proposition \ref{notpoly-bis} and develop a more direct geometric approach which yields a larger family of examples, including all pairs of hyperbolic elements in $\SL(2,\mathbb{R})$ with disjoint axes in the same direction and sufficiently large translation length, as well as all conjugate pairs of unipotent matrices satisfying a ping-pong condition.  
\end{remark}

\section{Further directions}\label{sec.further}
In this final section we gather a number of problems that are still open in connection with the joint spectrum. 

\bigskip

1) How quickly does $J(S)$ fill in if we look at $\kappa(S^n)/n$ or if we look at $\lambda(S^n)/n$? Namely given $\eps>0$, give an estimate on the first $n$ such that every point of $J(S)$ is at distance at most $\eps$ from $\kappa(S^n)/n$ or $\lambda(S^n)/n$.  This question has been studied for the joint spectral radius, see Morris \cite{morris-rapidly-converging} and Bochi-Garibaldi \cite{bochi-garibaldi}.  

\bigskip

2) What points on the boundary of $J(S)$ belong to the Lyapunov spectrum? Which ergodic measures realize them? Is it always the case that an ergodic measure realizing a boundary point of $J(S)$ as its Lyapunov has zero entropy, in the case of domination (see Bochi-Rams \cite[\S 1.7]{bochi-rams})?

\bigskip

3) Extend the results of this paper to non-invertible matrices. 

\bigskip

4) Describe which convex sets can arise as joint spectrum of a small perturbation of a given compact set $S$ without full domination. A partial answer is already given in \cite{bochi-morris} for the lower joint spectral radius. Describe the points of continuity of the map $S \mapsto J(S)$ for the Hausdorff metric.

\bigskip

5) The joint spectral radius is locally Lipchitz continuous, see \cite{wirth,kozyakin, bochi-garibaldi}. Does the same hold for the joint spectrum?

\bigskip

6) Bochi \cite{bochi-radius} shows that there are $k_0,c>0$ depending only on $d$ such that $$\max_{n\leq k_0, g \in S^n} \lambda_1(g)^{\frac{1}{n}} \geq c R(S)$$ for every $S \subset M_d(\C)$, where $R(S)$ is the joint spectral radius. Is there an analogous phenomenon for the joint spectrum?

\bigskip 

7) Given $x \in \R$, Feng studies in \cite{feng1,feng2} the family of sequences $b$ in $S^\N$ with $\lambda_1(b)=x$ and describes its Hausdorff dimension in terms of entropies of measures with the same Lyapunov exponent (see also the recent work \cite{diaz-gelfert-rams}). Can one obtain a multi-dimensional analogue, where one starts with $x \in J(S)$ and looks at the sequences $b$ such that $(\ref{convergence-sequence})$ holds?

\bigskip 

8) Sturmian measures are used to describe the boundary of the joint spectrum in the examples of Section \ref{section.non.polygonal.jointspectrum}. In the general case, say of a finite $S$ in a reductive group $G$, can one describe precisely the class of measures realizing boundary points in $J(S)$? 

\bigskip

9) Under what assumptions do the conclusions of Proposition \ref{confined} and Remark \ref{rk.example.not.attained} hold? Can one describe the i.i.d.\ Lyapunov spectrum more explicitly (e.g. how far is it away from the boundary of joint spectrum)? 

\bigskip

10) Extend the results of this paper to more general cocycles and base dynamics.

\end{document}